\documentclass[12pt]{amsart}

\usepackage{amsmath,amsthm,amssymb}

\newenvironment{Proof of}[1]{\emph{Proof of #1.}}{\hfill $\qquad \square$\par}

\DeclareMathOperator{\M}{Mes}
\DeclareMathOperator{\Closed}{Closed}

\DeclareMathOperator{\Aut}{Aut}

\DeclareMathOperator{\clsp}{\overline{span}}

\newcommand{\HH}{\mathcal H}

\newcommand{\K}{\mathcal K}
\newcommand{\RR}{\mathcal R}

\newcommand{\VV}{\mathcal V}

\newcommand{\morp}{\varrho}

\newcommand{\LL}{\mathcal L}

\newcommand{\al}{\alpha}

\newcommand{\FF}{\mathcal F}

\newcommand{\OO}{\mathcal O}

\newcommand{\B}{ B}

\newcommand{\D}{\mathcal D}
\newcommand{\G}{\mathcal G}
\newcommand{\C}{\mathbb C}

\newcommand{\Z}{\mathbb Z}
\newcommand{\N}{\mathbb N}
\newcommand{\T}{\mathbb T}
\newcommand{\TT}{\mathcal T}
\newcommand{\QQ}{\mathcal Q}
\newcommand{\supp}{\textrm{supp}\,}

\newtheorem{thm}{Theorem}[section]\newtheorem{lem}[thm]{Lemma} 
\newtheorem{prop}[thm]{Proposition} 
\newtheorem{cor}[thm]{Corollary}

\theoremstyle{remark}
\newtheorem{rem}[thm]{Remark}

\theoremstyle{definition}
\newtheorem{dfn}[thm]{Definition}
\newtheorem{ex}[thm]{Example}

\title[Exel's crossed product and completely positive maps]{Exel's crossed product and \\ crossed products
by completely positive maps}

\author{Bartosz Kosma  Kwa\'sniewski}
\address{Department of Mathematics and Computer Science, The University of Southern Denmark,
Campusvej 55, DK--5230 Odense M, Denmark // Institute of Mathematics, University  of Bialystok, ul. Akademicka 2, PL-15-267  Bialystok, Poland}
\address{Institute of Mathematics, Polish Academy of Science,  ul. \'Sniadeckich 8, PL-00-956 Warszawa, Poland (from October 1, 2013, to September 30, 2014)}
\email{bartoszk@math.uwb.edu.pl}
\urladdr{http://math.uwb.edu.pl/~zaf/kwasniewski}

\keywords{Completely positive map, transfer operator, crossed product, Exel's crossed product, graph $C^*$-algebra,  Cuntz-Pimsner algebra}

\subjclass[2010]{46L05,46L55}

\thanks{This research was supported by  NCN  grant number  DEC-2011/01/D/ST1/04112  and by a Marie Curie Intra European Fellowship within the 7th European Community Framework Programme; project `OperaDynaDual' (2014-2016). The author  thanks Adam Skalski and the anonymous reviewers
for their  comments and suggestions that improved the
quality of the paper}

\begin{document}

\begin{abstract}
We introduce  crossed products of a $C^*$-algebra $A$ by a completely positive map $\morp:A\to A$ relative to an ideal in $A$. When $\morp$ is multiplicative  they generalize various  crossed products by endomorphisms. When $A$ is commutative they include $C^*$-algebras associated to Markov operators by  Ionescu, Muhly, Vega, and to topological relations by Brenken, but in general they are not modeled by topological quivers popularized by   Muhly and Tomforde.

 We show that   Exel's crossed product $A\rtimes_{\alpha,\LL} \N$, generalized to the case where $A$ is not necessarily unital,  is  the crossed product of   $A$ by the transfer operator $\LL$ relative to the ideal generated by $\alpha(A)$. We give natural conditions under which $\alpha(A)$ is uniquely determined by $\LL$, and hence $A\rtimes_{\alpha,\LL} \N$ depends only on $\LL$. Moreover, the  $C^*$-algebra $\OO(A,\alpha,\LL)$ associated to $(A,\alpha,\LL)$ by Exel and Royer always coincides with our unrelative crossed product by $\LL$.

As another  non-trivial application of our construction we extend a result of Brownlowe, Raeburn and Vittadello, by  showing that the $C^*$-algebra of an arbitrary  infinite  graph $E$ can be realized as a crossed product of the diagonal algebra $\D_E$ by a `Perron-Frobenious' operator  $\LL$. The important difference to the previous result is that in general there is no  endomorphism $\alpha$ of $\D_E$ making $(\D_E,\alpha,\LL)$  an Exel system.
\end{abstract}

\maketitle


\section{Introduction}

In the present state of the art the theory of crossed products  of  $C^*$-algebras by endomorphisms breaks down  into two areas that involve two different constructions. The first approach originated in late 1970's in the  work of Cuntz \cite{cuntz} and was developed by many authors \cite{Paschke2}, \cite{Stacey},
  \cite{alnr}, \cite{Murphy},  \cite{Ant-Bakht-Leb},  \cite{KL}, \cite{kwa-rever}.  Another approach was initiated by Exel \cite{exel2} in the beginning of the present century and  immediately received a lot of attention; in particular, Exel's construction was extended in \cite{brv}, \cite{Larsen}, \cite{er}, \cite{Brownlowe}. By now,  both of the approaches have proved to be  useful in an innumerable variety of problems and their importance is well-acknowledged.
They (or their semigroup versions) serve as tools to construct and analyse the most intensively studied  $C^*$-algebras in recent years. These include: Cuntz algebras \cite{cuntz}, Cuntz-Krieger algebras  \cite{exel2}, Exel-Laca algebras \cite{er}, graph algebras \cite{brv}, \cite{hr}, \cite{kwa-interact}, higher-rank graph algebras \cite{Brownlowe},  $C^*$-algebras arising from semigroups \cite{alnr}, number fields \cite{bc-alg}, \cite{hecke5}, or algebraic dynamical systems \cite{BLS}. Among the applications one could mention their significant role in classification  of $C^*$-algebras \cite{Rordam},  study of phase transitions   \cite{diri}, or short exact sequences and tensor products \cite{larsen1}.

  In view of what has been said, it is somewhat surprising that the intersection of these two  approaches is relatively small: the two constructions coincide for injective corner endomorphisms \cite{exel2} and more generally for systems called complete in \cite{Ant-Bakht-Leb}, \cite{kwa-trans}, and  reversible in \cite{kwa-rever}. Nowadays, it is known, see, for instance, \cite{br}, \cite{KL}, that  the aforementioned crossed-products can be unified in the framework of relative Cuntz-Pimsner algebras $\OO(J,X)$ of Muhly and Solel \cite{ms2}. However, different constructions are associated with different $C^*$-correspondences and different ideals $J$.

	In fact, the relationship between the  two aforesaid lines of research  is  still shrouded in mystery and calls for clarification. One of the overall aims of the present paper is to cover this demand.  We do it   by  showing that  the two areas  are  different special cases of one natural construction of a crossed product by a completely positive map.
In particular, since completely positive maps are ubiquitous in the  $C^*$-theory and in quantum physics,  the crossed products we introduce have an ample potential for further study  and applications. We hope that the present article will not only clear the decks but also give an impulse for such a  development (see, for instance, our remarks  concerning crossed products of commutative algebras (subsection \ref{commutative example}); also the study of ergodic properties of non-commutative Perron-Frobenius operators that we introduce is of interest (see subsection \ref{Non-commutative Perron-Frobenius subsection})).

We note that Schweizer defined in  \cite[Subsection 3.3]{Schweizer} a crossed product by a completely positive map  as a particular case of Pimsner's (augmented) $C^*$-algebra
 \cite{p}. However, apart from giving a simplicity criterion \cite[Theorem 4.6]{Szwajcar} he  didn't study the structure of these algebras.  Schweizer's crossed product is covered by our construction  (cf. Remark \ref{Schweizer's crossed product}).

Let us explain our strategy in more detail. We introduce (in Definition \ref{definition main}) the relative crossed product $C^*(A,\morp;J)$ of a $C^*$-algebra $A$ by a completely positive mapping $\morp:A\to A$  relative to an ideal $J$ in $A$.  The unrelative crossed product is $C^*(A,\morp):=C^*(A,\morp;N_\morp^\bot)$ where $N_\morp$ is the largest ideal contained in  $\ker\morp$.
When $\alpha:=\morp$ is multiplicative, hence an endomorphism of  $A$, the crossed products $C^*(A,\alpha;J)$ cover the line of research we attributed to Cuntz.  More specifically (see  Subsection  \ref{The case when multiplicative} below), the $C^*$-algebras  $C^*(A,\alpha;J)$ coincide with crossed products by endomorphisms studied in \cite{KL}, for unital $A$, and in  \cite{kwa-rever},  for extendible $\alpha$. In particular, if $\alpha$ is  extendible then $C^*(A,\alpha;A)$ is  Stacey's crossed product \cite{Stacey}, and $C^*(A,\alpha;\{0\})$ is the partial isometric crossed product introduced,  in a semigroup context, by Lindiarni and Raeburn \cite{Lin-Rae} (see Proposition \ref{Proposition for endomorphisms}). Accordingly, $C^*(A,\alpha)=C^*(A,\alpha;\ker \alpha^\bot)$ is a good candidate for the (unrelative) crossed product  by an arbitrary endomorphism, cf. \cite{KL}, \cite{kwa-rever}.  In contrast to this multiplicative case, we claim that  Exel's crossed product $A\rtimes_{\alpha,\LL} \N$  is a crossed product by the transfer operator $\LL$
 (which as a rule is not multiplicative).  In order to make this statement precise  we  need to thoroughly re-examine - take `a new look at' Exel's construction.

We recall   that Exel introduced  in \cite{exel2} the crossed product $A\rtimes_{\alpha,\LL} \N$ of a unital $C^*$-algebra $A$ by an endomorphism $\alpha:A\to A$
 which  also depends  on the choice of a \emph{transfer operator}, i.e.  a positive linear map $\LL:A\to A$ such that  $\LL(\al(a)b)=a\LL(b)$, for all $a,b \in A$.
This construction was generalized to the non-unital case  in   \cite{brv}, \cite{Larsen}  were authors assumed that both $\alpha$ and $\LL$ extend to strictly continuous maps on the multiplier algebra $M(A)$. We show however that extendability of $\LL$ is automatic and since extendability of $\alpha$ does not play any  role in the definition, in the present paper,  we consider  crossed products $A\rtimes_{\alpha,\LL} \N$ for  Exel systems $(A,\alpha, \LL)$ where  $A$, $\alpha$ and $\LL$ are arbitrary.
Obviously, in a typical situation there are infinitely many different transfer operators for a fixed $\alpha$. On the other hand, under natural assumptions, such as faithfulness of  $\LL$,  which usually appear in applications \cite{exel2}, \cite{exel_vershik}, \cite{br}, \cite{brv}, the endomorphism $\alpha$ is uniquely determined by a fixed transfer operator $\LL$.
Moreover, any transfer operator $\LL$ is necessarily  a completely positive map and therefore it is  suitable to form a crossed product on its own. This provokes  the question:
\begin{center}
To what extent  $A\rtimes_{\alpha,\LL} \N$ depends on $\alpha$?
\end{center}

Before giving an answer, we  need to stress that the pioneering Exel's definition of $A\rtimes_{\alpha,\LL} \N$, \cite[Definition 3.7]{exel2}, was to some degree experimental. In general it requires a modification. Namely, Brownlowe and Raeburn in \cite{br} recognized $A\rtimes_{\alpha,\LL} \N$ as a relative Cuntz-Pimsner algebra $\OO(K_\alpha,M_\LL)$ where $M_\LL$ is a $C^*$-correspondence  associated to $(A,\alpha, \LL)$   and  $K_\alpha=\overline{A\alpha(A)A} \cap J(M_\LL)$ is the intersection of the ideal generated by $\alpha(A)$ and the ideal of elements that the left action $\phi$ of $A$ on $M_\LL$ sends to `compacts'. Then it follows from general results on relative Cuntz-Pimsner algebras, see \cite[Proposition 2.21]{ms}, \cite[Proposition 3.3]{katsura} or \cite[Lemma 2.2]{br},  that  $A$ embeds into $A\rtimes_{\alpha,\LL} \N$ if and only if $K_\alpha$ is contained in  the ideal $(\ker\phi)^\bot \cap J(M_\LL)$. But the latter condition is not  always satisfied. In particular,  the theory of Cuntz-Pimsner algebras, and most notably  the work of Katsura \cite{katsura1}, \cite{katsura}, indicates that   Exel's construction should be  improved by replacing in \cite[Definition 3.7]{exel2} the ideal $\overline{A\alpha(A)A}$ with $(\ker\phi)^\bot \cap J(M_\LL)$. This is  done by Exel and Royer in \cite{er}, cf. also \cite[Proposition 4.5]{Brownlowe},   where they  associate to $(A,\alpha,\LL)$ a  $C^*$-algebra $\OO(A,\alpha,\LL)$ which is isomorphic to Katsura's Cuntz-Pimsner algebra $\OO_{M_\LL}$ (as a matter of fact, authors of \cite{er} deal with more general Exel systems where $\alpha$ and $\LL$ are
only `partially defined').

Turning back to our question, the results of the present paper give the following  answer, which consists of three parts:

\begin{itemize}
\item[(1)] the modified Exel's crossed product $\OO(A,\alpha,\LL)$ always coincides with our unrelative crossed product $C^*(A,\LL)$ of $A$ by $\LL$ (Theorem \ref{Exel crossed products as relative Pimsner}),
\item[(2)] original Exel's crossed product $A\rtimes_{\alpha,\LL} \N$, for regular systems, does not depend on $\alpha$, if we assume certain conditions assuring  that $A$   embeds into $A\rtimes_{\alpha,\LL} \N$ (Proposition \ref{Exel without endomorphism}, Theorem \ref{complete systems crossed products}),
\item[(3)] the three algebras  $A\rtimes_{\alpha,\LL} \N$,  $\OO(A,\alpha,\LL)$, $C^*(A,\LL)$ coincide for most of systems appearing in applications (Proposition \ref{corollary on faithful transfers}, Theorem \ref{complete systems crossed products}, Theorem \ref{thm for Raeburn, Brownlowe and Vitadello} i)).
\end{itemize}
In connection with point (3)  it is interesting to note that, in general, there seems to be no clear relation between the ideals
$$
\overline{A\alpha(A)A}, \qquad  (\ker\phi)^\bot, \qquad J(M_\LL).
$$
However, for many natural systems $(A,\alpha,\LL)$,  for instance for all such systems arising from graphs  (cf. Lemma  \ref{existence of endomorphism} below), we always  have  $\overline{A\alpha(A)A}=(\ker\phi)^\bot\cap  J(M_\LL)$ and consequently $A\rtimes_{\alpha,\LL} \N=\OO(A,\alpha,\LL)=C^*(A,\LL)$. This shows that (by incorporating the ideal $\overline{A\alpha(A)A}$ into his original construction) Exel exhibited  an incredibly good intuition; especially that, in contrast to  $\overline{A\alpha(A)A}$, determining  $(\ker\phi)^\bot \cap J(M_\LL)$  is  very hard in practice.
In particular, it is an important task  to identify  Exel systems $(A,\alpha,\LL)$ for which   $A\rtimes_{\alpha,\LL} \N=  \OO(A,\alpha,\LL)=C^*(A,\LL)$. We find a large class of such objects in the present article.

We test the results of our findings on graph $C^*$-algebras. We recall that the main motivation for introduction of $A\rtimes_{\alpha,\LL} \N$ in \cite{exel2} was to realize Cuntz-Krieger algebras as crossed products  associated with one-sided Markov shifts. This result was adapted in \cite{brv}  to graph $C^*$-algebras $C^*(E)$ where $E$ is a locally finite graph with no sinks or sources (by \cite[Proposition 4.6]{Brownlowe}, it can  be generalized to graphs  admitting sinks).
 For such a graph $E$ the space of infinite paths $E^\infty$ is a locally compact Hausdorff space and the one-sided shift $\sigma: E^\infty \to E^\infty$ is a surjective proper local homeomorphism. In particular,  the formulas
\begin{equation}\label{composition with shift preliminary}
\alpha(a)(\mu)= a(\sigma(\mu)),\qquad    \LL(a)(\mu)=\frac{1}{|\sigma^{-1}(\mu)|}\sum_{\eta \in \sigma^{-1}(\mu)} a(\eta),
\end{equation}
$a\in   C_0(E^\infty)$, $\mu\in E^\infty$, yield well-defined mappings on $C_0(E^\infty)$. Actually, $(C_0(E^\infty), \alpha, \LL)$ is an Exel system and $C_0(E^\infty)$ is naturally isomorphic to the diagonal $C^*$-subalgebra $\D_E$ of $ C^*(E)$. By \cite[Theorem 5.1]{brv},  the isomorphism $C_0(E^\infty)\cong \D_E$ extends to the isomorphism $C_0(E^\infty)\rtimes_{\alpha,\LL} \N\cong C^*(E)$.  In order to  generalize that result to arbitrary graphs one is forced to pass to
a boundary path space $\partial E$ of $E$, cf. \cite{Webster}. Then $C_0(\partial E)\cong \D_E$, but the analogues of maps given by \eqref{composition with shift preliminary} are in general not well defined onto the whole of $C_0(\partial E)$. One possible solution, see \cite{Brownlowe}, is to consider `partial' Exel systems defined in \cite{er}. In the present paper, we circumvent this problem by studying a more general class of `Perron-Frobenious operators' of the form:
\begin{equation}\label{general formula for L of a graph3}
 \LL_\lambda(a)(\mu)=\sum\limits_{e \in E^1,\, e\mu \in \partial E} \lambda_{e}\, a(e\mu), \qquad a\in C_0(\partial E),
\end{equation}
where $\lambda=\{\lambda_e\}_{e\in E^1}$ is a family of strictly positive numbers indexed by the edges of $E$.   We find necessary and sufficient conditions on $\lambda$ assuring that $\LL_\lambda:C_0(\partial E) \to C_0(\partial E)$ is well-defined. For any such  $\lambda$ we get  an isomorphism
$$
C^*(E)\cong C^*(C_0(\partial E) ,\LL_\lambda).
$$
Moreover, the map induced on $ \D_E\cong C_0(E^\infty)$ by $\LL_\lambda$ extends in a natural way to a completely positive map on $C^*(E)$. The latter deserves a name of \emph{non-commutative Perron-Frobenious operator}. This indicates, at least in the present context, a somewhat superior role of a Perron-Frobenious operator $\LL_\lambda$ over the standard non-commutative Markov shift, cf., for instance, \cite{jp}, which  in general is not even well-defined.

Finally, we mention our findings concerning an arbitrary (necessarily completely) positive map $\morp$ on a commutative $C^*$-algebra $A=C_0(D)$, where $D$ is a locally compact Hausdorff space. Any such map defines a relation  on $D$:
$$
(x,y)\in R\, \, \, \stackrel{def}{\Longleftrightarrow} \, \, \, \left(\forall_{a\in A_+}\,\,  \morp(a)(x)=0 \, \Longrightarrow\,  a(y)= 0\right).
$$
If the set $R\subseteq D\times D$ is closed, then $\morp$ give rise to a topological relation  $\mu$ in the sense of  \cite{brenken} and a topological quiver $\QQ$ in the sense of \cite{mt}. Then we prove that for the corresponding $C^*$-algebras associated to $\mu$ and $\QQ$, in  \cite{brenken} and  \cite{mt}  respectively, we  have $C^*(A,\morp; A)\cong \mathcal{C}(\mu)$  and $C^*(A,\morp)\cong C^*(\QQ)$. In particular, if $\morp$ is a Markov operator in the sense of \cite{imv}, the  $C^*$-algebra $C^*(\morp)$ considered in \cite{imv} coincides with $C^*(A,\morp)$. However, as we explain in detail and show by concrete examples, when $R$ is not closed in $D\times D$, then $C^*(A,\morp)$ cannot be modeled in any obvious way by the $C^*$-algebras studied in \cite{mt}. In particular,  an analysis, similar to that in  \cite{imv}, for  general  positive maps on commutative $C^*$-algebras requires a generalization of the theory of
 topological quivers  \cite{mt}.

The content of the paper is organized as follows.

In  Section \ref{preliminary section}, which serves as preliminaries, we gather certain facts on positive maps,  explain in detail what we mean by a universal representation and recall definitions of relative Cuntz-Pimsner algebras. Also, we  present a  definition of Exel's crossed product $A\times_{\al, \LL} \N$ for arbitrary Exel systems $(A,\alpha,\LL)$, and  recall a definition of Exel-Royer's crossed product $\OO(A,\alpha,\LL)$ for  such systems.

In Section \ref{cp section} we introduce  relative crossed products $C^*(A,\morp;J)$ for a completely positive map $\morp:A\to A$.
We present three pictures of $C^*(A,\morp;J)$:  as a quotient of a certain Toeplitz algebra (Definition \ref{definition main}); as a relative Cuntz-Pimsner algebra  associated with a GNS correspondence $X_\morp$ of $(A,\morp)$ (Theorem \ref{identification theorem}); and as a universal $C^*$-algebra generated by  suitably defined covariant representations of $(A,\morp)$ (Proposition \ref{universal description of crossed products}).
 We finish this section by revealing relationships between  construction and various crossed products by endomorphisms (subsection \ref{The case when multiplicative}),  and with $C^*$-algebras associated to topological relations, topological quivers, and Markov operators  (subsection \ref{commutative example}).

In Section \ref{general exel systems section} we show  that the Toeplitz algebra $\TT (A,\alpha,\LL)$ of  $(A,\alpha,\LL)$ coincides with the
 Toeplitz algebra $\TT(A,\LL)$ of $(A,\LL)$ (Proposition \ref{destroying proposition}), which leads us to identities
$
\OO(A,\alpha,\LL)=C^*(A,\LL)$ and $A\times_{\al, \LL} \N = C^*(A,\LL;\overline{A\alpha(A)A})$
(Theorem \ref{Exel crossed products as relative Pimsner}). Using this result we conclude that $
A\times_{\al, \LL} \N=C^*(A,\LL)
$ for instance when $\LL$ faithful and $\alpha$ extendible (Proposition \ref{corollary on faithful transfers}). In Subsection \ref{Regular section} we study Exel systems $(A,\alpha,\LL)$ with the additional property that $E:=\alpha\circ \LL$ is a conditional expectation onto $\alpha(A)$.  We give a number of characterizations  and an intrinsic description of such Exel systems.   This leads us to convenient conditions implying that $
A\times_{\al, \LL} \N$ does not depend on $\alpha$ (cf. Proposition \ref{Exel without endomorphism}). In particular, if $\alpha(A)$ is a hereditary subalgebra of $A$ we prove that  $
A\times_{\al, \LL} \N=C^*(A,\LL)\cong C^*(A,\alpha)$  (Theorem \ref{complete systems crossed products}).

In the closing Section \ref{graphs section},  we analyze the $C^*$-algebra  $C^*(E)=C^*(\{p_v: v\in E^0\}\cup \{s_e: e\in E^1\})$ associated to an arbitrary infinite graph $E=(E^0,E^1,r,s)$. We briefly present Brownlowe's  \cite{Brownlowe} realization of   $C^*(E)$ as Exel-Royer's crossed product for a  partially defined  Exel system $(C_0(\partial E),\alpha,\LL)$.    We find conditions on the numbers $\lambda=\{\lambda_e\}_{e\in E^1}$ assuring that \eqref{general formula for L of a graph3} defines a  self-map on $C_0(\partial E)$ (Proposition \ref{proposition for lambda transfer operators}).  For any such choice of $\lambda$ we prove, using an algebraic picture of the system $(C_0(\partial E), \LL_\lambda)$, that
$
C^*(E)\cong C^*(\D_E,\LL)$,  where $\LL(a):=\sum_{e\in E^1} \lambda _e  s_e^* a s_e$, $a \in \D_E$ (see Theorem \ref{thm for Raeburn, Brownlowe and Vitadello}).
 If $E$ is locally finite and without sources then the (non-commutative) Markov shift $\alpha(a):=\sum_{e\in E^1} s_e a s_e^*$ is the unique endomorphism of $\D_E$ such that $(\D_E,\alpha,\LL)$ is an Exel system and $C^*(\D_E,\LL)=\D_E\times_{\al, \LL} \N$ (Theorem \ref{thm for Raeburn, Brownlowe and Vitadello} iii)). In general there is no endomorphism making $(\D_E,\alpha,\LL)$  an Exel system (Theorem \ref{thm for Raeburn, Brownlowe and Vitadello} ii)).
One of  possible interpretations of these results is that in order to  associate a non-commutative shift to an arbitrary infinite graph one is forced to fix a certain measure system and encode the shift in its `transfer operator', as the `composition endomorphism' does not exist.

\subsection{Conventions and notation}
  All ideals in  $C^*$-algebras (unless stated otherwise) are assumed to be closed and two-sided. If $I$ is an ideal in a $C^*$-algebra $A$ we denote by
	$I^{\bot}=\{a\in A: aI={0}\}$ the annihilator of $I$. We denote by $1$ the unit in the multiplier  algebra $M(A)$ of $A$.  Any approximate unit in $A$ is assumed to compose of contractive positive elements. All homomorphisms between $C^*$-algebras are assumed to be $*$-preserving.  For  actions $\gamma\colon A\times B\to C$
such as  multiplications, inner products, etc., we use the notation:
$$
\gamma(A,B)=\{\gamma(a,b) :a\in A, b\in B\}, \qquad \overline{\gamma(A,B)}=\clsp\{\gamma(a,b) :a\in A, b\in B\}.
$$
By the Cohen-Hewitt Factorization Theorem we have $\gamma(A,B)=\overline{\gamma(A,B)}$ whenever $\gamma$ can be interpreted as a continuous representation of a $C^*$-algebra $A$ on a Banach space $B$. We emphasize that we will use this fact without  further warning. In particular, a $C^*$-subalgebra $A$ of a $C^*$-algebra $B$ is non-degenerate   if  $AB=B$.

\section{Preliminaries} \label{preliminary section}
In this section, we present certain facts concerning positive maps. Most of them are known, but usually they are stated in the literature in the unital case. We also present definitions of a universal $C^*$-algebra and a universal representation, which are  well suited for our analysis. We  briefly recall definitions of  $C^*$-algebras associated to $C^*$-correspondences.  In the last part of this section, we introduce a definition of Exel crossed product for arbitrary Exel systems, and also recall the definition of  crossed products associated to
 such systems in \cite{er}.

\subsection{Positive maps}
Throughout this subsection we fix a positive map  $\morp:A\to B$ between two $C^*$-algebras $A$ and $B$.
This means that $\morp:A\to B$  is linear and $\morp(aa^*)\geq 0$ for every $a\in A$. Such $\morp$ is automatically $*$-preserving: $\morp(a^*)=\morp(a)^*$, $a\in A$; and bounded, see \cite[Lemma 5.1]{Lan}. We have the following formula for the norm of $\morp$, which is  well known for completely positive maps, cf. \cite[Lemma 5.3(i)]{Lan}, and less known for positive maps.\footnote{the author thanks Paul Skoufranis for providing the following short proof.}
\begin{lem}\label{boundness equality for a cp maps} For any  approximate unit  $\{\mu_\lambda\}_{\lambda\in \Lambda}$  in $A$ the norm of $\morp$ is given by the limit
$\|\morp\|=\lim_{\lambda\in \Lambda} \|\morp(\mu_\lambda)\|.
$
\end{lem}
\begin{proof} Recall that the double dual $A^{**}$ of $A$ can be identified with the enveloping  von  Neumann algebra of $A$, cf.   \cite[III, Theorem 2.4]{Tak}.
Similarly for $B$. Then the double dual $\morp^{**}:A^{**}\to B^{**}$ of $\morp:A\to B$ is  a $\sigma$-weakly continuous  extension of $\morp$.
As positive elements in $A$ are $\sigma$-weakly dense in the set of positive elements in $A^{**}$, $\morp^{**}$ is positive.  Hence  Russo-Dye theorem implies, see \cite[Corollary  2.9]{Paulsen},  that $
\|\morp^{**}\|=\|\morp^{**}(1)\|.
$
Moreover, since $\{\mu_\lambda\}_{\lambda\in \Lambda}$ converges $\sigma$-weakly  to $1$, $\{\morp^{**}(\mu_\lambda)\}_{\lambda\in \Lambda}$ converges $\sigma$-weakly to $\morp^{**}(1)$. Since the norm is weakly lower-semicontinuous we get
$$
\|\morp^{**}\|=\|\morp^{**}(1)\|\leq \liminf_{\lambda\in \Lambda} \|\morp^{**}(\mu_\lambda)\|=\liminf_{\lambda\in \Lambda} \|\morp(\mu_\lambda)\|.
$$
As clearly we have $
\limsup_{\lambda\in \Lambda} \|\morp(\mu_\lambda)\|\leq \|\morp\|\leq \|\morp^{**}\|
$, we get the desired equality.
\end{proof}
The formula for the kernel of the classic GNS representation yields also an important ideal for an arbitrary positive map.
\begin{prop}\label{almost faithfulness corollary}
The set
\begin{equation}\label{kernel of dynamical cp system}
N_\morp:=\{a\in A: \morp((ab)^*ab))=0 \textrm{ for all }b\in A\}
\end{equation}
is the largest ideal in $A$ contained in the kernel of the mapping $\morp:A\to B$.
\end{prop}
\begin{proof}
Obviously, $N_\morp$ is a closed right ideal in $A$. Let $a,b\in A$.  Since $b^*a^*ab \leq \|a^*a\| b^*b$  we get  $\morp((ab)^*ab)\leq \|a^*a\| \morp(b^*b)$. The latter inequality implies that $N_\morp$ is  a left ideal. In particular, if $a$ is a positive element in $N_\morp$, then $a^{1/4}\in N_\morp$ and therefore $\morp(a)= \morp((a^{1/4}a^{1/4})^*a^{1/4}a^{1/4})=0$. This implies that $N_\morp\subseteq \ker\morp$. Clearly, if  $I$ is an ideal in $A$ contained in  $\ker\morp$, then  $I\subseteq N_\morp$.
\end{proof}
\begin{dfn}\label{GNS-kernel}
We call the ideal in \eqref{kernel of dynamical cp system} the \emph{GNS-kernel} of  $\morp:A\to B$.
\end{dfn}
The ideal \eqref{kernel of dynamical cp system}  is closely related to the notion of almost faithfulness introduced, in the context of Exel systems, in \cite{br}.
Namely,   following \cite[Definition 4.1]{br}, we say that $\morp$ is \emph{almost faithful} on an ideal $I$ in $A$ if
$$
a\in I\textrm{ and } \morp((ab)^*ab))=0 \textrm{ for all }b\in A \,\, \Longrightarrow a=0.
$$
The above implication is equivalent to the equality $I\cap  N_\morp=\{0\}$. In other words,
$$
\morp \textrm{ is almost faithful on } I\,\,  \Longleftrightarrow \,\, I \subseteq N_\morp^\bot.
$$
In particular,   the annihilator $N_\morp^\bot$ of the GNS-kernel of $\morp$ is the largest ideal in $A$ on which $\morp$ is almost faithful. We recall that $\morp$ is \emph{faithful} on a $C^*$-subalgebra $C\subseteq A$ if for any $a\in C$,  $\morp(a^*a)=0$ implies $a=0$. The following lemma sheds considerable light on the relationship between the two aforementioned notions.
\begin{lem}\label{faithful vs almost faithful}
Let $C\subseteq A$ be a $C^*$-subalgebra  and consider the following conditions:
\begin{itemize}
\item[i)] $\morp$ is faithful on the ideal $\overline{ACA}$,
\item[ii)] $\morp$ is faithful on the hereditary $C^*$-subalgebra $CAC$,
\item[iii)] $\morp$ is almost faithful on the ideal $\overline{ACA}$.
\end{itemize}
Then \emph{i)} $\Rightarrow$ \emph{ii)} $\Rightarrow$ \emph{iii)}  and if $A$ is commutative then the above conditions are  equivalent.
\end{lem}
\begin{proof}
The  inclusion $CAC\subseteq\overline{ACA}$  yields the implication i)$\Rightarrow$ ii).
\\
ii) $\Rightarrow$ iii). Let $a\in N_\morp$ and $c\in C$. Since $N_\morp$ is an ideal, we have   $\morp((acb)^*acb)=0$ for all $b\in A$. Taking  $b=c^*$ and using  faithfulness of  $\morp$ on $CAC$ we  infer that  $(a cc^*)^*a cc^*=0$.  This implies that $ac=0$. Accordingly, $C\subseteq N_\morp^\bot$  and since $N_\morp^\bot$ is an ideal in $A$ we get  $\overline{ACA}\subseteq N_\morp^\bot$.
\\
Assume now that  $A$ is commutative and $ \overline{ACA}\subseteq N_\morp^\bot$. Consider an element $ac$ of $\overline{ACA}=AC$,  $a\in A$, $c\in C$,  such that $\morp((ac)^*ac)=0$.    For all $b\in A$ we have
$$
\morp((acb)^*acb)=\morp((bac)^* bac) \leq \|b^*b\|\morp((ac)^*ac)=0.
$$
Thus  (by almost faithfulness)  $ac=0$. Hence  $\morp$ is faithful on $\overline{ACA}=CA$.
\end{proof}

There is a natural $C^*$-subalgebra of $A$ on which $\morp$ is multiplicative.
\begin{dfn}
Let   $\morp:A\to B$ be a positive map. We call the set
\begin{equation}\label{multiplicative domain 2}
MD(\morp):=\{a\in A: \morp(b)\morp(a)=\morp(ba)\textrm{ and } \morp(a)\morp(b)=\morp(ab) \textrm{ for every  }b\in A \}
\end{equation}
the \emph{multiplicative domain of $\morp$}.
\end{dfn}
It is immediate that $MD(\morp)$ is a $C^*$-subalgebra of $A$. Hence $\morp:MD(\morp)\to B$ is a homomorphism of  $C^*$-algebras.
 In the literature, see e.g. \cite[p. 38]{Paulsen}, multiplicative domains are considered for contractive completely positive maps, which is due to the fact we express in Proposition \ref{multiplicative domain for ccp maps} below.  We recall  that $\morp$ is \emph{completely positive} if for every integer $n>0$ the amplified map $\morp^{(n)}:M_n(A)\to M_n(B)$ obtained by applying $\morp$ to each matrix element: $\morp^{(n)}\left((a_{ij})\right)=\left(\morp(a_{ij})\right)$, is positive, see \cite[p. 5]{Paulsen}, \cite[p. 39]{Lan}, or \cite[IV, Definition 3.3]{Tak}.
It is not hard to show, cf. \cite[Remark 5.1]{Paschke}, see \cite[IV, Corollary 3.4]{Tak}, that a linear map $\morp:A\to B$ is completely positive if and only if
$$
\sum_{i,j=1}^n b_i^*\morp(a_i^* a_j)b_j \geq 0, \qquad \textrm{for all }a_1,...,a_n \in A \textrm{ and } b_1,...,b_n \in B.
$$
The following fact  is a generalization of \cite[Theorem 3.18]{Paulsen} to not necessarily unital completely positive maps on not necessarily unital $C^*$-algebras.
\begin{prop}\label{multiplicative domain for ccp maps}
 Let $\morp:A\to B$ be a contractive completely positive map between $C^*$-algebras. Then
\begin{equation}\label{multiplicative domain 1}
MD(\morp)=\{a\in A: \morp(a)^*\morp(a)=\morp(a^*a)\textrm{ and } \morp(a)\morp(a)^*=\morp(aa^*) \}.
\end{equation}
In particular,  $MD(\morp)$ is the largest $C^*$-subalgebra of $A$ on which $\morp$ restricts to a homomorphism.
\end{prop}
\begin{proof}
To show the  equality \eqref{multiplicative domain 1} note that the argument of the proof of  \cite[Theorem 3.18]{Paulsen} applies, only modulo the fact that the Schwarz inequality
$$
\morp^{(2)}(a^*)\morp^{(2)}(a)\leq \morp^{(2)}(a^*a), \qquad \quad a \in M_2(A),
$$
used there holds for arbitrary contractive completely positive maps, see \cite[Lemma 5.3 (ii)]{Lan}.
Plainly, \eqref{multiplicative domain 1} implies that for any $C^*$-subalgebra $C$ of $A$ such that $\morp:C\to B$ is a homomorphism we have $C\subseteq  MD(\morp)$.
\end{proof}
We recall that any positive map $\morp:A\to B$ is automatically completely positive whenever $A$ or $B$ is commutative \cite[Corollary 3.5, Proposition 3.9]{Tak}. Of course any homomorphism is a completely positive contraction. Also it is well known, cf., for instance, \cite[III, Theorem 3.4, IV, Corollary 3.4 ]{Tak}, that if $B$ is a $C^*$-subalgebra of $A$ then for a linear idempotent $E:A\to B$ we have
\begin{align*}
E \textrm{ is contractive } &\Longleftrightarrow E \textrm{ is  positive and } B\subseteq MD(E)\\
 &\Longleftrightarrow E \textrm{ is completely positive and } B\subseteq MD(E)
\end{align*}
An idempotent $E$ satisfying the above equivalent conditions is called a \emph{conditional expectation}.

\begin{dfn}
Let $\morp:A\to B$ be a  positive map. We say that
  $\morp$ is \emph{strict}  if  $\{\morp(\mu_\lambda)\}_{\lambda\in \Lambda}$ is strictly convergent in $M(A)$ for some approximate unit $\{\mu_\lambda\}_{\lambda\in \Lambda}$ in $A$. We say that $\morp$ is \emph{extendible} if it extends to a strictly continuous mapping $\overline{\morp}:M(A)\to M(B)$.
\end{dfn}
\begin{rem} The  positive elements in $A$ are strictly dense in the set of positive elements in $M(A)$. Thus  if $\morp$ is an extendible (completely) positive map, then $\overline{\morp}:M(A)\to M(B)$ is also (completely) positive.  Clearly, every extendible map is strict, and it is well known that for homomorphisms  these notions are actually equivalent.
\end{rem}

\subsection{Universal $C^*$-algebras}
Defining a universal $C^*$-algebra for a given set of generators $\G$ subject to a set of relations $\RR$ can be   tricky as free $C^*$-algebras do not exist. A recent and perhaps the most compelling   approach   is elaborated by  Loring  in \cite{loring}, see \cite{loring} for references to previous approaches.  We propose a slightly more general  framework that fits our setting.  As in \cite{loring}  we concentrate on a class of representations of $\G$  that are determined by prescribed relations, rather than on the relations themselves.

\begin{dfn}
Let $\G$ be a set and let $\RR$ be a certain class of maps  from $\G$ to  $C^*$-algebras.  We refer to elements of $\RR$ as to \emph{representations of} $\G$.  We define a preorder relation on $\RR$ by writing $\pi\precsim \sigma$ for any $\pi, \sigma \in \RR$ such that the map
\begin{equation}\label{maps on generators to be extended}
\sigma(a)\longmapsto \pi(a),\qquad a\in \G,
\end{equation}
extends to a (necessarily unique) homomorphism from $C^*(\sigma(\G))$ onto $C^*(\pi(\G))$. 
We denote by $\approx$ the equivalence relation on $\RR$ induced by  this  preorder:
$\pi \approx \sigma$ $\Longleftrightarrow$  $\pi \precsim \sigma$  and $\sigma \precsim \pi$.
\end{dfn}
\begin{rem}\label{trivial universal lemma}
Let $\RR$ be a class of representations of a set $\G$. It is straightforward to see that, if $\pi, \sigma \in \RR$ are such that $\pi\approx \sigma$, then the map \eqref{maps on generators to be extended} extends to an isomorphism $C^*(\sigma(\G))\cong C^*(\pi(\G))$. Moreover, if $\pi, \sigma \in \RR$ are upper bounds for $(\RR, \precsim)$ then $\pi\approx \sigma$.
\end{rem}

\begin{dfn}\label{definition of a universal algebra}
Suppose that   a class $\RR$ of representations of a set $\G$ admits an upper bound $\iota \in \RR$, that is,   $\pi\precsim \iota$ for all $\pi\in \RR$. By the obvious abuse of language, cf.  Remark \ref{trivial universal lemma}, we say that  $\iota$   \emph{the universal representation}  of $\G$ and the $C^*$-algebra
$$
C^*(\G,\RR):=C^*(\iota(\G))
$$
 is \emph{the universal $C^*$-algebra for} $\RR$.
\end{dfn}
\begin{rem}
By the definition of the universal $C^*$-algebra  $C^*(\G,\RR)$, for any $\pi\in \RR$  the map $\iota(a)\mapsto \pi(a)$,  $a\in \G$, extends to a (necessarily unique) epimorphism from $C^*(\iota(\G))$ onto $C^*(\pi(\G))$. We will write equality between any two $C^*$-algebras generated by ranges of two universal representations for the same class of representations of the same set of generators.
\end{rem}

Now, we give a condition on $(\G,\RR)$ that imply existence of an upper bound in $\RR$.
A version of this condition appears in all of the previous approaches starting from Blackadar's \cite[Definition 1.1]{blackadar}.  It appears also  in Loring's characterisation \cite[Theorem 3.1.1]{loring} of $C^*$-relations admitting universal $C^*$-algebras.

For any set  of $C^*$-algebras $B_i$, $i\in I$, we denote their direct product by  $\prod_{i\in I}B_i$, so the elements of  $\prod_{i\in I}B_i$ are  $\prod_{i\in I}a_i$ where $a_i\in B_i$, $i\in I$, and $\sup_{i\in I}\|a_i\|<\infty$.
\begin{dfn}\label{closed under products definition}
We say that a class $\RR$ of representations of a set $\G$ is \emph{closed under products} if for every set  of mappings $\pi_i:\G\to B_i$,  $i\in I$, that belong to $\RR$ the following two conditions are satisfied:
\begin{itemize}
\item[i)] $\prod_{i\in I}\pi_i(a)\in \prod_{i\in I}B_i$ for all $a\in \G$.
\item[ii)] there exists an injective homomorphism $\tau:\prod_{i\in I}B_i\to B$ into a $C^*$-algebra $B$ such that the map $\pi:\G\to B$ given by $\pi(a):=\tau\left(\prod_{i\in I}\pi_i(a)\right)$, $a\in \G$,  belongs to $\RR$.
\end{itemize}
\end{dfn}

\begin{prop}\label{existence of universal stuff}
If   a  class $\RR$ of representations of a set $\G$ is  closed under products then $\RR$ has an upper bound and therefore the universal
$C^*$-algebra $C^*(\G,\RR)$ exists.
\end{prop}
\begin{proof}
First we need to show that the collection $\RR/\approx$ of equivalence classes for $\approx$ form a set.
To this end, we  note that there is a Hilbert space $H$ with the property that any $C^*$-algebra  $B$ generated by $|\G|$ generators can be embedded into $B(H)$. Indeed, any GNS representation of $B$ is determined by a function from the set  of generators to complex numbers, and the dimension of the resulting Hilbert space cannot exceed the cardinality of the free $*$-algebra $\FF(\G)$ generated by $\G$. Hence, by GNS construction, there is a faithful representation of  $B$ on a Hilbert space with dimension not exceeding  $|\{f:\G\to \C\}|\cdot |\FF(\G)|$. This implies our claim.

Let $H$ be the aforesaid Hilbert space and denote by $F$ the set of all mappings from $\G$ into $B(H)$. For each $\pi\in \RR$ we choose an embedding $\phi_\pi:C^*(\pi(\G))\to B(H)$, so that $\phi_{\pi}\circ \pi \in F$. We define an equivalence relation on $F$ in a similar fashion as  we did for $\RR$.    For $\pi, \sigma \in F$ we write
$$
\pi \approx_F \sigma \, \, \Longleftrightarrow \, \, \textrm{the map \eqref{maps on generators to be extended} extends to an isomorphism } C^*(\sigma(\G))\cong C^*(\pi(\G)).
$$
It is straightforward to see that, for any $\pi, \sigma \in \RR$ we have  $\pi\approx \sigma$ if and only if  $ \phi_{\pi}\circ \pi \approx_F  \phi_{\sigma}\circ \sigma$. Thus the assignment  $\RR\ni  \pi\longmapsto \phi_{\pi}\circ \pi \in F$ factors through to a  bijective assignment from $\RR/\approx$ onto a subset $I$ of the set $F/\approx_F$.

Accordingly, there is a set $\{\pi_i\}_{i\in I}\subseteq \RR$  such that for any $\sigma\in \RR$ we have $\sigma\approx \pi_i$ for some $i\in I$.
 Let $\pi \in \RR$ be the product of these representations $\{\pi_i\}_{i\in I}$
 as described in Definition \ref{closed under products definition}, so that, $\pi(a):=\tau\left(\prod_{i\in I}\pi_i(a)\right)$, $a\in \G$, for an injective homomorphism $\tau:\prod_{i\in I}B_i\to B$. To see that $\pi$ is an upper bound for $\RR$, let $\sigma\in \RR$. Choose $i_0\in I$ such that $\sigma\approx \pi_{i_0}$. Let $\Psi: C^*(\pi_{i_0}(\G))\to  C^*(\sigma(\G))$ be the isomorphism determined by $\Psi(\pi_{i_0}(a))=\sigma(a)$,  $a\in \G$. Denote by $p_{i_0}:\prod_{i\in I}B_i\to B_{i_0}$ the projection onto $B_{i_0}$ and let  $\tau^{-1}$ denote the inverse to the isomorphism $\tau:\prod_{i\in I}B_i\to \tau(\prod_{i\in I}B_i)$. Putting $\Phi:=\Psi\circ p_{i_0}\circ \tau^{-1}$ we get
$$
\Phi(\tau(a))=(\Psi\circ p_{i_0}) \left(\prod_{i\in I}\pi_i(a)\right)=\Psi(\pi_{i_0}(a))=\sigma(a), \quad\textrm{ for all } a\in \G.
$$
Hence $\Phi: C^*(\pi(\G)\to C^*(\sigma(\G))$ is the  homomorphism showing that $\sigma \precsim \pi$.
\end{proof}

In the present paper,  we will consider only two types of generators and their representations.   One type  comes from a $C^*$-correspondence $X$ over a $C^*$-algebra $A$. Then   the set of generators is  $\G=A\cup X$ and we  identify representations $\sigma$ of $\G$  with pairs $(\pi,\pi_X)$ where $\pi=\sigma|_{A}$ and  $\pi_X=\sigma|_{X}$. Another type comes from a $C^*$-dynamical system or an Exel system on a  $C^*$-algebra $A$. Then the set of generators is $\G=A\cup \{s\}$ where $s$ is an abstract element and we  identify representations $\sigma\in \RR$ with pairs $(\pi,S)$ where $\pi=\sigma|_{A}$ and  $S=\sigma(s)$. In the latter case we  will study the  $C^*$-subalgebra
$C^*(\iota(A)\cup \iota(A)\iota(s))$ of the universal $C^*$-algebra $C^*(\G,\RR)=C^*(\iota(\G))$ (which can be also viewed as a universal $C^*$-algebra but with a different set of generators).

\subsection{Relative Cuntz-Pimsner algebras}
We assume that the reader is familiar with the theory of Hilbert modules (for an  introduction see, for instance, \cite{Lan}).
 A (right) \emph{$C^*$-correspondence} over a $C^*$-algebra $A$ is a right Hilbert $A$-module $X$ together with a left action of $A$ on $X$ given by a homomorphism $\phi$ of $A$ into the $C^*$-algebra $\mathcal{L}(X)$ of all adjointable operators on $X$: we write $a\cdot x=\phi(a)x$.
Sometimes $C^*$-correspondences  are  called Hilbert bimodules, see  \cite{br}, \cite{brv}. However, it seems to become a standard to use the term Hilbert bimodule in the sense of \cite[Definition 3.1]{BMS}. Namely, by a \emph{Hilbert bimodule} over $A$ we mean a  space  $X$ which is at the same time a right and a left $C^*$-correspondence and  the corresponding right $\langle \cdot  , \cdot  \rangle_A$ and left ${_A\langle} \cdot , \cdot  \rangle$  $A$-valued inner products satisfy
 $x \cdot \langle y ,z \rangle_A = {_A\langle} x , y  \rangle \cdot z$,  for all $x,y,z\in X$,
cf.  \cite[Definition 3.1]{katsura1} or \cite[Definition 1.10]{kwa-doplicher} and the remarks below these definitions.
\begin{dfn}
A representation $(\pi,\pi_X)$ of a $C^*$-correspondence $X$   consists of a representation $\pi:A\to \B(H)$ in a Hilbert space $H$ and a linear map $\pi_X:X\to \B(H)$ such that
 $$
 \pi_X(a\cdot x\cdot b)=\pi(a)\pi_X(x)\pi(b),\quad \pi_X(x)^*\pi_X(y)=\pi(\langle x, y\rangle_A),\quad  a,b\in A,\, x\in X.
 $$
The $C^*$-algebra generated by $\pi(A)\cup \pi_X(X)$ is denoted by $C^*(\pi,\pi_X)$.
\end{dfn}
\begin{rem}
If $(\pi,\pi_X)$ is a representation of a $C^*$-correspondence $X$ then for each $x\in X$ we have $\|\pi_X(x)\|^2=\|\pi(\langle x, x\rangle_A)\|\leq \|\langle x, x\rangle_A\|=\|x\|$. Thus the map $\pi_X$ is automatically contractive (it is isometric if  $\pi$ is faithful), and  using Proposition \ref{existence of universal stuff} one readily sees that a universal representation of $X$  exists.
\end{rem}
\begin{dfn}[Pimsner]
We denote by   $(i_A,i_X)$ the  universal representation of a $C^*$-corres\-pondence $X$ and we call $\TT(X):=C^*(i_A,i_X)$ the \emph{Toeplitz algebra}  of $X$.
\end{dfn}
\begin{rem}
Originally,  Pimsner  \cite{p} constructed the Toeplitz algebra $\TT(X)$ by means of the  Fock representation of $X$, which as he noticed  is the universal representation of $X$.
\end{rem}
We recall that the set $\K(X)$ of  \emph{generalized compact operators} on $X$   is the closed linear span  of  the operators $\Theta_{x,y}$ where $\Theta_{x,y} (z)=x \langle y,z\rangle_A$ for $x, y, z\in X$. In particular, $\K(X)$ is an ideal in $\LL(X)$.
 Any  representation $(\pi,\pi_X)$ of $X$ induces a homomorphism
 $(\pi,\pi_X)^{(1)}: \K(X) \to \B(H)$ which satisfies
 $$
 (\pi,\pi_X)^{(1)}(\Theta_{x,y})=\pi_X(x)\pi_X(y)^*, \qquad (\pi,\pi_X)^{(1)}(T)\pi_X(x)=\pi_X(Tx)
 $$
 for $x, y \in X$ and $T\in \K(X)$, cf. \cite[Page 202]{p} or \cite[Proposition 4.6.3]{kpw}. Let $J(X):=\phi^{-1}(\K(X))$. For any representation $(\pi,\pi_X)$ of $X$ the restrictions $(\pi,\pi_X)^{(1)}\circ \phi|_{J(X)}$ and  $\pi|_{J(X)}$ yield two representations of $J(X)$. Putting constraints on the set on which these two representations coincide leads us to  the following definition.
\begin{dfn}[Muhly and Solel]
Let $J$ be an ideal in $J(X)=\phi^{-1}(\K(X))$. We say that a representation $(\pi,\pi_X)$ of  $X$ is $J$-\emph{covariant} if
$$
(\pi,\pi_X)^{(1)}(\phi(a))=\pi(a), \qquad \textrm{ for all }a\in J.
$$
We denote by $(j_A,j_X)$ the universal $J$-covariant representation of $X$ and we call the $C^*$-algebra $\OO(J,X):=C^*(j_A,j_X)$ the  \emph{relative Cuntz-Pimsner algebra} determined by $J$.
\end{dfn}
\begin{rem}\label{remark on relative Cuntz-Pimsners} It is clear that the relative Cuntz-Pimsner algebra $\OO(J,X)$ is naturally isomorphic to the quotient of the Toeplitz algebra $\TT(X)$ by the ideal generated by $
\{i_A(a)- (i_A,i_X)^{(1)}(\phi(a)): a\in J\}$. Actually, Muhly and Solel \cite[Definition 2.18]{ms} introduced   the  $C^*$-algebras $\OO(J,X)$ as quotients of $\TT(X)$. The $C^*$-algebra $\OO(J(X),X)$, related to the ideal $J(X)$, coincides with the (augmented) $C^*$-algebra associated to $X$ by Pimsner  \cite{p}.
\end{rem}
Katsura \cite{katsura1},  \cite{katsura} observed that among the relative Cuntz-Pimsner algebras $\OO(J,X)$, perhaps, the most natural one is determined by the ideal $J$ equal to
\begin{equation}\label{katsura's ideal}
J_X:=(\ker\phi)^\bot \cap J(X).
\end{equation}
In particular,   \cite[Proposition 2.21]{ms} and  \cite[Proposition 3.3]{katsura}, see also  \cite[Lemma 2.3]{br}, imply the following proposition.
\begin{prop}\label{proposition for the referee}
Let $X$ be a $C^*$-correspondence and let $J$ be an ideal in $J(X)$. The universal representation $j_A:A\to  \OO(J,X)$ is injective if and only if $J\subseteq (\ker\phi)^\bot$.
\end{prop}
\begin{dfn}[Katsura \cite{katsura1}, Definition 2.6] The (unrelative) \emph{Cuntz-Pimsner algebra} associated to a $C^*$-correspondence $X$
is $\OO_X:=\OO(J_X,X)$ where $J_X$ is Katsura's ideal \eqref{katsura's ideal}.
\end{dfn}

If  the $C^*$-correspondence $X$ is \emph{essential}, that is if $AX=X$, we may restrict our attention  to representations $(\pi,\pi_X)$ where $\pi$ is  non-degenerate. This is  due to the  following statement which was proved in \cite{brv} for a certain concrete $C^*$-correspondence $X$. However, the proof uses only the fact that $X$ is essential.
 \begin{lem}[\cite{brv}, Lemma 3.4]\label{lemma 0.1}
 For any representation $(\pi,\pi_X)$ of an essential $C^*$-correspon\-dence $X$ on the Hilbert space $H$, the  subspace $K=\pi(A)H$ is reducing for $(\pi,\pi_X)$ and we have $\pi|_{K^\bot}=0$ and $\pi_X|_{K^\bot}=0$.
 \end{lem}
Since all the $C^*$-correspondences considered in the text will be essential,
\begin{quote}
 \emph{all the representations $(\pi,\pi_X)$ of $C^*$-correspondences will be assumed to be non-degenerate,}
\end{quote}
 in the sense that $\pi$ is non-degenerate. It will force our universal homomorphisms to be also non-degenerate. We recall that a homomorphism $h:A\to B$ between two $C^*$-algebras is \emph{non-degenerate} if $h(A)$ is non-degenerate in $B$, that is if $h(A)B=B$.

\subsection{Exel's and Exel-Royer's crossed products}

Initially, Exel defined his  crossed product for unital  $C^*$-algebras  \cite{exel2},  and then it was generalized  in \cite{brv}, \cite{Larsen} to  Exel systems that consist of extendible maps.  Nevertheless, the definition of the crossed product makes sense for an arbitrary Exel system and can be expressed as follows.

\begin{dfn}
Let   $\al:A\to A$ be an endomorphism of a $C^*$-algebra $A$ and let $\LL:A\to A$ be  a positive linear map such that
\begin{equation}\label{transfer operator relation}
\LL(a\al(b))=\LL(a)b, \qquad \textrm{for all }a,b \in A.
\end{equation}
Then $\LL$ is called  a \emph{transfer operator} for $\alpha$ and the triple $(A,\al,\LL)$ is  an \emph{Exel system}.
\end{dfn}
\begin{dfn}\label{exel crossed product definition} A \emph{representation of an Exel system} $(A,\al, \LL)$ is a pair $(\pi,S)$ consisting of   a non-degenerate representation $\pi:A\to \B(H)$ and an operator $S\in \B(H)$ such that
\begin{equation}\label{Exel relations}
S\pi(a) =\pi(\al(a))S\,\,\, \textrm{  and  }\,\,\,  S^*\pi(a)S=\pi(\LL(a)) \textrm{ for all } a\in A .
\end{equation}
A \emph{redundancy} of a representation $(\pi,S)$ of $(A,\al, \LL)$  is a pair $(\pi(a),k)$ where $ a\in A$ and $ k\in  \overline{\pi(A)SS^*\pi(A)}$ are such that
$$
\pi(a)\pi(b)S=k\pi(b)S, \qquad \textrm{ for all }  b\in A.
$$
The \emph{Toeplitz algebra} $\TT(A,\al,\LL)$ of $(A,\al,\LL)$ is the  $C^*$-algebra generated by $i_A(A)\cup i_A(A)t$ for a universal  representation   $(i_A,t)$ of $(A,\al,\LL)$. \emph{Exel's crossed product}  $A\times_{\al, \LL} \N$  of $(A,\al, \LL)$ is the  quotient $C^*$-algebra of $\TT(A,\al,\LL)$ by the ideal generated by the set
$$
\{i_A(a)-k: a\in \overline{A\al(A)A} \textrm{ and }(i_A(a),k) \textrm{ is a redundancy of } (i_A,t)\}.
$$
\end{dfn}

Existence of the universal representation  $(i_A,t)$ of an Exel system $(A,\al,\LL)$ can be deduced from Proposition \ref{existence of universal stuff}, cf. the proof of Lemma \ref{existence resistance} below. It can be also obtained by realizing $\TT(A,\al,\LL)$ as a Toeplitz algebra of a  $C^*$-correspondence $M_\LL$ introduced  by Exel in   \cite{exel2}.

More specifically, let  $(A,\al,\LL)$ be an Exel system. One makes $A$ into a semi-inner product (right) $A$-module $A_{\LL}$ by putting
$
m\cdot a :=m\al(a)$, $\langle m,n \rangle_{\LL}:=\LL(m^*n)$, $n,m\in A_\LL, a\in A$,
and defines $M_{\LL}$ to be the associated Hilbert $A$-module:
$$
M_{\LL}:=\overline{A_\LL/N}, \qquad N:=\{m \in A_\LL:\langle m,m \rangle_{\LL} =0\}.
$$
Denoting by $q:A_\LL \to M_{\LL}$ the quotient map one gets,  cf. \cite{exel2}, \cite{br}, that
$$
a \cdot q(m):= q(am), \qquad m \in A_\LL,\,\, a\in A,
$$
yields a well defined left action of $A$ on $M_{\LL}$ making $M_{\LL}$ into  a $C^*$-correspondence.  We note that  $A\cdot M_{\LL}=M_{\LL}$, that is $M_{\LL}$ is an essential $C^*$-correspondence. Moreover, the kernel of the left action of $A$ on $M_{\LL}$ coincides with the GNS-kernel $N_\LL$ of $\LL$, cf. Definition \ref{GNS-kernel}.

The following fact was proved in \cite[Lemmas 3.2 and 3.3]{brv} for extendible Exel systems. However,  the proofs exploit only  extendability of a transfer operator. We will show in Proposition \ref{extendibility of transfer operators} below, that this is automatic. Another reason for omitting the proof of the next Proposition \ref{slight generalization of brv} is that  it will follow from our more general results, cf. Corollary \ref{Toeplitz corollary}. We include it here, for the sake of discussion.
\begin{prop}\label{slight generalization of brv}
We have a one-to-one correspondence between  representations $(\pi,S)$ of  $(A,\al,\LL)$ and   representations $(\pi,\pi_{M_{\LL}})$ of the $C^*$-correspondence $M_{\LL}$. In particular, we have an isomorphism $\TT(A,\al,\LL)\cong  \TT(M_{\LL})$.
\end{prop}
Previous versions of the above result  were a point of departure in \cite{er}. More specifically,  the authors of \cite{er} considered `partial Exel systems' $(A,\alpha, \LL)$ where $\LL$ is not everywhere defined and $\alpha$ may attain values outside of $A$. For such triples they defined a crossed product $\OO(A,\alpha, \LL)$, in essence,  simply to be $\OO_{M_{\LL}}$, where $M_{\LL}$  is a generalization of the $C^*$-correspondence defined above to the `partial case'. In the present paper we will only make use of  \cite[Definition 1.6]{er} applied to `global' Exel systems. Thus we adopt the following definition.
\begin{dfn}
The \emph{Exel-Royer's crossed product} $\OO(A,\alpha, \LL)$ associated to an Exel system $(A,\al,\LL)$ is the
 quotient  of $\TT(A,\al,\LL)$ by the ideal generated by the set
$$
\{i_A(a)-k: a\in J_{M_{\LL}} \textrm{ and }(i_A(a),k) \textrm{ is a redundancy of } (i_A,t)\}
$$
where $J_{M_{\LL}}$ is Katsura's ideal \eqref{katsura's ideal} associated to the $C^*$-correspondence $M_{\LL}$.
\end{dfn}
\begin{rem}\label{Remark for the referee who loves Kastura's work}
Since the kernel of the left action of $A$ on $M_\LL$ is equal to $N_\LL$, we have $J_{M_{\LL}}=N_\LL^\bot\cap J(M_\LL)$. It can be readily deduced from Proposition \ref{slight generalization of brv} and Remark \ref{remark on relative Cuntz-Pimsners}, see \cite[Proposition 4.5]{Brownlowe}, that we have a natural isomorphism $\OO(A,\alpha, \LL)\cong \OO_{M_{\LL}}$.
\end{rem}
\section{Crossed products by  completely positive  maps}\label{cp section}
Throughout this section, we fix a  completely positive map $\morp:A\to A$, and refer to the  pair $(A,\morp)$ as to a \emph{$C^*$-dynamical system}. We introduce relative crossed products $C^*(A,\morp;J)$ as quotients of a certain Toeplitz algebra. Then we realize them as relative Cuntz Pimsner algebras and  as universal $C^*$-algebras generated by appropriately defined covariant representations of $(A,\morp)$. At the end of this section we discuss two important special cases when: 1) $\morp$ is multiplicative; 2)   $A$ is commutative.

\subsection{Crossed products}
Following  the original idea of Exel \cite{exel2}, we first define  a Toeplitz algebra, and then construct  crossed products by `eliminating redundancies' in the latter.

\begin{dfn}\label{representation of a cp map}
A \emph{representation of} $(A,\morp)$  is   a pair $(\pi,S)$ consisting of a non-degenerate representation $\pi:A\to \B(H)$ and an   operator $S\in \B(H)$  such that
\begin{equation}\label{representation of morp}
  S^*\pi(a) S=\pi(\morp(a))\qquad  \textrm{ for all }a\in A.
 \end{equation}
 If $\pi$ is faithful we call $(\pi,S)$ \emph{faithful}.
 We denote by $C^*(\pi,S)$ the $C^*$-algebra generated by $\pi(A)\cup\pi(A)S$.  We define the \emph{Toeplitz algebra of $(A,\morp)$} to be the $C^*$-algebra $\TT(A,\morp):=C^*(i_A(A),t)$ where $(i_A,t)$ is the universal representation  of $(A,\morp)$. Hence    for any  representation $(\pi,S)$ of $(A,\morp)$ the assignments
\begin{equation}\label{toeplitz epimorphism}
 i_A(a)\longmapsto \pi(a), \qquad i_A(a)t \longmapsto \pi(a)S, \qquad a\in A,
 \end{equation}
 define the  epimorphism from $\TT(A,\morp)$ onto $C^*(\pi,S)$.
  \end{dfn}
	\begin{lem}\label{existence resistance}
	Any $C^*$-dynamical system $(A,\morp)$ admits a universal representation and hence the Toeplitz algebra $\TT(A,\morp)$ exists.
	 \end{lem}
	\begin{proof} Let  $(\pi,S)$ be a representation of $(A,\morp)$ and let $\{\mu_\lambda\}_{\lambda\in \Lambda}$ be an approximate unit  in $A$. Using  non-degeneracy of $\pi$ and  relation \eqref{representation of morp} we get
		$$
		\|S\|^2=\lim_{\lambda\in \Lambda}\|\pi(\mu_\lambda)S\|^2=\lim_{\lambda\in \Lambda}\|S^*\pi(\mu_\lambda^2)S\|=\lim_{\lambda\in \Lambda}\|\pi(\morp(\mu_\lambda^2))\|\leq \|\morp\|.
		$$
Now let  $\{(\pi_i,S_i)\}_{i\in I}$ be a set of representation, where  $\pi_i:A\to \B(H_i)$ and $S_i\in \B(H_i)$. The above inequality implies that the direct product $\prod_{i\in I}S_i $ is an element of $\prod_{i\in I} B(H_i)$. In particular,  embedding $\prod_{i\in I} B(H_i)$ in a non-degenerate way into $B(H)$ for some Hilbert space $H$, we see that 	$(\prod_{i\in I}\pi_i, \prod_{i\in I} S_i)$ is a representation of $(A,\morp)$.
Thus the assertion follows by Proposition \ref{existence of universal stuff}.
	\end{proof}
To study the structure of $C^*(\pi,S)=C^*(\pi(A)\cup \pi(A)S)$  one needs to understand the relationship between   the following `monomials':
$$
\pi(a)S \pi(b)S\pi(c), \quad \pi(a)S^* \pi(b)S^*\pi(c),\quad\pi(a)S^* \pi(b)S\pi(c),\quad \pi(a)S \pi(b)S^*\pi(c).
$$
As we will see in the course of our analysis, the first two  behave like `simple tensors', and by \eqref{representation of morp}  the third one  is in $\pi(A)$.
Establishing the relationship with the fourth `monomial'  requires determining additional data which is encoded in  an  ideal that we are about to introduce. In the context of $C^*$-correspondences, this ideal is closely related with the one considered in \cite[Definition 5.8]{katsura}, and is  called the ideal of covariance in \cite[Definition 4.5]{KL}.
\begin{dfn}\label{redundancy definition}
By a \emph{redundancy} of a representation $(\pi,S)$ of $(A,\morp)$  we  mean a pair $(\pi(a),k)$ such that $a\in A$, $k\in \overline{\pi(A)S\pi(A)S^*\pi(A)}$ and
 $$
 \pi(a)\pi(b)S=k \pi(b)S \,\, \textrm{ for all }b \in A.
 $$
Let $J_{(\pi,S)}$ be the set of elements $a\in A$ such that $(\pi(a),k)$ is a redundancy of $(\pi,S)$ with   $\pi(a)=k$. Clearly, it is an ideal in $A$ and
\begin{align*}
 J_{(\pi,S)}=\{a\in A: \pi(a) \in \overline{\pi(A)S\pi(A)S^*\pi(A)}\}.
\end{align*}
 We call it the \emph{ideal of covariance} for $(\pi,S)$.
  \end{dfn}
	\begin{rem} Using \eqref{representation of morp},  we see that  $\overline{\pi(A)S\pi(A)S^*\pi(A)}$ is a $C^*$-algebra that acts on the  space $\pi(A)S$. Moreover, this action is faithful. Thus, if $(\pi(a),k)$ is a redundancy of $(\pi,S)$, then $k$ is uniquely determined by $a$, and we have $\pi(a)=k$ if and only if $a\in  J_{(\pi,S)}$.
	\end{rem}
 Let us consider the  GNS-kernel $N_\morp$ of $\morp$, see \eqref{kernel of dynamical cp system}, and  a representation $(\pi,S)$  of $(A,\morp)$. If $a\in N_\morp$ then the pair $(\pi(a),0)$ is necessarily a redundancy  because $\|\pi(a)\pi(b)S\|^2\leq\|\morp((ab)^*ab))\|=0$, for all $b\in A$. In particular,    $a\in J_{(\pi,S)}\cap N_\morp$ implies  $\pi(a)=0$. Accordingly, if  $(\pi,S)$ is faithful then  $J_{(\pi,S)}\subseteq N_\morp^\bot$. An argument of this sort stands behind Katsura's motivation for introducing the ideal \eqref{katsura's ideal}.  It explains the special role of the ideal $N_\morp^\bot$ in the following definition.

\begin{dfn}\label{definition main}
We define the \emph{crossed product} $C^*(A,\morp)$ of $A$ by $\morp$ to be the quotient  of the Toeplitz $C^*$-algebra $\TT(A,\morp)$ by the ideal generated by the set
$$
\{i_A(a)-k: a\in N_\morp^\bot \textrm{ and }(i_A(a),k) \textrm{ is a redundancy of } (i_A,t)\}.
$$
For any ideal $J$  in $A$  we define the \emph{relative crossed product} $C^*(A,\morp;J)$  to be the quotient  of the Toeplitz $C^*$-algebra $\TT(A,\morp)$ by the ideal  generated by the set
$$
\{i_A(a)-k: a\in J \textrm{ and }(i_A(a),k) \textrm{ is a redundancy of } (i_A,t)\}.
$$
We denote by $(j_A, s)$ the representation of $(A,\morp)$ that generates $C^*(A,\morp;J)$.
\end{dfn}
\subsection{Crossed products as relative Cuntz-Pimsner algebras}
 A  $C^*$-cor\-respon\-dence  associated to a completely positive map  was already considered by Paschke  \cite[section 5]{Paschke} and  sometimes is called the GNS or the KSGNS-correspondence (for Kasparov, Stinespring, Gelfand, Naimark, Segal), cf. \cite{Lan}, \cite{imv}.  Namely, we let
  $X_\morp$ to be a Hausdorff completion of the algebraic tensor product $A\odot A$ with respect to the seminorm associated to the $A$-valued sesquilinear  form  given by
\begin{equation}\label{inner product for rho}
 \langle a\odot b, c\odot d  \rangle_\morp:= b^* \morp(a^*c)d, \qquad a,b,c,d \in A.
\end{equation}
In the sequel we  use the symbol $a\otimes b$ to denote the image of the simple tensor $a\odot b$ in $X_\morp$. The space $X_\morp$ becomes a $C^*$-correspondence over $A$ with  the left and right actions determined by: $a\cdot (b\otimes c)= (ab)\otimes c$ and  $(b \otimes c) \cdot a= b\otimes  (ca)$ where  $a,b,c\in A$.
\begin{dfn}
We  call  $X_\morp$ defined above the  \emph{$C^*$-correspondence of} $(A,\morp)$.
\end{dfn}

\begin{rem} Clearly, the $C^*$-correspondence $X_\morp$ is essential. The GNS-kernel \eqref{kernel of dynamical cp system} of $\morp$ coincides with the kernel of the left action of $A$ on $X_\morp$.  Hence the left action of $A$ on $X_\morp$ is faithful if and only if $\morp$ is almost faithful on $A$.
\end{rem}
If $A$ is not unital we  give a meaning to the symbol  $a\otimes 1$, $a\in A$,  using the following lemma.
\begin{lem}\label{bread remark2}
Let $\{\mu_\lambda\}_{\lambda\in \Lambda}$ be an approximate unit  in $A$. Then, for any $a\in A$, the limit
\begin{equation}\label{tensor one mappings2}
a\otimes 1 :=\lim_{\lambda\in \Lambda} a \otimes \mu_\lambda
\end{equation}
 exists  and  defines a bounded linear map $A \ni a \mapsto a\otimes 1 \in X_\morp$ of norm $\|\morp\|^{\frac{1}{2}}$.
\end{lem}
\begin{proof}
Since $\|a \otimes (\mu_\lambda - \mu_{\lambda'})\|^2=\|(\mu_\lambda - \mu_{\lambda'})\morp (a^*a)(\mu_\lambda - \mu_{\lambda'})\|$ tends to zero,
 as $\lambda$ and $\lambda'$ tend to `infinity', the net $\{ a \otimes \mu_\lambda\}_{\lambda\in \Lambda}$ is Cauchy and hence convergent. Since
$$
\sup_{a\in A, \|a\|=1}\|a \otimes 1\|^2=\sup_{a\in A, \|a\|=1}\|\morp(a^*a)\|=\|\morp\|,
$$
we see that $A \ni a \mapsto a\otimes 1 \in X_\morp$ is a bounded operator of norm $\|\morp\|^{\frac{1}{2}}$.
\end{proof}
\begin{rem}\label{bread remark}  Let $\{\mu_\lambda\}_{\lambda\in \Lambda}$ be an approximate unit  in $A$. The limit
\begin{equation}\label{tensor one mappings}
 1\otimes a :=\lim_{\lambda\in \Lambda} \mu_\lambda \otimes a, \qquad a \in A,
\end{equation}
in general  may not exist. However,  if $\morp$ is strict then it does exist and the  map  $A \ni a \mapsto 1\otimes a \in X_\morp$ is   linear bounded, again of norm $\|\morp\|^{\frac{1}{2}}$, see \cite[p. 50]{Lan}.
\end{rem}
The mapping in the latter remark, which exists when $\morp$ is strict, plays a key  role  in the construction of KSGNS-dilation of $\morp$, cf. \cite[Theorem 5.6]{Lan}. We adjust this construction to get a description of representations of the $C^*$-correspondence $X_\morp$ for arbitrary $\morp$.

\begin{prop}\label{characterization of representations of cp map} Let $X_\morp$  be the $C^*$-correspondence of $(A,\morp)$. We have a one-to-one correspondence between   representations $(\pi,S)$ of  $(A,\morp)$ and   representations $(\pi,\pi_{X_\morp})$ of ${X_\morp}$ where
\begin{equation}\label{associated correspondence representation}
\pi_{X_\morp}(a\otimes b)= \pi(a)S\pi(b),\qquad  a\otimes b \in {X_\morp},
\end{equation}
\begin{equation}\label{definition of partial isometry for partial morphisms}
S^*\pi_{X_\morp}(a\otimes b )=\pi(\morp(a)b),\quad  a,b\in A,  \quad S^*|_{(\pi_{X_\morp}({X_\morp})H)^{\bot}}\equiv 0.
\end{equation}
For the corresponding representations, we have $C^*(\pi,S)=C^*(\pi(A)\cup\pi_{X_\morp}({X_\morp}))$
and  for any approximate unit $\{\mu_\lambda\}_{\lambda\in \Lambda}$ in $A$:
\begin{equation}\label{conjugate by limit}
S=\textrm{s-}\lim_{\lambda\in \Lambda} \pi_{X_\morp}(\mu_\lambda\otimes 1)
\end{equation}
where the limit is taken in the strong operator topology.  If $\morp$ is strict,  the limit in \eqref{conjugate by limit} is strictly convergent in $M(C^*(\pi,S))$, and the  multiplier $S\in M(C^*(\pi,S))$ is determined by the formula $S\pi(a)=\pi_{X_\morp}(1\otimes a)$, $a\in A$, cf. Remark \ref{bread remark}.
\end{prop}
\begin{proof} Let  $(\pi,S)$ be a  representation of $(A,\morp)$ on $H$. The following computation
\begin{align*}
 \Big(\sum_{i}\pi(a_i)S\pi(b_i)\Big)^*  \sum_j\pi(c_j)S\pi(d_j)&=\sum_{i,j}\pi(b_i^*) S^*\pi(a_i^*c_j)S\pi(d_j)
 \\
 &=\sum_{i,j}\pi(b_i^*\morp(a_i^*c_j)d_j)
 \\ &
 =\pi(\langle\sum_{i}  a_i\otimes b_i, \sum_{j} c_j\otimes d_j\rangle_\morp)
\end{align*}
implies that the mapping \eqref{associated correspondence representation}   extends to a linear contractive map $\pi_{X_\morp}:{X_\morp}\to B(H)$. It is evident that $(\pi,\pi_{X_\morp})$  is a representation of ${X_\morp}$. We have  $
S^*\pi_{X_\morp}(a\otimes b )=S^*\pi(a) S \pi(b)=\pi(\morp(a)b)
$, for any $a,b\in A$. Since $\pi$ is non-degenerate, the range   of
 $S$ is contained in
$
\pi(A)SH=\pi(A)S\pi(A)H\subseteq \pi_{X_\morp}({X_\morp})H
$. Hence   $S^*|_{(\pi_{X_\morp}({X_\morp})H)^{\bot}}\equiv 0$ and   \eqref{definition of partial isometry for partial morphisms} holds.  Furthermore, note that for any approximate unit $\{\mu_\lambda\}_{\lambda\in \Lambda}$   in $A$, the net $\{\pi(\mu_\lambda)\}_{\lambda\in \Lambda}$ converges strongly to the identity in $B(H)$. Thus by Lemma \ref{bread remark2} and \eqref{associated correspondence representation}, we have  $\pi_{X_\morp}(a\otimes 1)= \lim_{\lambda \in \Lambda} \pi(a)S\pi(\mu_\lambda)=\pi(a)S$, for any $a\in A$. Therefore
$$
\pi(A)S=\pi_{X_\morp}(A\otimes 1) \subseteq \pi_{X_\morp}({X_\morp})=\pi(A)S\pi(A).
$$
Hence $C^*(\pi,S)=C^*(\pi(A)\cup\pi(A)S)=C^*(\pi(A)\cup\pi(A)S\pi(A))=C^*(\pi(A)\cup\pi_{X_\morp}({X_\morp}))$.

Suppose now that $(\pi,\pi_{X_\morp})$ is a  representation  of ${X_\morp}$. We need to show that there exists an operator $S\in B(H)$ such that  $(\pi,S)$ is a representation of $(A,\morp)$ satisfying  \eqref{associated correspondence representation}.  Let $\{\mu_\lambda\}_{\lambda\in \Lambda}$ be an approximate unit
  in $A$ and consider the net of bounded operators $S_\lambda:=\pi_{X_\morp}(\mu_\lambda \otimes 1)$,  $\lambda\in \Lambda$. Note that
	$S_\lambda\pi(a)=\pi_{X_\morp}(\mu_\lambda\otimes a)$, $a\in A$.    To see that \eqref{conjugate by limit} determines a bounded operator it suffices to show that the net $\{S_\lambda\}_{\lambda\in \Lambda}$ is strongly Cauchy. To this end, let $a\in A$, $h\in H$ and  $\lambda\leq \lambda'$, in the directed set $\Lambda$. Then
 \begin{align*}
\|(S_\lambda-S_{\lambda'})\pi(a)h\|^2 &=\|\pi_{X_\morp}((\mu_\lambda-\mu_{\lambda'})\otimes a)h\|^2
\\
&=\langle h,  \pi( \langle (\mu_\lambda-\mu_{\lambda'})\otimes a), (\mu_\lambda-\mu_{\lambda'})\otimes a\rangle_\morp)h\rangle
\\
&=\langle h, \pi(a^*\morp ((\mu_\lambda-\mu_{\lambda'})^2) a)h \rangle
\\
&\leq \langle h, \pi(a^*\morp (\mu_\lambda-\mu_{\lambda'}) a) h \rangle
\\
&=\langle \pi(a) h, \pi(\morp (\mu_\lambda-\mu_{\lambda'}))   \pi(a)h \rangle.
\end{align*}
Since the net $\{\pi(\morp (\mu_\lambda))\}_{\lambda\in \Lambda}$ is strongly convergent  the last expression tends to zero.
Hence $S:=\textrm{s-}\lim_{\lambda \in \Lambda} S_\lambda$ defines a bounded operator. Let $a, b \in A$. As
$\|(a- a\mu_\lambda)\otimes b\|^2=\|b^*\morp((a^*-\mu_\lambda a^*)(a- a\mu_\lambda))b\|$ tends to zero, we  get that $
a\otimes b=  \lim_{\lambda \in \Lambda} a\mu_\lambda \otimes b
$   in ${X_\morp}$.  Thus
$$
\pi_{X_\morp}(a\otimes b)=  \lim_{\lambda \in \Lambda} \pi_{X_\morp}(a\mu_\lambda \otimes b)=\pi(a)\lim_{\lambda \in \Lambda} S_\lambda  \pi(b)=\pi(a) S  \pi(b),
$$
that is  \eqref{associated correspondence representation} holds.  Moreover, for any $a,b\in A$ and $h,f \in H$ we have
\begin{align*}
\langle \pi_{X_\morp}(a\otimes b)h, S f\rangle &=\lim_{\lambda \in \Lambda}\langle \pi_{X_\morp}(a\otimes b)h, \pi_{X_\morp}(\mu_\lambda \otimes 1) f\rangle
\\
&=\lim_{\lambda \in \Lambda} \langle \pi(\morp(\mu_\lambda a ) b)h, f\rangle =\langle \pi(\morp(a ) b)h, f\rangle.
\end{align*}
Hence  $S^*\pi_{X_\morp}(a\otimes b )=\pi(\morp(a ) b)$ and therefore
$
S^* \pi(a) S \pi (b)= S^*\pi_{X_\morp}(a\otimes b )=\pi(\morp(a )) \pi(b).
$
Since $\pi$ is non-degenerate this implies \eqref{representation of morp}.
\\
Suppose now that $\morp$ is strict.  Then, in view of  Remark \ref{bread remark}, for any $a\in A$ the following limit exists
$$
\lim_{\lambda \in \Lambda} S_\lambda \pi(a)=\lim_{\lambda \in \Lambda} \pi_{X_\morp}(\mu_\lambda \otimes a)= \pi_{X_\morp}(1 \otimes a).
$$
Similarly, we get $\lim_{\lambda \in \Lambda} \pi(a) S_\lambda=\pi_{X_\morp}(a\otimes 1)$. Thus, since $\pi(A)$ is  non-degenerate
in $C^*(\pi,S)=C^*(\pi(A)\cup \pi(A)S\pi(A))$,  the limit in \eqref{conjugate by limit} is strictly convergent.
\end{proof}
\begin{rem}\label{remark on strictness} Let $(j_A,s)$ be the representation of $(A,\morp)$ that generates the relative crossed product $C^*(A,\morp;J)$. By the above proposition $s$ is  an element of the enveloping von Neumann algebra $C^*(A,\morp;J)^{**}$ and if  $\morp$ is strict then actually $s\in M(C^*(A,\morp;J))$.
\end{rem}
The following lemma is a translation of \cite[Proposition 3.3]{katsura} to our setting, cf.   also \cite[Lemma 3.5]{br}. It implies that  when considering the relative crossed products   it suffices to  restrict attention to ideals $J$ contained in the ideal $J(X_\morp)=\phi^{-1}(\K(X_\morp))$. It will also lead us to the main result of this subsection.

\begin{lem}\label{lemma of raeburn}
Let $(\pi,S)$ be a  faithful representation  of $(A,\morp)$ and let $(\pi,\pi_{X_\morp})$ be the  representation of  $X_\morp$ with $\pi_{X_\morp}$ given by \eqref{associated correspondence representation}. A pair $(\pi(a),k)$ is a redundancy of $(\pi,S)$ if and only if  $a\in J({X_\morp})$ and $k=(\pi,\pi_{X_\morp})^{(1)}(\phi(a))$.
\end{lem}

\begin{proof}  Note that  $(\pi,\pi_{X_\morp})^{(1)}:\K({X_\morp})\to \overline{\pi(A)S\pi(A)S^*\pi(A)}$.
Thus if  $a\in J({X_\morp})$ and $k=(\pi,\pi_{X_\morp})^{(1)}(\phi(a))$ then $k\in \overline{\pi(A)S\pi(A)S^*\pi(A)}$. Moreover,    for any $b,c\in A$ we have
\begin{align*}
 \pi(a)\pi(b)S \pi(c)&= \pi(a)\pi_{X_\morp}(b\otimes c)=\pi_{X_\morp}(\phi(a)b\otimes c)= (\pi,\pi_{X_\morp})^{(1)}(\phi(a)) \pi_{X_\morp}(b\otimes c)
\\
&=(\pi,\pi_{X_\morp})^{(1)}(\phi(a))\pi(b)S \pi(c).
\end{align*}
Hence by non-degeneracy of $\pi$, we see  that $(\pi(a),k)$ is a redundancy.
\\
Now let $(\pi(a),k)$ be any redundancy. Then $k=(\pi,\pi_{X_\morp})^{(1)}(t)$ for a certain $t\in \K({X_\morp})$, and for any $b,c\in A$ we have
\begin{align*}
 \pi_{X_\morp}(\phi(a)b\otimes c)&=\pi(a)\pi_{X_\morp}(b\otimes c)=\pi(a)\pi(b) S \pi(c)=k\pi(b) S \pi(c)
\\
&=(\pi,\pi_{X_\morp})^{(1)}(t)\pi_{X_\morp}(b\otimes c) =\pi_{X_\morp}(t (b\otimes c)).
\end{align*}
Since faithfulness of $\pi$ implies injectivity of  $\pi_{X_\morp}$, we get   $\phi(a)b\otimes c=t b\otimes c $, for all $b,c\in A$. Consequently, $\phi(a)=t$ as desired.
\end{proof}
\begin{thm}\label{identification theorem}
Let $X_\morp$ be the $C^*$-correspondence of $(A,\morp)$. For any ideal $J$ in $A$  we have
$$
C^*(A,\morp;J)=C^*(A,\morp;J\cap J({X_\morp}))\cong \OO(J\cap J({X_\morp}), {X_\morp}).
$$
The universal homomorphism $j_A:A \to C^*(A,\morp;J)$ is injective if and only if $\morp$ is almost faithful on $J\cap J({X_\morp})$, that is if and only if $J\cap J({X_\morp})\subseteq N_\morp^\bot$. In particular,
$$
C^*(A,\morp)\cong \OO_{X_\morp}
$$
and $j_A:A\to C^*(A,\morp)$  is injective.
\end{thm}
\begin{proof}
 By   Proposition \ref{characterization of representations of cp map} Toeplitz algebras $\TT(A,\morp)$ and $\TT({X_\morp})$  are naturally isomorphic and identifying them explicitly we may assume that the universal representation $(i_A,i_{X_\morp})$ of ${X_\morp}$ in $\TT({X_\morp})$ satisfies  $i_{X_\morp}(a\otimes b)=i_A(a) t \, i_A(b)$ for $a,b\in A$. Then $\TT(A,\morp)=\TT({X_\morp})$ and by Lemma \ref{lemma of raeburn}
\begin{align}
&\{i_A(a)-k: a\in J \textrm{ and }(i_A(a),k) \textrm{ is a redundancy of } (i_A,t)\}\nonumber
\\
& =\{i_A(a)-k: a\in J\cap J({X_\morp}) \textrm{ and }(i_A(a),k) \textrm{ is a redundancy of } (i_A,t)\}\label{equation 15}
\\
&=\{i_A(a)- (i_A,i_{X_\morp})^{(1)}(\phi(a)): a\in J\cap J({X_\morp})\}. \nonumber
\end{align}
Hence the three algebras $C^*(A,\morp;J)$,  $C^*(A,\morp;J\cap J({X_\morp}))$ and $\OO(J\cap J({X_\morp}),{X_\morp})$ arise as quotients of the same algebra by the same ideal, cf. Remark \ref{remark on relative Cuntz-Pimsners}. Thus they are isomorphic. For the remaining  part of the assertion  apply Proposition \ref{proposition for the referee} and the fact that $(\ker\phi)^\bot = N_\morp^\bot$.
\end{proof}

\begin{rem}\label{Schweizer's crossed product}
In \cite[Subsection 3.3]{Schweizer}, a crossed product by a completely positive map $\morp$ was defined   as Pimsner's (augmented) $C^*$-algebra associated to the $C^*$-correspondence $X_\morp$. Thus Schweizer's crossed product  is isomorphic to the relative crossed product $C^*(A,\morp;A)=C^*(A,\morp;J({X_\morp}))$.
\end{rem}
\begin{rem}\label{remark on interactions}
Using the notion of multiplicative domain, an \emph{interaction} $(\VV,\HH)$ on a $C^*$-algebra $A$ introduced by Exel in \cite[Definition 3.1]{exel-inter} can be defined as a pair of positive maps $\VV, \HH: A\to A$ such that
$$
\VV\HH\VV=\VV, \quad \HH \VV \HH=\HH,\quad \VV(A)\subseteq MD(\HH) ,\quad \HH(A)\subseteq MD(\VV).
$$
Then $\VV$ and $\HH$ are automatically contractive completely positive maps, see \cite[Corollary 3.3]{exel-inter}.
In \cite{kwa-interact} the author considered an interaction $(\VV, \HH)$ on a unital $C^*$-algebra $A$ such that the ranges $\VV(A)$, $\HH(A)$ are corners in $A$.
The crossed product $C^*(A,\VV, \HH)$ defined in \cite[Definition 2.10]{kwa-interact} is a universal $C^*$-algebra generated by a copy of $A$ and an operator $s$ subject to relations
$$
\VV(a)=sas^* \quad \textrm{ and  } \quad \HH(a)=s^*a s\quad \textrm{ for all } a\in A.
$$
By \cite[Corollary 2.16]{kwa-interact},  $C^*(A,\VV, \HH)$ coincides with the covariance algebra associated to $(\VV, \HH)$ in \cite{exel-inter}. The $C^*$-correspondences $X_\VV$ and $X_\HH$ are (mutually opposite) Hilbert bimodules, see \cite[Lemma 2.11]{kwa-interact}. Hence by \cite[Proposition 2.14]{kwa-interact} and  \cite[Proposition 3.7]{katsura1},   $C^*(A,\VV, \HH)$ is isomorphic to both $\OO_{X_\VV}$ and $\OO_{X_\HH}$. In particular, by Theorem \ref{identification theorem}, we  get
\begin{equation}\label{interacion isomorphisms}
C^*(A,\VV, \HH)\cong  C^*(A,\VV)\cong C^*(A,\HH),
\end{equation}
where $C^*(A,\VV)$ and $C^*(A,\HH)$ are crossed products in the sense of Definition \ref{definition main}.
\end{rem}

\subsection{Universal description and gauge-invariant uniqueness theorem}
We  describe the $C^*$-algebra  $C^*(A,\morp;J)$ as a  universal object in the following way.

\begin{dfn}\label{definition of covariance}
Let $J$ be an ideal in $A$. We  say that a representation  $(\pi,S)$ of $(A,\morp)$ is \emph{$J$-covariant} if $J\cap J(X_\morp)\subseteq J_{(\pi,S)}$ or just  \emph{covariant} if $N_\morp^\bot\cap J(X_\morp)\subseteq J_{(\pi,S)}$.
\end{dfn}

\begin{prop}\label{universal description of crossed products}
Let $J$ be an ideal in $A$. The  crossed product $C^*(A,\morp;J)$ is  universal with respect to  $J$-covariant representations of $(A,\morp)$, that is  $(j_A, s)$ is   $J$-covariant and for every  $J$-covariant representation $(\pi,S)$ of $(A,\morp)$  the mapping
\begin{equation}\label{epimorphism induced by covariant rep}
j_A(a) \longmapsto \pi(a),\qquad j_A(a)s \longmapsto  \pi(a)S, \qquad a\in A,
\end{equation}
extends to a (necessarily unique)  epimorphism $\pi \rtimes_{J} S: C^*(A,\morp;J) \to C^*(\pi,S)$.
\end{prop}
\begin{proof}
That $(j_A, s)$ is $J$-covariant  follows from  the equality \eqref{equation 15} and the definition of $C^*(A,\morp;J)$.  Let $(\pi,S)$ be a  $J$-covariant representation  of $(A,\morp)$, and let  $\pi \rtimes_{\{0\}} S:\TT(A,\morp)\to C^*(\pi,S)$ be the  epimorphism determined by  \eqref{toeplitz epimorphism}. Clearly, $\pi \rtimes_{\{0\}} S$ maps redundancies of $(i_A,i_{X_\morp})$ onto redundancies of $(\pi,S)$. Thus in view of     \eqref{equation 15}, using $J\cap J(X_\morp)\subseteq J_{(\pi,S)}$, we conclude that $\pi \rtimes_{\{0\}} S$  factors through to the desired epimorphism  $\pi \rtimes_{J} S: C^*(A,\morp;J) \to C^*(\pi,S)$.  \end{proof}
Using Katsura's gauge-invariant uniqueness  theorem  for  Cuntz-Pimsner algebras \cite[Theorem 6.4]{katsura} we get the following version of this standard tool for $C^*(A,\morp)$. One could  formulate a more general result for crossed products $C^*(A,\morp;J)$, cf. for instance \cite[Theorem 9.1]{kwa-doplicher}, but we will not need it here.
\begin{prop}\label{Gauge invariant uniqueness theorem} Let $(\pi,S)$ be a covariant representation of $(A,\morp)$. The epimorphism  given by \eqref{epimorphism induced by covariant rep} is an isomorphism:
$$
C^*(\pi, S)\cong C^*(A,\morp)
$$
if and only if $(\pi,S)$ is faithful and there exists  a strongly continuous action $\beta:\T\to \Aut (C^*(\pi, S))$ such that $\beta_z(\pi(a))=\pi(a)$ and  $\beta_z(\pi(a)S)= z\pi(a)S$ for all $a\in A$ and $z\in \T$.
\end{prop}
\begin{proof}
If $(\pi,S)$ is a faithful covariant representation of $(A,\morp)$ then by Lemma \ref{lemma of raeburn} the corresponding representation $(\pi,\pi_{X_\morp})$ of $X_{\morp}$ is covariant in the sense of \cite[Definition 3.4]{katsura}. Since $C^*(A,\morp)\cong \OO_{X_\morp}$, by Theorem \ref{identification theorem}, it suffices to apply  \cite[Theorem 6.4]{katsura}.
\end{proof}

\subsection{The case when $\morp$ is multiplicative} \label{The case when multiplicative}
In this section, we assume  that $\morp$ is multiplicative and we denote it by $\alpha$. In other words, we assume that  $\alpha:A\to A$ is an endomorphism.
We show that  our relative crossed products $C^*(A,\alpha;J)$ coincide with various crossed products by endomorphisms appearing in the literature. The latter are typically studied in the case where $\alpha$ is extendible. We warn the reader that  representations of endomorphisms  are considered in a different convention than we adopted in Definition \ref{representation of a cp map}. For the sake of discussion we include the following definition.

\begin{dfn}\label{crossed product defn} Let $\alpha:A\to A$ be an extendible endomorphism and let $J$ be an ideal in $A$. We say that $(\pi,U)$ is a  \emph{representation of the endomorphism} $\alpha$ if  $(\pi,S)$, where $S=U^*$, is a representation of $(A,\alpha)$ in the sense of Definition \ref{representation of a cp map}. Thus we assume that $(\pi,U)$ consists of a non-degenerate representation $\pi:A\to \B(H)$ and an   operator $U\in \B(H)$  such that
\begin{equation}\label{representation relation for U}
  U\pi(a) U^*=\pi(\alpha(a))\qquad  \textrm{ for all }a\in A.
 \end{equation}
Further under these assumptions:
\begin{itemize}
\item[i)] We say that $(\pi,U)$ is a \emph{$J$-covariant representation of} $\alpha$ if
$$
J\subseteq \{a\in A: U^*U\pi(a)=\pi(a)\}.
$$
We  put $C^*_{endo}(A,\alpha;J):=C^*(i_A(A)\cup i_A(A)u)$ where  $(i_A,u)$ is the universal $J$-covariant representation of the endomorphism $\alpha$. We also put  $C^*_{endo}(A,\alpha):=C^*_{endo}(A,\alpha;(\ker\alpha)^\bot)$.

\item[ii)] We say that  $(\pi,U)$ is \emph{isometric} if $U$ is an isometry. The \emph{Stacey's crossed product} is $A\rtimes^{1}_\alpha\N:= C^*(\{i_A(a)u^n u^{*m}: a\in A, \,\, n,m\in \N\})$  where  $(i_A,u)$ is the universal isometric representation of $\alpha$.
\item[iii)] We say that  $(\pi,U)$ is \emph{partial-isometric} if
\begin{equation}\label{superfluous condition}
U \textrm{ is a partial isometry and }U^*U \textrm{ belongs to the commutant of } A.
\end{equation}
The \emph{partial-isometric crossed product} is $A\rtimes_\alpha \N:=C^*(i_A(A)\cup ui_A(A))$ where $(i_A,u)$ is the universal partial-isometric representation of $\alpha$.
\end{itemize}
\end{dfn}
\begin{rem}\label{dupa remark}
The relation \eqref{representation relation for U} implies that $U$ is necessary a partial isometry (cf. Proposition \ref{Proposition for endomorphisms0} below). All of the authors mentioned below assumed that the universal operator $u$ belongs to the multiplier algebra of the corresponding crossed product. We did not require that explicitly in Definition \ref{crossed product defn}   but it  follows from the axioms (see the second part of Proposition \ref{Proposition for endomorphisms}). Moreover, we have the following comments:

i). The crossed product $C^*_{endo}(A,\alpha;J)$ was introduced in  \cite[Definition 1.12]{KL} in the case $A$ is unital. It was generalized to the non-unital case in  \cite[Definition 4.8]{kwa-rever}. These papers deal only with the ideals $J$ contained in $(\ker\alpha)^\bot$. However, as explained, for instance in \cite[Remark 4.3]{kwa-rever} or \cite[Subsection 5.3]{KL}, if $J\subsetneq (\ker\alpha)^\bot$, then  there is a  canonical quotient system $(A/R,\alpha_R)$ such that the image $q_R(J)$ of $J$ in $A/R$ is contained in $(\ker\alpha_R)^\bot$ and  $C^*_{endo}(A,\alpha;J)\cong C^*_{endo}(A/R,\alpha_R;q_R(J))$.

ii). The crossed product $A\rtimes^{1}_\alpha\N$ was introduced in \cite[Definition 3.1]{Stacey} as a crossed product of multiplicity $1$. The author of \cite{Stacey} did not assume explicitly that $\alpha$ is extendible but he uses extendibility of $\alpha$ in his arguments. We note that a representation $(\pi,U)$ of  $\alpha$ is isometric if and only if it is $A$-covariant.

iii).  The crossed product $A\rtimes_\alpha \N$ was defined in \cite[p. 73]{Lin-Rae} (in a semigroup context) essentially as Fowler's Toeplitz crossed product \cite[p. 344]{fowler}, cf. \cite[Proposition 3.4]{fowler} or \cite[Proposition 4.7]{Lin-Rae}.
\end{rem}

The fact that the following covariance relation \eqref{additional relation for endomorphisms}  is automatic went unnoticed in a few papers  preceding \cite{KL}, cf. \cite[Remark 1.3]{KL}.
\begin{prop}\label{Proposition for endomorphisms0}
Let  $\alpha:A\to A$ be an endomorphism and let   $(\pi, S)$ be a  representation of $(A,\alpha)$, in the sense of Definition \ref{representation of a cp map}. Then we automatically have that $S$ is a partial isometry and
\begin{equation}\label{additional relation for endomorphisms}
\pi(a) S=S \pi(\alpha(a)),\qquad \textrm{ for all }a\in A.
\end{equation}
In particular, the projection $SS^*$ belongs to the commutant of $\pi(A)$, and the ideal of covariance for $(\pi, S)$, in the sense of Definition \ref{redundancy definition}, is given by the formula
\begin{equation}\label{ideal of covariance}
 J_{(\pi,S)}=\{a\in A: SS^*\pi(a)=\pi(a)\}.
\end{equation}
\end{prop}
\begin{proof}
Let $\{\mu_\lambda\}_{\lambda \in \Lambda}$ be an approximate unit in $A$. Multiplicativity of $\alpha$ implies that the increasing net $\{\pi(\alpha(\mu_\lambda))\}_{\lambda \in \Lambda}$ converges strongly to a projection in $\B(H)$.  By \eqref{representation of morp}, and non-degeneracy of $\pi$, we have  $s\textrm{-}\lim_{\lambda \in \Lambda}\pi(\alpha(\mu_\lambda) )=s\textrm{-}\lim_{\lambda \in \Lambda} S^*\pi(\mu_\lambda)S=SS^*$. Hence $S$ is a partial isometry.
Following the argument in the proof of \cite[Lemma 1.2]{KL}, with $U=S^*$, we get  the commutation relation \eqref{additional relation for endomorphisms}.

Now,  by \eqref{additional relation for endomorphisms} and its adjoint version  ($S^*\pi(a)=\pi(\alpha(a))S^*$, $a\in A$)  one gets
$SS^*\pi(a)=S\pi(\alpha(a))S^*=\pi(a)SS^*$, for $a\in A$. Hence $SS^*$ belongs to the commutant of $\pi(A)$.

By \eqref{additional relation for endomorphisms}   we have $\overline{\pi(A)S\pi(A)S^*\pi(A)}=S\pi(\alpha(A)A\alpha(A))S^*$. Since $S$ is a partial isometry, this implies that
$J_{(\pi,S)}\subseteq \{a\in A: SS^*\pi(a)=\pi(a)\}$.  Conversely, for any $a\in A$ such that  $SS^*\pi(a)=\pi(a)$, again by \eqref{additional relation for endomorphisms}, we have
$$
\pi(a)=SS^*\pi(a)=S\pi(\alpha(a))S^*\in S\pi(\alpha(A)A\alpha(A))S^*=\overline{\pi(A)S\pi(A)S^*\pi(A)}.
$$
This proves \eqref{ideal of covariance}.
\end{proof}
\begin{rem}
If $A$ is unital and $(\pi,U)$ is a representation of  $\alpha$ (as in Definition \ref{crossed product defn}), then the set $I_{(\pi,U)}:=\{a\in A: U^*U\pi(a)=\pi(a)\}$ was called in \cite[Definition 1.7]{KL} the ideal of covariance for $(\pi,U)$. In view of  \eqref{ideal of covariance} we have $I_{(\pi,U)}=J_{(\pi,S)}$ where $S=U^*$.
\end{rem}
The following lemma implies that in the definition of crossed products by extendible endomorphisms we can put the generating operator either on the left or on the right of the generating algebra (the outcome will be the same).
\begin{lem}\label{dupa lemma}
If $\alpha:A\to A$ is  an extendible endomorphism and  $(\pi, S)$ is representation of $(A,\alpha)$, then
$$
C^*(\pi(A)\cup\pi(A)S)=C^*(\pi(A)\cup S\pi(A))=C^*(\{\pi(a)S^{*n} S^{m}: a\in A, \,\, n,m\in \N\}).
$$
\end{lem}
\begin{proof}
By  \eqref{additional relation for endomorphisms} we have $\pi(A)S= S\pi(\alpha(A))\subseteq S\pi(A)$. Thus we get
$$
C^*(\pi(A)\cup\pi(A)S)\subseteq C^*(\pi(A)\cup S\pi(A))\subseteq C^*(\{\pi(a)S^{*n} S^{m}: a\in A, \,\, n,m\in \N\}).
$$
We show that the rightmost algebra is contained in the leftmost one. Recall that for any approximate unit $\{\mu_\lambda\}_{\lambda \in \Lambda}$   in $A$ the net $\{\pi(\alpha(\mu_\lambda))\}_{\lambda \in \Lambda}$ converges strongly to $S^*S$, cf. the proof Proposition \ref{Proposition for endomorphisms0}.  Moreover, since $\alpha$ is extendible, for any $a\in A$, the net $\{a\alpha(\mu_\lambda)\}_{\lambda \in \Lambda}$ converges in norm (to  $a\overline{\alpha}(1)$). Thus for any $a\in A$, using the adjoint of \eqref{additional relation for endomorphisms}, we get
$$
\pi(a)S^*=\pi(a)S^*SS^*=\lim_{\lambda \in \Lambda} \pi(a\alpha(\mu_\lambda))S^*=\lim_{\lambda \in \Lambda}  \pi(a)S^*\pi(\mu_\lambda)\in  \pi(A)S^*\pi(A)
$$
($\pi(A)S^*\pi(A)$ is a closed space by the  Cohen-Hewitt Factorization Theorem).
Hence $\pi(A)S^*\subseteq \pi(A)S^*\pi(A)$. By \eqref{additional relation for endomorphisms}, we also have $\pi(A)S\subseteq S\pi(A)$. The last two  inclusions, used inductively, give us that
$$
\pi(A)S^{*n}S^{m}\subseteq\underbrace{(\pi(A)S^*) ... (\pi(A)S^*)}_{n \textrm{ times}}\underbrace{(\pi(A)S) ... ( \pi(A)S)}_{m \textrm{ times}}.
$$
Hence $
 C^*(\{\pi(a)S^{*n} S^{m}: a\in A, \,\, n,m\in \N\}) \subseteq C^*(\pi(A)\cup\pi(A)S).$
\end{proof}
\begin{cor}\label{corollary for browary} For any extendible endomorphism  $\alpha:A\to A$ we have that
$
A\rtimes^{1}_\alpha\N=C^*_{endo}(A,\alpha;A)$ and $A\rtimes_\alpha\N= C^*_{endo}(A,\alpha;\{0\})$.
\end{cor}
\begin{proof} Since a representation $(\pi,U)$ of $\alpha$ is isometric if and only if it is $A$-covariant, we may identify the corresponding universal representations. Then applying  Lemma \ref{dupa lemma} to $(i_A,u^*)$ we get
\begin{align*}
 A\rtimes^{1}_\alpha\N&=C^*(\{i_A(a)u^n u^{*m}: a\in A, \,\, n,m\in \N\})=C^*(i_A(A)\cup i_A(A)u)
 \\
 &=C^*_{endo}(A,\alpha;A).
\end{align*}
Similarly, using  Proposition \ref{Proposition for endomorphisms0} we see that partial-isometric representations of $\alpha$ coincide with $\{0\}$-covariant representations of $\alpha$ (which are simply representations of $\alpha$). Hence identifying the corresponding universal representations and applying  Lemma \ref{dupa lemma} to $(i_A,u^*)$ we get
$
A\rtimes_\alpha\N=C^*(i_A(A)\cup ui_A(A))=C^*(i_A(A)\cup i_A(A)u)=C^*_{endo}(A,\alpha;\{0\}).
$
\end{proof}

The crossed products $C^*_{endo}(A,\alpha;J)$ can be realized as relative Cuntz-Pimsner algebras associated to a certain $C^*$-correspondence $E_\alpha$ associated to $\alpha$ (this fact is  extensively discussed in \cite{KL} in the case when $A$ is unital). The $C^*$-correspondence in question was already considered by Pimsner \cite{p} and it is defined by the formulas
$$
E_\alpha:=\alpha(A)A,  \quad \langle x, y\rangle_A :=x^*y, \quad
a\cdot x\cdot b:=\alpha(a)xb, \quad x,y \in \alpha(A)A,\,\, a,b \in A.
$$
Pimsner's $C^*$-correspondence $E_\alpha$ and KSGNS $C^*$-correspondence $X_\alpha$ are naturally isomorphic.
\begin{lem}\label{lemma on the correspondence of endomorphism} If $\alpha:A\to A$ is an endomorphism, then the map $X_\alpha \ni a\otimes b\to \alpha(a)b\in E_\alpha$ determines an  isomorphism of $C^*$-correspondences $X_\alpha\cong E_\alpha$.
In particular,  we have $J(X_\alpha)=J(E_\alpha)=A$.
\end{lem}
\begin{proof} We leave it to the reader, as an easy exercise, to check that the prescribed map determines the desired isomorphism. For every $a \in A$ and $x\in E$ we can write $a=a_1a_2$,  where $a_1,a_2\in A$, and then $a\cdot x= \Theta_{\alpha(a_1),\alpha(a_2^*)} x$. Thus $J(E_\alpha)=A$.
\end{proof}
Now we are ready to state the main result of this subsection.
\begin{prop}\label{Proposition for endomorphisms}
Let  $\alpha:A\to A$ be an endomorphism and let  $J$ be an ideal in $A$. A  representation $(\pi, S)$ of $(A,\alpha)$ is $J$-covariant (in the sense of Definition \ref{definition of covariance}) if and only if
\begin{equation}\label{covariance for endomorphisms}
J\subseteq \{a\in A: SS^*\pi(a)=\pi(a)\}.
\end{equation}
Thus the $C^*$-algebra $C^*(A,\alpha;J)$ is generated by  $j_A(A) \cup j_A(A)s$ where $(j_A,s)$ is a universal representation for the representations of $(A,\alpha)$ satisfying \eqref{covariance for endomorphisms}.

If $\alpha$ is extendible, then $s \in M(C^*(A,\alpha;J))$,  $C^*(A,\alpha;J)$ is generated by  $j_A(A) \cup sj_A(A)$, and the assignments $j_A(a)\to i_A(a)$, $sj_A(a)\to u^*i_A(a)$, $a\in A$, establish the isomorphisms
$$
C^*(A,\alpha;J)\cong C^*_{endo}(A,\alpha;J),\qquad C^*(A,\alpha)\cong C^*_{endo}(A,\alpha),
$$
$$
C^*(A,\alpha;A)\cong A\rtimes^1_\alpha\N,  \qquad C^*(A,\alpha;\{0\})\cong A\rtimes_\alpha\N.
$$

\end{prop}
\begin{proof}
By the formula \eqref{ideal of covariance} and the second part of Lemma \ref{lemma on the correspondence of endomorphism} we get that a  representation $(\pi, S)$ of $(A,\alpha)$ is $J$-covariant  if and only if \eqref{covariance for endomorphisms} holds. Then the universal picture of $C^*(A,\alpha;J)$ follows by Proposition \ref{universal description of crossed products}.

Assume now that $\alpha$ is extendible. Then $s \in M(C^*(A,\alpha;J))$ by the last part of Proposition \ref{characterization of representations of cp map}, and $C^*(A,\alpha;J)=C^*(j_A(A) \cup sj_A(A))$ by Lemma \ref{dupa lemma}. Thus the universal descriptions immediately give us the isomorphism $C^*(A,\alpha;J)\cong C^*_{endo}(A,\alpha;J)$.  Taking   $J=(N_\alpha)^\bot=(\ker\alpha)^\bot$, we get $C^*(A,\alpha)\cong C^*_{endo}(A,\alpha)$. By Corollary \ref{corollary for browary} we get
 $C^*(A,\alpha;A)\cong A\rtimes^1_\alpha\N$ and $C^*(A,\alpha;\{0\})\cong A\rtimes_\alpha\N$.
\end{proof}

\subsection{The case when $A$ is commutative}\label{commutative example}

 In this subsection, we assume that $A=C_0(D)$ is the algebra of continuous,  vanishing at infinity  functions on  a locally compact Hausdorff space  $D$.
We let $\M(D)$ be the space of Radon positive measures on $D$ and treat $\M(D)$ as the subset of the dual space $A^*$ equipped with the $w^*$-topology. Let us start with a few   simple observations.
\begin{lem}
We have a one-to-one correspondence given by the relation
$$
\morp(a)(x)=\int_{D} a(y) d\mu_x (y), \qquad x\in D, a\in A,
$$
between positive maps $\morp$ on $A$ and continuous, uniformly bounded maps
\begin{equation}\label{Brenken's topological relation}
D \ni x \stackrel{\mu}{\longmapsto} \mu_x \in  \M(D)
\end{equation}
that vanish at infinity in the   $w^*$-sense, that is for every $a\in A$
and every $\varepsilon >0$ the set  $\{x :|\mu_x(a)| \geq \varepsilon\}$  is compact in $D$. Under this correspondence $\|\morp\|=\sup_{x\in D}\|\mu_x\|$.
\end{lem}
\begin{proof}
The assertion  readily follows  from  Riesz theorem, cf., for instance,  \cite[Section 1]{brenken}. In particular, using Lemma \ref{boundness equality for a cp maps} we get $\|\morp\|=\sup_{x\in D}\|\mu_x\|$.
\end{proof}
 We denote by  $\Closed(D)$  the set of all closed subsets of $D$. A mapping $\Phi:D\to \Closed(D)$  is \emph{lower-semicontinuous}   if   for every open $V\subseteq D$ the set $\{x\in D:V \cap \Phi(x)\neq \emptyset\}$ is  open. For any such mapping the set $\textrm{Dom}(\Phi):=\{x\in D: \Phi(x)\neq \emptyset\}$ is open in $D$.

\begin{lem}\label{lower-semicontinuous lemma}
Any continuous mapping \eqref{Brenken's topological relation}  induces a lower-semicontinuous  mapping
\begin{equation}\label{Brenken's topological support}
 D \ni x \stackrel{\Phi}{\longmapsto} \supp \mu_x \in  \Closed(D).
\end{equation}
\end{lem}
\begin{proof}
Assume that $V \cap \Phi(x_0)\neq \emptyset$ where $x_0\in D$ and  $V\subseteq D$ is open. Then there is a positive function $a\in C_c(D)$ that vanishes outside $V$ and such that  $\mu_{x_0}(a)>0$. Continuity of $\mu$ implies that $U:=\{x\in D:\mu_x(a)>0\}$ is  open. Since $x_0\in U\subseteq \{x\in D:V \cap \Phi(x)\neq \emptyset\}$, this proves the assertion.
\end{proof}
Let us fix a positive map $\morp$ and the corresponding maps $\mu$ and $\Phi$ given by \eqref{Brenken's topological relation} and  \eqref{Brenken's topological support}. We associate to $\Phi$ the following relation on $D$:
\begin{equation}\label{support of morp}
R_\Phi:=\bigcup_{x\in \textrm{Dom}(\Phi)}\{x\}\times  \Phi(x) .
\end{equation}
If $R_\Phi$ is closed  in $D\times D$, then $\mu$ is  called  a \emph{topological relation} in \cite{brenken}. In this case, as we show below, $\mu$ can also be viewed as a topological quiver. In general,  the relationship with topological quivers is subtle.  We recall the relevant definition, see \cite[Example 5.4]{ms2} or  \cite[Definition 3.1]{mt}.
We adopt the convention concerning the roles of the maps  $r,s$ as presented in \cite{ms2} (it differs from the one in \cite{mt}). It  fits to conventions we adopt in Section \ref{graphs section} for graph $C^*$-algebras. It is also consistent with the notation used for topological graphs by Katsura \cite{ka1}. We stress that we use the term topological graph in a broader sense than   \cite[Definition 2.1]{ka1}. Namely, we do not assume that the source map is a local homeomorphism.  Also we do not assume that  the  topological spaces underlying a topological quiver are second countable, as it is done in \cite[Definition 3.1]{mt}.
\begin{dfn}\label{topological quiver defn}
A \emph{topological graph }is a quadruple $E=(E^0,E^1,r,s)$  consisting of locally compact Hausdorff spaces $E^0$, $E^1$ and  continuous  maps $r,s:E^{1}\to E^0$, where $s$ is additionally assumed to be open. An \emph{$s$-system} on the graph $E$ is a family of Radon measures $\lambda=\{\lambda_v\}_{v\in E^0}$  on $E^1$ such that
\begin{itemize}
\item[(Q1)] $\supp \lambda_v=s^{-1}(v)$ for all $v \in E^0$,
\item[(Q2)]  $v \to \int_{E^1} a(e) d\lambda_{v} (e)$ is an element of $C_c(E^0)$ for all $a\in C_c(E^1)$.
\end{itemize}
The quintuple $(E^0,E^1,r,s,\lambda)$ where $E:=(E^0,E^1,r,s)$ is a topological graph and $\lambda$ is an $s$-system on $E$ is called a \emph{topological quiver}.
\end{dfn}

The following lemma and proposition should be compared with \cite[Proposition 2.2]{imv} stated (essentially without a proof) in the context of Markov operators, see discussion below.
\begin{lem}\label{pseudo-quiver} Consider the quintuple
$$\QQ_\morp:=(D,\overline{R_\Phi},r,s,\lambda)$$
where  $R_\Phi$ is given by \eqref{support of morp}, $s(x,y):=x$, $r(x,y):=y$ and $\lambda_x$ is a measure supported on $\{x\}\times \Phi(x) $ given by  $\lambda_x( \{x\}\times U)=\mu_x(U)$. Then $\QQ_\morp$ satisfies all axioms of  topological quiver, except that openness of the source map and axiom (Q1) hold for the restriction $s|_{R_\Phi}$ to $R_\Phi$, rather than for $s:\overline{R_\Phi}\to D$ itself.
\end{lem}
\begin{proof} Openness of $s:R_\Phi\to D$ is equivalent to lower-semicontinuity of  $\Phi$ and thus follows from Lemma \ref{lower-semicontinuous lemma}. To show the axiom (Q2) define a map $\Psi:C_c(D)\odot C_c(D) \to  C_0(D)$ by the formula
$$
\Psi(\sum_{i}a_i\odot b_i)(x):=\int \sum_{i}a_i\odot b_i  \, d\lambda_{x} = \sum_{i} a_i(x) \int b_i\,d \mu_x  =\left(\sum_{i} a_i \morp(b_i)  \right)(x).
$$
 It is well defined because $\sum_{i} a_i \morp(b_i)\in C_0(D)$ and it is linear because it is given by the integral. It is  bounded with $\|\Psi\|\leq \sup_{x\in D}\|\lambda_x\|=\sup_{x\in D}\|\mu_x\|=\|\morp\|$. Thus, since $C_c(D)\odot C_c(D)$ is uniformly dense in $C_0(D\times D)$ we deduce  that the formula $\Psi(a)(x)=\int a d\lambda_{x}$   defines  a bounded linear map $\Psi:C_0(D\times D)\to C_0(D)$. Concluding, for any   $a\in C_c(\overline{R_\Phi})$ we see that $x \to \int_{E^1} a(x,y) d\lambda_x (y)$ defines a continuous function on $D$ which vanishes outside the compact set $s(\supp(a))$. This proves (Q2). The rest is clear by construction.
\end{proof}
\begin{rem} By the above lemma the quintuple $(D,R_\Phi,r,s,\lambda)$ is a topological quiver whenever $R_\Phi$ is locally compact.
However,  if $R_\Phi$ is   not closed in $D\times D$ then the mapping \eqref{mapping W} below, with $R_\Phi$ in place  of $\overline{R_\Phi}$, is not well defined.\end{rem}
\begin{prop}\label{Correspondence of a pseudo-quiver} The quintuple $\QQ_\morp=(D,\overline{R_\Phi},r,s,\lambda)$ from Lemma \ref{pseudo-quiver} gives rise to a $C^*$-correspondence $X_{\QQ_\morp}$ which is the Hausdorff completion of the semi-inner $C^*$-correspondence defined on $C_c(\overline{R_\Phi})$  via
$$
(a \cdot f \cdot b) (x,y)=\big((a \circ s) f (b\circ r)\big) (x,y)=a(x)f(x,y)b(y),
$$
and
$$
\langle f , g \rangle_{\QQ_\morp} (x)=\int_{R_\Phi} \overline{f} g d\lambda_x=\int_{\Phi(x)} \overline{f(x,y)} g(x,y) d\mu_x(y),
$$
$f,g \in C_c(\overline{R_\Phi})$, $a,b \in C_0(D)=C_0(D)$. Moreover, putting  $W(a\odot b)(x,y):=a(x)b(y)$, for $(x,y)\in \overline{R_\Phi}$ the mapping determined by
\begin{equation}\label{mapping W}
C_c(D)\odot C_c(D) \ni a\odot b  \to  W(a\odot b)  \in  C_c( \overline{R_\Phi}),
\end{equation}
 factors through and extends to an isomorphism of $C^*$-correspondences $X_\morp\cong  X_{\QQ_\morp}$.
\end{prop}
\begin{proof} It is a routine exercise to check that $C_c(\overline{R_\Phi})$ is a semi-inner product (right) $A$-module, equipped with a left $A$-module action by adjointable operators, \cite[p. 3]{Lan}. Thus we get the $C^*$-correspondence $X_{\QQ_\morp}$. Furthermore, one readily checks that \eqref{mapping W} determines  a well-defined map $W:C_c(D)\odot C_c(D) \to C_c( \overline{R_\Phi})$ satisfying
$$
a W( f\odot g) b =W(af \odot gb), \qquad \langle W(f\odot g) , W(f\odot g) \rangle_{\QQ}= \langle f\odot g , f\odot g \rangle_{\morp},
$$
for all $a,b \in C_0(D), f,g\in C_c(D)$.  Since, the image of $C_c(D)\odot C_c(D)$ is dense in $X_\morp$ it follows that $W$
 factors through and extends to an isometric homomorphism of $C^*$-correspondences $\mathcal{W}:X_\morp \to X_{\QQ_\morp}$. To see  it is surjective, note that by the Stone-Weierstrass theorem for any   $f\in C_c( \overline{R_\Phi})$ we can find  a sequence $\{f_n\}\subseteq C_c(D)\odot C_c(D)$ such that $W(f_n)$ converges uniformly  to $f$, and thus all the more in the semi-norm induced by  $\langle \cdot , \cdot \rangle_{\QQ}$.   Hence $\mathcal{W} $ is the desired isomorphism.
\end{proof}
By \cite[Definition 3.17]{mt}, the \emph{$C^*$-algebra associated to a topological quiver} $\QQ$  is  the Cuntz-Pimsner algebra $\OO_{X_\QQ}$ of a certain $C^*$-correspondence $X_\QQ$.  If the quintuple ${\QQ_\morp}=(D,\overline{R_\Phi},r,s,\lambda)$ defined in Lemma \ref{pseudo-quiver} is a topological quiver, then the $C^*$-correspondence  $X_{\QQ_\morp}$  constructed in Proposition \ref{Correspondence of a pseudo-quiver} coincides with  the $C^*$-correspondence associated to ${\QQ_\morp}$ in  \cite{mt}.  Thus Proposition \ref{Correspondence of a pseudo-quiver} and Theorem \ref{identification theorem} give the following proposition.

\begin{prop}\label{corollary for okulary1} Suppose that  the quintuple ${\QQ_\morp}=(D,\overline{R_\Phi},r,s,\lambda)$ defined in Lemma \ref{pseudo-quiver} is a topological quiver (it is automatic when $R_\Phi$ is closed). Then the crossed product $C^*(A,\morp)$ is naturally isomorphic to the $C^*$-algebra  $C^*({\QQ_\morp})$ associated to ${\QQ_\morp}$.
\end{prop}

If the set $R_\Phi$ is closed, that is, if $\mu$ is a topological relation, then the $C^*$-correspondence  $X_{\QQ_\morp}$  constructed in Proposition \ref{Correspondence of a pseudo-quiver} coincides with the one associated to $\mu$ in \cite{brenken}. In particular, Brenken defines a \emph{$C^*$-algebra $\mathcal{C}(\mu)$ associated to the topological relation} $\mu$ as Pimsner's (augmented) $C^*$-algebra $\OO(J(X_{\QQ_\morp}),X_{\QQ_\morp})$.
Hence by Proposition \ref{Correspondence of a pseudo-quiver} and Theorem \ref{identification theorem} we have the following proposition.
\begin{prop}\label{corollary for okulary} If  $\morp:C_0(D)\to C_0(D)$ is such that the map $\mu$ given by \eqref{Brenken's topological relation} is a topological relation then the $C^*$-algebra $\mathcal{C}(\mu)$ associated to $\mu$ is naturally isomorphic to the relative crossed product $C^*(A,\morp; A)$.
\end{prop}
Suppose now that $D$ is compact and $\morp$ is unital (equivalently, every measure $\mu_x$, $x\in D$, is a probability distribution).  Then   $\morp$ is called a \emph{Markov operator}  in  \cite[Definition 1]{imv}. The closure $\overline{R_\Phi}$ of the set  \eqref{support of morp} coincides with the support of $\morp$  defined in  \cite[Definition 4]{imv}. It seems that the authors of  \cite{imv} tacitly assumed that the corresponding set $R_\Phi$ is closed, since they model their algebras via $C^*$-algebras associated to topological quivers. Namely, they define \cite[Definition 7]{imv} the $C^*$-algebra $C^*(\morp)$ of the Markov operator $\morp$ to be the $C^*$-algebra of the quintuple $(D, \overline{R_\Phi}, r,s, \lambda)$ described in Lemma \ref{pseudo-quiver}. But in general  neither  $(D, R_\Phi, r,s, \lambda)$ nor $(D, \overline{R_\Phi}, r,s, \lambda)$, satisfies all of the axioms of a topological quiver.
\begin{ex} Consider the following  Markov operators $\morp_1$, $\morp_2$ on $C([0,1])$:
\begin{align*}
&\morp_1(a)(x)=
\begin{cases}
\int_0^1 a(t) f_x(t)\, dt & x \neq 0 \\
a(1), & x=0
\end{cases},  \qquad  f_x(t)=\frac{t^{1/x}}{\int_0^1 s^{1/x}\, ds};
\\
&\morp_2(a)(x)=
f(x)a(0)+ f(1-x)a(1), \,\,\, f(x)=\chi_{[0,1/3)}(x) + (2-3x)\chi_{[1/3, 2/3)}(x).
\end{align*}
Then the corresponding relations on $[0,1]$ are
$$
R_{1}=\Big((0,1]\times [0,1]\Big)\cup\{(0,1)\}, \qquad   R_{2}=\Big([0,2/3)\times \{0\}  \Big) \cup   \Big((1/3,1]\times \{1\} \Big).
$$
Plainly, $R_{1}$ is not locally compact in  $(0,1)$, while  the source map  on  $\overline{R_{2}}$  is not open.
\end{ex}
The above example shows that  the assertion in \cite[Proposition 2.2]{imv} is false.  Proposition \ref{Correspondence of a pseudo-quiver} could be considered a  correct version of this statement. It  suggests that in general the $C^*$-algebra $C^*(\morp)$ should be defined as the Cuntz-Pimsner algebra   $\OO_{X_{\QQ_\morp}}$ of the $C^*$-correspondence  $X_{\QQ_\morp}$ described in Proposition \ref{Correspondence of a pseudo-quiver}. Then  Theorem \ref{identification theorem} gives us the following proposition.

\begin{prop} The $C^*$-algebra $C^*(\morp)$ of a Markov operator $\morp:C(D)\to C(D)$ is naturally isomorphic to the crossed product $
C^*(A,\morp)$.
\end{prop}

\section{A new look at Exel systems and their crossed products}\label{general exel systems section}

Throughout this section,  $(A,\alpha,\LL)$ denotes an Exel system. We show that Exel-Royer's crossed product $\OO(A,\alpha, \LL)$ is the  crossed product $C^*(A,\LL)$  and   Exel's crossed product $A\rtimes_{\alpha,\LL} \N$  is the relative crossed product $C^*(A,\LL; \overline{A\al(A)A})$. We study in detail the structure of  Exel systems with the property that $\alpha\circ \LL$ is a conditional expectation, and  discuss  cases when  we have  $A\rtimes_{\alpha,\LL} \N=C^*(A,\LL)$.

\subsection{Exel's crossed products as crossed products by transfer operators}
Let us start with the following simple but fundamental observation.
\begin{lem} Any transfer operator is a completely positive map.
\end{lem}
\begin{proof}
Using \eqref{transfer operator relation} and its symmetrized version:  $
\LL(\al(b)a)=b\LL(a)$, $a,b\in A$,   for any $a_i$, $b_i \in A$, $i=1,...,n$, we get
\begin{equation}\label{completely positivity of L}
\sum_{i,j=1}^n b_i^*\LL(a_i^*a_j)b_j= \LL\Big(\big(\sum_{i=1}^n a_i\al(b_i)\big)^* \big(\sum_{j=1}^n a_j\al(b_j)\big)\Big)\geq 0.
\end{equation}
Hence $\LL$ is a completely positive map.
\end{proof}

Authors of  \cite{brv} and \cite{Larsen} considered Exel systems  $(A,\al,\LL)$  under the additional assumption that both $\alpha$ and $\LL$ are extendible. It turns out that  extendibility of $\LL$ is automatic.

\begin{prop}\label{extendibility of transfer operators}
Any transfer operator $\LL$ for $\al$ is extendible. Its strictly continuous extension $\overline{\LL}:M(A)\to M(A)$ is determined by the formula
\begin{equation}\label{extension of transfer operator}
\overline{\LL}(m)a=\LL(m\al(a)), \qquad a \in A, \,\, m\in M(A).
\end{equation}
In particular,  $
\overline{\LL}(1)
$ is a positive central element in $M(A)$, and if $\alpha$ is extendible  then the triple $(M(A),\overline{\al},\overline{\LL})$ is an Exel system.
\end{prop}
\begin{proof}
Fix $m\in M(A)$. We claim that   \eqref{extension of transfer operator} defines a  multiplier, that is an adjointable mapping $\overline{\LL}(m):A\to A$  where we view $A$ as  the standard Hilbert $A$-module. Indeed, for any $a,b \in A$ we have
$$
(\overline{\LL}(m)a)^*b=\LL(m\al(a))^*b=\LL(\al(a^*)m^*)b=\LL(\al(a^*)m\al(b))=a^*(\overline{\LL}(m^*)b).
$$
Hence $\overline{\LL}(m)\in M(A)$ and $\overline{\LL}(m)^*=\overline{\LL}(m^*)$. Accordingly, \eqref{extension of transfer operator} defines a $*$-preserving  mapping $\overline{\LL}:M(A)\to M(A)$. It follows directly from \eqref{extension of transfer operator} that $\overline{\LL}$ is  a  strictly continuous extension of $\LL:A\to A$.   Moreover, for $a\in A$, we have
$
\overline{\LL}(1)a=\LL(1\al(a))=\LL(\al(a)1)=a\overline{\LL}(1)
$. Thus $\overline{\LL}(1)$ belongs to the commutant of $A$ in $M(A)$. This  commutant coincides with the center of $M(A)$. If additionally $\alpha$ is extendible then $\overline{\LL}$ is a transfer operator for $\overline{\alpha}$ because  \eqref{transfer operator relation} is preserved when   passing to strict limits.
\end{proof}
Another somehow unexpected fact is that the first relation in \eqref{Exel relations} is superfluous.
\begin{prop}\label{destroying proposition}
Suppose that $(A,\al, \LL)$ is an  Exel system. For any representation $(\pi,S)$ of $(A,\LL)$ we automatically have
$$
S\pi(a) =\pi(\al(a))S, \quad a\in A.
$$
Thus the classes of representations of $(A,\LL)$ and $(A,\al,\LL)$ coincide and
$$\TT(A,\LL) = \TT(A,\al,\LL).$$
Moreover, the notions of  redundancy  for  $(\pi,S)$ as a representation of $(A,\LL)$ and as a representation of $(A,\al,\LL)$ coincide.
\end{prop}
\begin{proof} Let $(\pi,S)$ be a representation of $(A,\LL)$, $\overline{\pi}:M(A)\to\B(H)$ the  extension of $\pi$, and $\overline{\LL}$ the strictly continuous extension of $\LL$, which exists by Proposition \ref{extendibility of transfer operators}. It readily follows  that $(\overline{\pi},S)$ is a representation of $(M(A),\overline{\LL})$. In particular,  $S^*S=\overline{\pi}(\overline{\LL}(1))$. Using this and  \eqref{extension of transfer operator},    one sees that   each of the  expressions
$$
S^*\pi(\al(a^*))S\pi(a), \quad S^*\pi(\al(a^*))\pi(\al(a))S, \quad \pi(a^*)S^*S\pi(a), \quad \pi(a^*)S^*\pi(\al(a))S
$$
is equal to $\pi\big(\LL(\alpha(a^*a))\big)$, for any $a\in A$. Hence we get
\begin{align*}
\|S\pi(a) -\pi(\al(a))S\|^2&=\|\big(S^*\pi(\al(a^*)) -\pi(a^*)S^*\big)\big(S\pi(a) -\pi(\al(a))S\big)\|=0.
\end{align*}
This finishes the proof of the first part of the assertion. For the second part it suffices to show that $\overline{\pi(A)S\pi(A)S^*\pi(A)}=\overline{\pi(A)SS^*\pi(A)}$. In view of what we  have just shown we have
$$
\overline{\pi(A)S\pi(A)S^*\pi(A)}=\overline{\pi(A)\pi(\al(A))SS^*\pi(A)}\subseteq\overline{\pi(A)SS^*\pi(A)}.
$$
Moreover, the last inclusion is the equality because  $\pi(a)S=\lim_{\lambda \in \Lambda}\pi(a\al(\mu_\lambda))S$ for any approximate unit $\{\mu_\lambda\}_{\lambda \in \Lambda}$ in $A$. Indeed,
\begin{align*}
\|\pi(a)S-\pi(a\al(\mu_\lambda))S\|^2&=
\|\pi\big(\LL(a^*a)- \LL(a^*a)\mu_\lambda- \mu_\lambda\LL(a^*a) + \mu_\lambda\LL(a^*a )\mu_\lambda\big)\|
\end{align*}
clearly tends to $0$.
\end{proof}
The above coincidence can be explained at the level of $C^*$-correspondences.
\begin{lem}\label{Exel vs GNS}
The $C^*$-correspondence $M_{\LL}$ associated to $(A,\al,\LL)$ and the $C^*$-correspon\-dence $X_\LL$ associated to $(A,\LL)$ are isomorphic, via the mapping
determined by $a\otimes b \longmapsto q(a\al(b))$, $a, b\in A$.
\end{lem}
\begin{proof}
That $a\otimes b \longmapsto q(a\al(b))$ yields a well defined isometry follows  from the equality in \eqref{completely positivity of L}. Clearly, it is a $C^*$-correspondence map. It is onto because $q(a\alpha(\mu_\lambda))$ converges in $M_{\LL}$ to $q(a)$ for any $a\in A$ and any approximate unit $\{\mu_\lambda\}$ in $A$, cf. \cite[Lemma 3.6]{kakariadis}.
\end{proof}
\begin{cor}\label{Toeplitz corollary}
For every Exel system $(A,\al,\LL)$ we have
$$
\TT( M_{\LL})\cong \TT(X_\LL)\cong \TT(A,\LL) = \TT(A,\al,\LL).
$$
\end{cor}
\begin{proof} We have $\TT( M_{\LL})\cong \TT(X_\LL)$ by  Lemma \ref{Exel vs GNS}. Proposition \ref{characterization of representations of cp map} implies that $\TT(X_\LL)\cong \TT(A,\LL)$,      and we have  $\TT(A,\LL) = \TT(A,\al,\LL)$ by Proposition \ref{destroying proposition}.
\end{proof}
\begin{rem}\label{slight generalization of brv remark}
The isomorphism $\TT( M_{\LL})\cong \TT(A,\al,\LL)$  was proved in  \cite[Corollary  3.3]{br}, cf. \cite[Theorem 3.7]{kakariadis}, for $A$ unital, and  independently in \cite[Proposition 3.1]{brv} and \cite[Proposition 4.1]{Larsen}, for extendible Exel systems.
\end{rem}
Now we are in a position to show the main result of this subsection.
\begin{thm}\label{Exel crossed products as relative Pimsner}
For any  Exel system $(A,\al, \LL)$ we have
$$
C^*(A,\LL)=\OO(A,\al, \LL)\cong \OO_{M_\LL}, \qquad A\times_{\al, \LL} \N = C^*(A,\LL;J),
$$
where $J:=\overline{A\al(A)A}$. In particular,  we have
$$
A\times_{\al, \LL} \N \cong\OO(X_\LL, J\cap J(X_\LL))\cong\OO( M_{\LL},J\cap J(M_{\LL})),
$$
and the universal homomorphism $j_A:A \to A\times_{\al, \LL} \N $ is injective if and only if $\LL$ is almost faithful on $\overline{A\al(A)A}\cap J(M_\LL)$.
\end{thm}
\begin{proof}
Equality  $A\times_{\al, \LL} \N = C^*(A,\LL;J)$ follows from Proposition  \ref{destroying proposition}. To get $\OO(A,\al, \LL)=C^*(A,\LL)$ combine Proposition  \ref{destroying proposition}, Lemma \ref{Exel vs GNS}   and the first part of Theorem \ref{identification theorem}. The isomorphism  $C^*(A,\LL)\cong \OO_{M_\LL}$ follows now either from Remark \ref{Remark for the referee who loves Kastura's work} or from Lemma \ref{Exel vs GNS}   and  Theorem \ref{identification theorem}. The second part of the assertion follows now from  Lemma \ref{Exel vs GNS}   and Theorem \ref{identification theorem}.
\end{proof}
\begin{rem}
The second part of the assertion in the above theorem  generalizes \cite[Proposition 3.10 and Theorem 4.2]{br} proved in the unital case, and \cite[Theorems 4.1 and 4.3]{brv} where authors assumed extendibility of the Exel system.
\end{rem}
Brownlowe, Raeburn and Vitadello  proved in \cite[Corollary 4.2]{brv} that  Exel's crossed products for Exel systems $(C_0(T), \alpha,\LL)$ induced by classical dynamical systems $(T, \tau)$ are naturally isomorphic to $\OO_{M_\LL}$. For these systems $\LL$  is faithful and $\alpha$ is extendible. It turns our that the latter properties suffice to get the corresponding isomorphism.
\begin{prop}\label{corollary on faithful transfers} Let  $(A,\alpha, \LL)$ be an Exel system and suppose that one of the following conditions hold:
\begin{itemize}
\item[i)] $\LL$ is almost faithful on $A$ and $\alpha$ is a non-degenerate homomorphism.
\item[ii)] $\LL$ is   faithful and $\alpha$ is extendible.
\end{itemize}
Then  $
A\times_{\al, \LL} \N = C^*(A,\LL)\cong \OO_{M_\LL}.
$
\end{prop}
\begin{proof}
i). The assumptions imply that $N_\LL^\bot=A$ and $\overline{A\alpha(A)A}=A$. By Theorem \ref{Exel crossed products as relative Pimsner}, we get $A\times_{\al, \LL} \N  =C^*(A,\LL; A)=C^*(A,\LL;N_\LL^\bot)=C^*(A,\LL)\cong \OO_{M_\LL}$.

ii). By item i) it suffices to show that $\alpha$ is non-degenerate. For any $a\in A$ the element $
\LL((a-a\overline{\alpha}(1))^*(a-a\overline{\alpha}(1)))$ is equal to
$$
\LL(a^*a) - \LL(\overline{\alpha}(1)a^*a) - \LL(a^*a \overline{\alpha}(1))+ \LL(\overline{\alpha}(1)a^*a\overline{\alpha}(1))=0.
$$
Thus by faithfulness of $\LL$ we have  $a=\overline{\alpha}(1)a$. This implies that $A=\overline{\alpha}(1)A=\alpha(A)A$.
\end{proof}

\begin{rem}
Kakariadis and Peters \cite[Remark 3.19]{kakariadis} raised a question  whether Exel's crossed product $A\times_{\al, \LL} \N$ is always isomorphic to the Cuntz-Pimsner algebra $\OO_{M_\LL}$ (by Theorem \ref{Exel crossed products as relative Pimsner}, the latter is always isomorphic to  $C^*(A,\LL)$). The answer to this question, stated as it is,  is  no. The  reason is that  $A$ always embeds into $\OO_{M_\LL}\cong C^*(A,\LL)$ while for Exel's crossed product in general this fails, see for instance   \cite[Example 4.7]{br}.
Thus, taking into account Theorem \ref{Exel crossed products as relative Pimsner}, we propose the following  modified version of this question:
\begin{quote}
\emph{Let $(A,\alpha,\LL)$ be an Exel system such that $\overline{A\alpha(A)A}\cap J(M_\LL)\subseteq N_\LL^\bot$. Do the crossed products $C^*(A,\LL)$ and $A\times_{\al, \LL} \N$ coincide?}
\end{quote}
Since $C^*(A,\LL;J)=C^*(A,\LL; J\cap  J(M_\LL))$, the answer to the above question, for systems under consideration, is positive if and only if
\begin{equation}\label{inclusion relation}
 N_{\LL}^\bot \cap  J(M_\LL) \subseteq \overline{A\alpha(A)A}.
\end{equation}
The most problematic part in establishing \eqref{inclusion relation} is determining $J(M_\LL)$. For instance, when $\alpha$ is extendible and $\LL$ is faithful then  $\overline{A\alpha(A)A}=A=N_\LL^\bot$ and hence \eqref{inclusion relation} holds independently of $J(M_\LL)$. Interestingly enough,  if additionally  $E=\LL\circ \alpha:A\to \alpha(A)\subseteq A$ is a conditional expectation of finite-type then $J(M_\LL)=A$,  see \cite{exel_vershik}.  However, in general we have $\overline{A\alpha(A)A}\neq N_\LL^\bot$ and $J(M_\LL)\cap N_{\LL}^\bot \neq  N_{\LL}^\bot$. This may happen already when $\LL$ is faithful but $\alpha$ is not extendible,  cf. Lemmas \ref{description of the relevant ideals} and \ref{existence of endomorphism}. Surprisingly, when $\alpha(A)$ is a corner in $A$, see  Theorem \ref{complete systems crossed products} below, and  also for all Exel systems considered in Section \ref{graphs section} we  have $J(M_\LL)\cap N_{\LL}^\bot=\overline{A\alpha(A)A}$.
\end{rem}
\subsection{Regular transfer operators}\label{Regular section}

 Most of natural  Exel systems appearing in applications, see \cite{exel2}, \cite{exel_vershik}, \cite{br}, \cite{brv},  \cite{kwa-trans},  have the property that  $\alpha\circ \LL$ is a conditional expectation onto $\alpha(A)$. In \cite{exel2} Exel called   transfer operators with that property  \emph{non-degenerate}. However,  we have reasons to change this name. Firstly, the  term  `non-degenerate' when referred to a positive operator   is sometimes used to mean
a  faithful map \cite[page 60]{exel-inter} or a strict map \cite[subsection 3.3]{Schweizer}. Secondly, there are historical reasons. Namely,  transfer operators on  unital commutative $C^*$-algebras, under the name averaging operators,  were studied   at least from 1950's see \cite{Pelczynski}, cf. \cite{kwa-trans}. The averaging operators are called \emph{regular} exactly when the corresponding composition $\alpha\circ \LL$ is a conditional expectation, cf. \cite[Proposition 2.1.i)]{kwa-trans}. Therefore we adopt the following definition.
\begin{dfn}  Let $(A,\al,\LL)$ be an  Exel system. We say that both the transfer operator $\LL$  and the Exel system $(A,\al,\LL)$ are  \emph{regular}, if  $E:=\alpha\circ \LL$ is a conditional expectation onto $\alpha(A)$.
\end{dfn}
Let us start with a simple fact.
\begin{lem}\label{structure of transfer operators}
Let $(A,\al,\LL)$ be an  Exel system. The range $\LL(A)$ of the transfer operator $\LL$ is a self-adjoint two-sided (not necessarily closed) ideal in $A$ such that $\ker\al\subseteq \LL(A)^\bot$.
\end{lem}
\begin{proof}
Since $\LL$ is linear and $*$-preserving, $\LL(A)$  is a self-adjoint   linear space. The space $\LL(A)$  is a two-sided ideal in $A$ by \eqref{transfer operator relation}.
For $a\in \ker \al$ we have
$
a\LL(A)= \LL(\al(a)A)=\LL(0)=0.
$
Hence $\ker\al\subseteq \LL(A)^\bot$.
\end{proof}
Suppose that the central positive element $\overline{\LL}(1)\in M(A)$, described in Proposition \ref{extendibility of transfer operators},   is a projection.  Then by  the above lemma the multiplier $\overline{\LL}(1)$ projects $A$ onto an ideal contained in $(\ker\al)^\bot$. It turns out that $\LL$ is regular exactly  when  $\overline{\LL}(1)$  projects $A$ onto $(\ker\al)^\bot$.
\begin{prop}
\label{Ex2.3}
Let $(A,\al,\LL)$ be an Exel system and let $\{\mu_\lambda\}$  be an approximate unit  in $A$. The following conditions are equivalent:
\begin{itemize}
\item[i)] $\LL$ is regular, that is $E=\al\circ \LL:A\to \al(A)$ is a conditional expectation,
\item[ii)] $\{\alpha(\LL(\mu_\lambda))\}$ is an approximate unit in $\al(A)$,
\item[iii)] $\alpha\circ\LL\circ\alpha =\alpha$,
\item[iv)] $\{\LL(\mu_\lambda)\}$ converges strictly to a  projection $\overline{\LL}(1)\in M(A)$ onto $(\ker\al)^\bot$,
\item[v)] $(\alpha,\LL)$ is an interaction in the sense of  \cite[Definition 3.1]{exel-inter}, see Remark \ref{remark on interactions}.
\end{itemize}
In particular, if the above equivalent conditions hold, then   $\ker\al$ is a complemented ideal in $A$,
 $
 \overline{\LL}(1)A=\LL(A)=(\ker\al)^{\bot}$, $(1-\overline{\LL}(1))A=\ker\al$,
 and
\begin{equation}\label{non-degenerate transfers formula}
 \LL(a)=\al^{-1}(E(a)),  \qquad \al(a)=\LL^{-1}(\overline{\LL}(1)a), \qquad a\in A,
 \end{equation}
 where    $\al^{-1}$ is the inverse to the isomorphism $\al:(\ker\al)^\bot \to \al(A)$, and $\LL^{-1}$ is the inverse to the isomorphism $\LL:\al(A)\to \LL(A)$.
\end{prop}
\begin{proof}  i)$\Rightarrow$ii). For  any $a\in \alpha(A)$ we have
$
\lim_{\lambda \in \Lambda}\al(\LL(\mu_\lambda)) a=\lim_{\lambda \in \Lambda} E(\mu_\lambda a)=E(a)=a.
$

ii)$\Rightarrow$iii).  For any $a\in A$ we have
$$
\al(a)=\lim_{\lambda \in \Lambda}\al(\LL(\mu_\lambda))\al(a)=\lim_{\lambda \in \Lambda}\al(\LL(\mu_\lambda)a)=\lim_{\lambda \in \Lambda}\al(\LL(\mu_\lambda \al(a)))=\al(\LL(\al(a))).
$$

iii)$\Rightarrow$iv). By Proposition \ref{extendibility of transfer operators}, $\{\LL(\mu_\lambda)\}$ converges strictly to a  central element $\overline{\LL}(1)$ in $M(A)$. In particular, using  \eqref{extension of transfer operator},
for  $a\in A$, we get
$$
\overline{\LL}(1)^2a=\overline{\LL}(1) \LL(\al(a))= \LL(\al(\LL(\al(a))))=\LL(\al(a))=\overline{\LL}(1)a.
$$
Hence $\overline{\LL}(1)$ is a projection. On one hand $(1-\overline{\LL}(1))A\subseteq \ker\al$ because
$$
\al((1-\overline{\LL}(1))a)=\al(a)-\al(\LL(\al(a)))=\al(a)-\al(a)=0,
$$
for any $a\in A$. On the other hand, $\ker\al \subseteq (1-\overline{\LL}(1))A$ because if $a\in \ker\al$ then
$$
(1-\overline{\LL}(1))a=a-\LL(\al(a))=a.
$$
Accordingly, $\ker\al =(1-\overline{\LL}(1))A$ and $(\ker\al)^\bot =\overline{\LL}(1)A$.

iv)$\Rightarrow$v).
If   $\overline{\LL}(1)$ is a projection onto $(\ker\al)^\bot$ then in view of \eqref{extension of transfer operator} for $a\in A$ we have
$$
\al(a)=\al(\overline{\LL}(1)a)=\al(\LL(\al(a)),
$$
that is $\al=\al\circ \LL\circ \al$. By Lemma \ref{structure of transfer operators},  $\ker\al\subseteq \LL(A)^\bot$.  This implies that   $\LL(A)\subseteq (\LL(A)^\bot)^\bot\subseteq (\ker\al)^\bot=\overline{\LL}(1)A$. Consequently,
$$
\LL(a)=\overline{\LL}(1) \LL(a)=\LL(\al(\LL(a))),
$$
that is $\LL=\LL\circ \al\circ\LL$. We note that as $
\overline{\LL}(1)A=\LL(\al(A))\subseteq \LL(A)
$ we actually get  $\LL(A)=\overline{\LL}(1)A=(\ker\al)^\bot$. Clearly, $\LL(A)\subseteq MD(\alpha)=A$. We have $\alpha(A)\subseteq MD(\LL)$ because
$$
\LL(\alpha(a)b)=a \LL(b)= \overline{\LL}(1) a\LL(b)=\LL(\alpha(a))\LL(b),
$$
and similarly $\LL(b\alpha(a))= \LL(b)\LL(\alpha(a))$, $a,b\in A$.

v)$\Rightarrow$i). $E=\al\circ \LL$ is a conditional expectation by \cite[Corollary 3.3]{exel-inter}.
\end{proof}
\begin{rem}
If $A$ is unital  the equivalence of  conditions i), ii), iii) above (with units in place of approximate units) was proved by Exel \cite[Proposition 2.3]{exel2}, and  in \cite[Proposition 1.5]{kwa-trans} it was noticed that they imply that $\LL(1)$ is a central projection with $\LL(1)A=\LL(A)=(\ker\al)^{\bot}$.
 \end{rem}

The following classification of regular transfer operators generalizes  \cite[Theorem 1.6]{kwa-trans} to the non-unital case.

\begin{prop}\label{proposition for rocky}
Fix an endomorphism $\al:A\to A$. If $\alpha$   admits a regular transfer operator then its  kernel is a complemented ideal in $A$ and the formulas
$$
 \LL=\al^{-1} \circ E , \qquad E=\alpha\circ \LL
$$
where    $\al^{-1}$ is the inverse to  $\al:(\ker\al)^\bot \to \al(A)$, establish a one-to-one correspondence between conditional expectations $E$ from $A$ onto $\al(A)$ and regular transfer operators $\LL$ for $\al$.

In particular, if the range of $\al$ is a hereditary subalgebra of $A$, then $\al$ admits at most one regular transfer operator.
\end{prop}
\begin{proof} The first part  follows immediately from  Proposition \ref{Ex2.3}. For the second part notice that every conditional expectation $E:A\to B\subseteq A$ is determined by its restriction to the hereditary $C^*$-subalgebra $BAB$ of $A$ generated by $B$.
\end{proof}

Now, we reverse the situation and parametrize all regular Exel systems for a fixed transfer operator.
To this end, we recall, cf. \cite[subsection 1.3]{Schweizer}, that  a completely positive contraction $\morp:A\to B$ is called  a \emph{retraction} if there exists a homomorphism $\theta:B\to A$ such that $\morp\circ \theta= id_B$; then $\theta$ is called a \emph{section} of $\morp$.

\begin{prop}\label{proposition implying interactions}
A completely positive mapping $\LL:A\to A$ is a regular transfer operator for a certain endomorphism if and only if $\LL(A)$ is a complemented ideal in $A$ and $\LL:A\to \LL(A)$ is a retraction.

If the above conditions hold, we have bijective correspondences between the following objects:
\begin{itemize}
\item[i)] endomorphisms $\alpha:A\to A$ making $(A,\alpha,\LL)$ into a regular Exel system,
\item[ii)] sections $\theta:\LL(A)\to A$ of $\LL:A\to \LL(A)$,
\item[iii)] $C^*$-subalgebras  $B\subseteq MD(\LL)$ such that $\LL: B\to \LL(A)$ is a bijection.
\end{itemize}
These correspondences are given by the relations
\begin{equation}\label{alpha_b definition}
\alpha(a)=\theta(\overline{\LL}(1)a), \qquad a\in A, \qquad B= \alpha(A)=\theta(\LL(A)),
\end{equation}
where   $\overline{\LL}(1)$ is the projection  onto $\LL(A)$ and $\theta:\LL(A)\to B$ is the inverse to $\LL:B\to \LL(A)$.
\end{prop}
\begin{proof} By virtue of Proposition \ref{Ex2.3}, if $\LL$ is a regular transfer operator for $\alpha$ then $\LL(A)$ is a complemented ideal in $A$ and $\alpha$ is given by the first formula in \eqref{alpha_b definition}. Conversely, if $\LL(A)$ is a complemented ideal in $A$, then the projection $\overline{\LL}(1) \in M(A)$ onto $\LL(A)$ commutes with elements of $A$. Hence for any section $\theta:\LL(A)\to A$,  the first formula in \eqref{alpha_b definition} defines a homomorphism such that $(A,\alpha,\LL)$ is a regular Exel system.
Thus we have a bijection between objects in items i) and ii). If $\theta$ is a section of $\LL:A\to \LL(A)$ and $\alpha$ is the corresponding endomorphism then $B:= \alpha(A)=\theta(\LL(A))$ is a $C^*$-subalgebra of $MD(\LL)$ by Proposition \ref{Ex2.3}v). By the same proposition  $\LL: B\to \LL(A)$ is a bijection. Conversely, if $B$ is a $C^*$-algebra such that  $B\subseteq MD(\LL)$ and $\LL: B\to \LL(A)$ is a bijection, then the inverse to $\LL:B\to \LL(A)$ is a section of $\LL:A\to \LL(A)$. This shows the  correspondence between objects in ii) and iii).
\end{proof}

\begin{ex}[cf. Example 4.7 in \cite{br}]
Let $\LL:C([0,2])\to C([0,2])$ be given by $\LL(a)(x)=a(x/2)$. Regular Exel systems $(C([0,2]),\alpha,\LL)$ are parametrized by continuous extensions of the mapping $[0,1]\ni x \to 2x \in [0,2]$; for any such system we have
$$
\alpha(a)(x)=
\begin{cases}
a(2x), & \textrm{if }x\in [0,1]\\
a(\gamma(x)), &  \textrm{if }x\in [1,2]
\end{cases}, \qquad\quad a\in C([0,2]),
$$
where $\gamma:[1,2]\to [0,2]$ is continuous and $\gamma(1)=2$. In other words, the  algebras $B$ in Proposition \ref{proposition implying interactions} correspond to continuous mappings on $[0,2]$ whose restriction to $[0,1]$ is $2x$.
\end{ex}

\subsection{Exel's crossed products for regular Exel systems}\label{crossed products for regular Exel systems subsection}
By Theorem \ref{Exel crossed products as relative Pimsner}, $A$ embeds into $A\times_{\al, \LL} \N$ if and only if $\overline{A\alpha(A)A}\cap J(M_\LL)\subseteq N_\LL^\bot$.
In this subsection, we consider regular Exel systems satisfying   stronger, but easier to  check in practice, condition:
 $\overline{A\alpha(A)A}\subseteq N_\LL^\bot$. The latter inclusion holds,  for instance, for  systems with faithful  transfer operators or with corner endomorphisms. These  are the cases when Exel's crossed product boasts its greatest successes, see \cite{exel2}, \cite{exel_vershik}, \cite{br}, \cite{brv}. We show that for such systems Exel crossed product $A\times_{\al, \LL} \N$ can be defined without a use of  $\alpha$.

We recall that any positive map $\LL:A\to A$ restricts to the homomorphism  $\LL:MD(\LL)\to \LL(A)$. The kernel $(\ker \LL|_{MD(\LL)})$ of this homomorphism is an ideal in $ MD(\LL)$ and we may consider its annihilator $(\ker \LL|_{MD(\LL)})^{\bot}$  in $ MD(\LL)$. Thus  $(\ker \LL|_{MD(\LL)})^{\bot}$  is a $C^*$-subalgebra of $A$.
\begin{prop}\label{Exel without endomorphism}
Suppose that $(A,\alpha,\LL)$ is a regular Exel system such that  $\LL$ is  faithful on $\overline{A\alpha(A)A}$.  Then $\alpha(A)=(\ker \LL|_{MD(\LL)})^{\bot}$. Hence $\alpha$ is uniquely determined by $\LL$ and
$$
A\times_{\al, \LL} \N=C^*(A,\LL; \overline{A(\ker \LL|_{MD(\LL)})^{\bot}A}).
$$
In particular, if $\overline{A(\ker \LL|_{MD(\LL)})^{\bot}A}=N_\LL^{\bot}$,  then $
A\times_{\al, \LL} \N=C^*(A,\LL).
$
\end{prop}
\begin{proof}
On one hand, by Proposition \ref{proposition implying interactions},   $\alpha(A)\subseteq \overline{A\alpha(A)A}\cap MD(\LL)$ and $\LL:\alpha(A)\to \LL(A)$ is an isomorphism. On the other hand, since $\LL$ is  faithful on $\overline{A\alpha(A)A}$, the map $\LL:\overline{A\alpha(A)A}\cap MD(\LL)\to \LL(A)$ is an injective homomorphism. This implies that $\alpha(A)=\overline{A\alpha(A)A}\cap MD(\LL)$. Since   $\alpha(A)=\overline{A\alpha(A)A}\cap MD(\LL)$ is an ideal in  $MD(\LL)$ and $\LL:\alpha(A)\to \LL(A)$ is an isomorphism we actually get  $\alpha(A)=(\ker \LL|_{MD(\LL)})^{\bot}$.  Thus, by Proposition \ref{proposition implying interactions}, $\alpha$ is uniquely determined  by $\LL$.  By Theorem \ref{Exel crossed products as relative Pimsner} we get $A\times_{\al, \LL} \N = C^*(A,\LL; \overline{A(\ker \LL|_{MD(\LL)})^{\bot}A})$, and if $\overline{A(\ker \LL|_{MD(\LL)})^{\bot}A}=N_\LL^{\bot}$, then we actually have $A\times_{\al, \LL} \N  =C^*(A,\LL; N_\LL^\bot)=C^*(A,\LL; N_\LL^\bot\cap J(X_\LL))=C^*(A,\LL)$.
\end{proof}

Now we consider Exel systems $(A,\alpha,\LL)$ where $\alpha$ and $\LL$ have somehow equal rights.  Algebras arising from such systems were studied for instance in \cite{Paschke}, \cite{exel2}, \cite{Ant-Bakht-Leb}, \cite{kwa-trans}, \cite{kwa-interact}, \cite{kwa-rever}.
\begin{dfn}
We say that a regular Exel system $(A,\alpha,\LL)$ is a  \emph{corner system}  if  $\alpha(A)$ is a hereditary subalgebra of $A$.
\end{dfn}
The above terminology is justified by Lemma \ref{characterization of corner systems} and Remark \ref{Remarks on reversible systems} below. We note that corner systems $(A,\alpha,\LL)$ satisfy condition $\overline{A\alpha(A)A}\subseteq N_\LL^\bot$. Indeed, since $\alpha(A)$ is hereditary and $\LL$ is faithful on $\alpha(A)$, Lemma \ref{faithful vs almost faithful} implies that $\LL$ is almost faithful on $\overline{A\alpha(A)A}$.

\begin{lem}\label{characterization of corner systems} Let $(A,\alpha,\LL)$ be an Exel system. The following statements are equivalent:
\begin{itemize}
\item[i)] $(A,\alpha, \LL)$ is a corner system.
\item[ii)] $\alpha$ is extendible and
\begin{equation}\label{complete equation}
\alpha(\LL(a))=\overline{\alpha}(1) a \overline{\alpha}(1) \qquad \textrm{ for all }a \in A.
\end{equation}
\item[iii)] $\alpha$ has a  complemented kernel and a corner range;  $\LL$ is a unique regular transfer operator for $\alpha$ and it is given by the formula
\begin{equation}\label{complete transfer operator formula}
\LL(a)=\alpha^{-1}(p a p), \qquad a\in A,
\end{equation}
where $p\in M(A)$ is a projection such that $\alpha(A)=pAp$, and  $\alpha^{-1}$ is the inverse to the isomorphism $\alpha:(\ker\alpha)^{\bot}\to pAp$.
\item[iv)] $\LL(A)$ is  a complemented ideal in $A$ and  $(\ker \LL|_{MD(\LL)})^{\bot}$  is a hereditary subalgebra of $A$ mapped by $\LL$ onto $\LL(A)$;  $\alpha$   is given by the formulas:
\begin{equation}\label{corner endomorphism formula}
\alpha|_{\LL(A)^{\bot}}\equiv 0   \quad \textrm{ and  }\quad \alpha(a)=\LL^{-1}(a), \,\, \textrm{ for } a\in \LL(A),
\end{equation}
where $\LL^{-1}$ is the inverse to the isomorphism $\LL:(\ker \LL|_{MD(\LL)})^{\bot}\to \LL(A)$.
\end{itemize}
If the above equivalent statements hold then  $\alpha(A)=\overline{\alpha}(1) A \overline{\alpha}(1)=(\ker \LL|_{MD(\LL)})^{\bot}$.
\end{lem}
\begin{proof}

i)$\Rightarrow$ii). Let $\{\mu_\lambda\}_{\lambda\in \Lambda}$ be an approximate unit in
$A$.  Since  $\LL$ is isometric on  $\alpha(A)=\alpha(A)A\alpha(A)$,  for any $a\in A$, we have
\begin{align*}
\|\big(\alpha(\mu_\lambda)-\alpha(\mu_{\lambda'})\big)a\|^2&=\|\big(\alpha(\mu_\lambda)-\alpha(\mu_{\lambda'})\big)aa^*\big(\alpha(\mu_\lambda)-\alpha(\mu_{\lambda'})\big)\|\\
&= \|\LL\Big(\big(\alpha(\mu_\lambda)-\alpha(\mu_{\lambda'})\big)aa^*\big(\alpha(\mu_\lambda)-\alpha(\mu_{\lambda'})\big)\Big)\|
\\
&\leq 2 \|\LL(aa^*) (\mu_\lambda - \mu_{\lambda'})\|.
\end{align*}
The last term is arbitrarily small for sufficiently large $\lambda$ and $\lambda'$. Accordingly,  $\{\alpha(\mu_\lambda)\}_{\lambda\in \Lambda}$ is strictly Cauchy and thereby  strictly convergent. Hence $\alpha$ is extendible and we  have  $\alpha(A)=\overline{\alpha}(1) A\overline{\alpha}(1)$. Since  $E(a)=\overline{\alpha}(1) a\overline{\alpha}(1)$ is the unique  conditional expectation onto $\alpha(A)$ we conclude, using Proposition \ref{proposition for rocky},  that  \eqref{complete equation} holds.

ii)$\Rightarrow$iii). Note that \eqref{complete equation} implies that $\alpha(A)=\overline{\alpha}(1) A\overline{\alpha}(1)$ is a corner in $A$.  In particular, $(A,\alpha,\LL)$ is regular because $E(a)=(\alpha\circ\LL)(a)=\overline{\alpha}(1) a\overline{\alpha}(1)$ is a conditional expectation onto $\alpha(A)$. Thus by Proposition \ref{proposition for rocky},  $\ker\alpha$ is complemented and $\LL(a)=\alpha^{-1}(\overline{\alpha}(1) a \overline{\alpha}(1))$ $a\in A$.

iii)$\Rightarrow$iv). Decomposing $A$ into parts $pAp$, $(1-p)Ap$, $pA(1-p)$, $(1-p)A(1-p)$, the map $\LL$ assumes the form
$$
\LL\left(  \begin{array}{c c}
a_{11} & a_{12}\\
 a_{21} & a_{22}
\end{array}\right)
= \alpha^{-1}(a_{11}).
$$
By Proposition \ref{multiplicative domain for ccp maps}, it is immediate that $pAp \oplus (1-p)A(1-p)\subseteq  MD(\LL)$. Moreover,
if $a=a_{12}+ a_{21}\in MD(\LL)$, where $a_{12}\in (1-p)Ap$ and $a_{21}\in pA(1-p)$, then using \eqref{multiplicative domain 1} we get
$$
\alpha^{-1}(a_{12}^*a_{12})=0,\qquad  \alpha^{-1}(a_{21}a_{21}^*)=0.
$$
Since $a_{12}^*a_{12}$ and $a_{21}a_{21}^*$ belong to $pAp$, it follows that $a=a_{12}+ a_{21}=0.$ Hence $MD(\LL)=pAp \oplus (1-p)A(1-p)$. Consequently,   $(\ker \LL|_{MD(\LL)})^{\bot}=pAp$.  Now  the formula \eqref{corner endomorphism formula} is immediate.

iv)$\Rightarrow$i). It follows  from Proposition \ref{proposition implying interactions}.
\end{proof}
\begin{rem}
Transfer operators satisfying \eqref{complete equation} are called   complete transfer operators in  \cite{Ant-Bakht-Leb}, \cite{kwa-trans}, \cite{kwa-interact}, \cite{kwa-rever}. The pair $(A,\alpha)$ where $\alpha$ is an endomorphism satisfying condition iii) of Lemma \ref{characterization of corner systems}, is called a reversible $C^*$-dynamical system in  \cite{kwa-rever}.
\end{rem}

  In the case when $A$ is unital, Exel systems  satisfying \eqref{complete equation} were considered in  \cite{Ant-Bakht-Leb} and \cite{kakariadis}. In particular, it was  shown in \cite[Theorem 4.16]{Ant-Bakht-Leb} that $A\times_{\al, \LL} \N$ is isomorphic to $C^*(A,\alpha)$, and in  \cite[Theorem 3.14]{kakariadis} that $A\times_{\al, \LL} \N$ is isomorphic to $\OO_{M_{\LL}}$. By   Theorem \ref{Exel crossed products as relative Pimsner}, we know that $\OO_{M_{\LL}}\cong \OO(A,\alpha, \LL)= C^*(A,\LL)$. Hence combining these results we get the following isomorphisms $A\times_{\al, \LL} \N \cong C^*(A,\LL)\cong C^*(A,\alpha)$.
We now generalize this fact to the non-unital case.
\begin{thm}\label{complete systems crossed products}
Suppose   $(A,\alpha, \LL)$  is a corner Exel system. Then $\alpha$ and $\LL$ determine each other uniquely and we have
$$
A\times_{\al, \LL} \N =C^*(A,\LL) \cong C^*(A,\alpha).
$$
In particular, $A\times_{\al, \LL} \N$ can be viewed as a  $C^*$-algebra generated by $j_A(A)\cup j_A(A)s$ where $j_A:A\to A\times_{\al, \LL} \N$ is a non-degenerate homomorphism,  $s\in M( A\times_{\al, \LL} \N)$,
\begin{equation}\label{complete relations}
sj_A(a)s^*=j_A(\alpha(a)), \qquad s^*j_A(a)s =j_A(\LL(a)),\qquad a\in A,
\end{equation}
and the pair  $(j_A,s)$ is universal for the pairs with the above properties.
\end{thm}
\begin{proof}
Lemma \ref{characterization of corner systems} implies that  $\alpha$ and $\LL$ determine each other uniquely. By Theorem \ref{Exel crossed products as relative Pimsner},   to prove the equality $A\times_{\al, \LL} \N =C^*(A,\LL)$ it suffices to  show that $A\alpha(A)A=N_\LL^\bot\cap J(X_\LL)$. To this end, we use the  isomorphism  $X_\LL\cong M_\LL$ from Lemma \ref{Exel vs GNS}. For $x\in A$ we have $q(x)=q(x\overline{\alpha}(1))\in M_\LL$. For any $x$, $y$, $z\in A$ we get
\begin{align*}
\Theta_{q(x),q(y)}q(z)&= q(x) \LL(y^*z)= q(x \alpha(\LL(y^*z)))=q(x \overline{\alpha}(1) y^*z\overline{\alpha}(1))= (x\overline{\alpha}(1)y) q(z).
\end{align*}
Thus $\phi(x\overline{\alpha}(1)y)= \Theta_{q(x),q(y)}$. It follows that $\phi$ sends $\overline{A\alpha(A)A}=\overline{A\overline{\alpha}(1)A}\subseteq N_\LL^\bot=(\ker\phi)^\bot$ isometrically onto $\K(X_\LL)\cong \K(M_\LL)$. Hence $\overline{A\alpha(A)A}=J(X_\LL)\cap N_{\LL}^\bot$ and we have $A\times_{\al, \LL} \N =C^*(A,\LL)$.

 In order to show that the first relation in \eqref{complete relations} holds in $A\times_{\al, \LL} \N =C^*(A,\LL)$ (the second holds trivially) it suffices to check that $(i_A(\alpha(a)), ti_A(a)t^*)$, for  $a\in A$, is a redundancy of the Toeplitz representation $(i_A,t)$ (note that  $\alpha(A)\subseteq N_\LL^\bot$).  Invoking  Proposition \ref{destroying proposition}  we have $ti_A(a)=i_A(\alpha(a))t$. Thus, using \eqref{complete equation}, for any $b,c \in A$ we get
\begin{align*}
(ti_A(a) t^*)\, i_A(b)t\,i_A(c)&= t i_A(a \LL(b)c)=i_A(\alpha(a \LL(b)c))t=i_A(\alpha(a)b\alpha(c)) t
\\
&=i_A(\alpha(a)) \, i_A(b)t\,i_A(c).
\end{align*}
Since $i_A$ is non-degenerate this shows that $(i_A(\alpha(a)), ti_A(a)t^*)$ is a redundancy and thus \eqref{complete relations} holds. Moreover, since the ideal $J(X_\LL)\cap N_{\LL}^\bot=\overline{A\alpha(A)A}$ is generated by $\alpha(A)$ we see that the kernel of the quotient map $\TT(A,\LL)\to C^*(A,\LL)$ is the ideal generated by differences  $i_A(\alpha(a)) - ti_A(a)t^*$, $a\in A$. Hence $(C^*(A,\LL),j_A,s)$ is universal with respect to relations \eqref{complete relations}, cf. Proposition \ref{universal description of crossed products}.
By \cite[Proposition 4.6]{kwa-rever}, cf. Proposition \ref{Proposition for endomorphisms}, $C^*(A,\alpha)$ is universal with respect to the same relations and thus   $C^*(A,\alpha)\cong C^*(A,\LL)$.
\end{proof}
\begin{rem}\label{Remarks on reversible systems}  If $A$ is unital then an Exel system $(A,\alpha,\LL)$ is a corner system if and only if $(\alpha,\LL)$ is a corner interaction studied in \cite{kwa-interact}, see  Proposition \ref{Ex2.3}v). In particular,  the isomorphism $C^*(A,\LL)\cong C^*(A,\alpha)$ is an instance of  the isomorphism \eqref{interacion isomorphisms}. An examination of the argument leading to \eqref{interacion isomorphisms} shows that it  holds also in the non-unital case if one defines a corner interaction as an interaction $(\VV,\HH)$ over $A$ where both $\VV$ and $\HH$ are extendible and have corner ranges. Thus corner interactions give a symmetrized framework for corner Exel systems, and one could think of them as partial automorphism of $A$ whose domain and range are corners in $A$.
\end{rem}

\section{Graph $C^*$-algebras as crossed products by completely positive maps}\label{graphs section}
In this section, we test  Exel's construction  and the results of the present paper against the original idea standing behind \cite{exel2} that Cuntz-Krieger
 algebras (or more generally graph $C^*$-algebras) could be viewed as  crossed products  associated to topological Markov shifts. We recall Brownlowe's \cite{Brownlowe} realization of  graph $C^*$-algebras $C^*(E)$ as Exel-Royer's crossed product for   partially defined  Exel system $(\D_E,\alpha,\LL)$ and discuss when the maps $\alpha$ and $\LL$ can be extended to the whole of diagonal algebra $\D_E$. This  leads us to a complete description of Perron-Frobenious operators on $\D_E$ associated to quivers on $E$. We prove that the  crossed product of $\D_E$ by any such operator is  isomorphic to  $C^*(E)$.

\subsection{Graph $C^*$-algebras as Exel-Royer's crossed products}
For graphs and their $C^*$-algebras we use the notation and conventions of \cite{Raeburn}, \cite{brv}, \cite{hr}.
Throughout this section, we fix an arbitrary countable directed graph $E = (E^0,E^1, r, s)$. Hence   $E^0$ and $E^1$ are countable sets  and  $r,s:E^{1}\to E^0$ are arbitrary  maps.  We denote by  $E^n$, $n>0$,  the set of finite paths $\mu=\mu_1...\mu_n$ satisfying $s(\mu_i)=r(\mu_{i+1})$,
 for all $i=1,...,n$. Then    $|\mu|=n$ stands for the length of $\mu$ and    $E^*=\bigcup_{n=0}^\infty E^n$ is the set of all finite paths (vertices are treated as paths of length zero).  We put  $E^\infty$ to be the set of infinite paths. The  maps $r,s$ extend naturally to $E^*$ and $r$ extends also to $E^\infty$.

The \emph{graph $C^*$-algebra} $C^*(E)$ is generated by a universal Cuntz-Krieger $E$-family consisting of partial isometries $\{s_e: e\in E^1\}$ and mutually  orthogonal projections $\{p_v: v\in E^1\}$ such that  $s_e^*s_e=p_{s(e)}$, $s_e s_e^*\leq p_{r(e)}$   and $p_v=\sum_{r(e)=v} s_e s_e^*$ whenever the sum is finite (i.e. $v$ is a  finite receiver). It follows that $C^*(E)=\clsp\{s_\mu s_\nu^*: \mu, \nu \in E^*\}$ where $s_\mu:=s_{\mu_1} s_{\mu_2}....s_{\mu_n}$ for $\mu=\mu_1...\mu_n\in E^n$, $n>0$, and $s_\mu=p_{\mu}$ for $\mu\in E^0$. We denote by  $\D_E:=\clsp\{s_\mu s_\mu^*: \mu \in E^*\}$ the \emph{diagonal  $C^*$-subalgebra} of $C^*(E)$.

It is attributed to folklore, see \cite{fmy} or \cite{Webster}  for an extended discussion, that the Gelfand spectrum of $\D_E$ can be identified with the boundary space of $E$. To be more specific, we define $E^*_{inf}:=\{\mu\in E^*: |r^{-1}(s(\mu))|=\infty\}$ and $E^*_s:=\{\mu\in E^*: r^{-1}(s(\mu))=\emptyset\}$, so $E^*_{inf}$ is the set of paths that start in infinite receivers, and $E^*_s$ is the set of paths that start in sources. For any $\eta\in E^*\setminus E^0$ let ${\eta}E^{\leq \infty}:=\{\mu=\mu_1... \in E^*\cup E^\infty: \mu_1...\mu_{|\eta|}=\eta\}$ and for $v\in E^0$ put ${v}E^{\leq \infty}:=\{\mu \in E^*\cup E^\infty: r(\mu)=v\}$. The \emph{boundary space} of $E$, cf. \cite[Section 2]{Webster} or \cite[Subsection 4.1]{Brownlowe}, is the set
$$
\partial E:= E^\infty\cup E^*_{inf} \cup E^*_{s}
$$
equipped with the topology generated by the `cylinders' $D_\eta:=\partial E\cap {\eta}E^{\leq \infty}$, $\eta \in E^*$, and their complements. In fact, the sets   $D_\eta \setminus \bigcap_{\mu \in F} D_\mu$, where $\eta\in E^*$ and $F\subseteq {\eta}E^{\leq \infty}\cap E^*$ is finite, form a basis of compact and open sets for Hausdorff topology on $\partial E$,   \cite[Section 2]{Webster} or  \cite[Section 2]{Brownlowe}.
Passing to a dual description of the assertion in \cite[Theorem 3.7]{Webster} we get the following proposition.
\begin{prop}\label{isomorphism proposition}
We have an isomorphism $\D_E\cong C_0(\partial E)$ determined by the formula
\begin{equation}\label{isomorphism proposition fromula}
s_\mu s_\mu^*  \longmapsto  \chi_{D_\mu } ,  \qquad \mu \in E^*.
\end{equation}
\end{prop}

  The one-sided \emph{topological Markov shift}  associated to $E$ is the map $\sigma:\partial E\setminus E^0 \to \partial E$  defined, for $\mu=\mu_1\mu_2...\in \partial E\setminus E^0$,  by the formulas
$$
\sigma(\mu):=\mu_2\mu_3...\,\, \textrm{ if }\,\,\mu \notin E^1, \quad \textrm{ and} \quad \sigma(\mu):=s(\mu_1)\,\, \textrm{ if }\,\,\mu=\mu_1 \in E^1.
$$
By \cite[Proposition 2.1]{Brownlowe} the shift $\sigma$ is a local homeomorphism. Furthermore,  results of \cite[Propositions 2.1 and 4.4]{Brownlowe}  imply the following  proposition (we adopt the convention that a sum over the empty set is zero).
\begin{prop}[Brownlowe] \label{Brownlowe proposition}
The formulas
\begin{equation}\label{composition with shift}
\alpha(a)(\mu)=a(\sigma(\mu)), \qquad \LL(a)(\mu)=\sum\limits_{\nu \in \sigma^{-1}(\mu)} a(\nu) 
\end{equation}
define respectively a homomorphism $\alpha:C_0(\partial E) \to  M(C_0(\partial E\setminus E^0))$ and a linear map
 $\LL:C_c(\partial E\setminus E^0)\to C_c(\partial E)$. Moreover, the triple $(C_0(\partial E), \alpha, \LL)$ forms a $C^*$-dynamical system in the sense of \cite[Definition 1.2]{br}, and   we have an isomorphism
$$
\OO(C_0(\partial E), \alpha, \LL)\cong C^*(E).
$$
 \end{prop}

The above mappings \eqref{composition with shift} have the following important algebraic description.
The isomorphism $\D_E\cong C_0(\partial E)$ from Proposition \ref{isomorphism proposition} gives rise to $*$-isomorphisms  $M(\clsp\{s_\mu s_\mu^*: \mu \in E^*\setminus \{0\}\})\cong M(C_0(\partial E\setminus E_0))$ and $\textrm{span}\{s_\mu s_\mu^*: \mu \in E^*\setminus E^0\}\cong C_c(\partial E)$. Using these isomorphisms the mappings in \eqref{composition with shift} are intertwined respectively with a homomorphism $\Phi:\D_E \to  M(\clsp\{s_\mu s_\mu^*: \mu \in E^*\setminus E^0\})$ and a linear map $\Phi_*:\textrm{span}\{s_\mu s_\mu^*: \mu \in E^*\setminus E^0\}\to \textrm{span}\{s_\mu s_\mu^*: \mu \in E^*\}$
which are given by the formulas
\begin{equation}\label{composition with shift algebraic}
\Phi(a) =\sum_{e\in E^1} s_e a s_e^* , \qquad \Phi_*(a)=\sum_{e\in E^1} s_e^* a s_e,
\end{equation}
where the first sum is strictly convergent and the second is finite.
(It suffices to check it on the spanning elements $s_\mu s_\mu^*$  and $\chi_{D_\mu }$, $\mu \in E^*$, which we leave to the reader.)
 When $E$ has no infinite emitters (see Proposition \ref{non-commutative Markov shifts equivalences} below) the formula  $\Phi(a) =\sum_{e\in E^1} s_e a s_e^*$ defines a  self-map on the whole of the graph $C^*$-algebra $C^*(E)$. In the literature, this mapping, usually considered when   $E$ is locally finite (i.e. $r$ and $s$ are finite-to-one), is  called
a \emph{non-commutative  Markov shift} and its ergodic properties are well studied, cf., for instance, \cite{jp}.

\subsection{Non-commutative Perron-Frobenius operators arising from quivers}\label{Non-commutative Perron-Frobenius subsection}
The mappings $\alpha$ and $\LL$ considered in Proposition \ref{Brownlowe proposition} are  viewed as partial mappings on $C_0(\partial E)$, cf.  \cite{er}, \cite{Brownlowe}. Now we  discuss the problem of when the formulas \eqref{composition with shift}, or their analogues, define honest mappings on  $C_0(\partial E)$.

\begin{prop}\label{non-commutative Markov shifts equivalences}
The following conditions are equivalent:
\begin{itemize}
\item[i)] the first of formula in \eqref{composition with shift} defines an endomorphism $\alpha:C_0(\partial E)\to C_0(\partial E)$,
\item[ii)] $\sigma:\partial E\setminus E^0 \to \partial E$ is a proper map (preimage of a compact set is compact),
\item[iii)] $\sigma$ is a finite-to-one mapping,
\item[iv)] there are no infinite emitters in $E$,
\item[v)] the sum $\sum_{e\in E^1} s_e a s_e^*$ converges in norm for every $a\in C^*(E)$,
\item[vi)]  the range of the homomorphism $\Phi$ is contained in $\clsp\{s_\mu s_\mu^*: \mu \in E^*\setminus E^0\}\subseteq \D_E$, and hence $\Phi:\D_E\to \D_E$ is an endomorphism.
\end{itemize}
In particular, if the above equivalent conditions hold, then the first formula in \eqref{composition with shift algebraic} defines a completely positive map $\Phi: C^*(E) \to  C^*(E)$ which restricts to an endomorphism $\Phi:\D_E\to \D_E$.
\end{prop}
\begin{proof}

i)$\Leftrightarrow$ii). It is a well known general fact  that  a continuous  mapping $\tau:X\to Y$ between locally compact Hausdorff spaces $X,Y$, gives rise to the composition operator from $C_0(Y)$   to $C_0(X)$ (rather than to $C_b(X)=M(C_0(X))$) if and only if $\tau$ is proper.

ii)$\Rightarrow$iii).
If $\sigma$ is a proper local homeomorphism  then $\sigma^{-1}(\mu)$,  $\mu \in \partial E$, is compact and cannot have a cluster point.
Hence $\sigma$  is finite-to-one.

iii)$\Rightarrow$iv).  It    follows  readily from the definition of $\sigma$.

iv)$\Rightarrow$v).  Consider a net,  indexed by finite sets $F\subseteq E^1$ ordered by inclusion, consisting of mappings $\alpha_F:C^*(E)\to C^*(E)$ given by $\alpha_F(a):=\sum_{e\in F} s_e as_e^*$.
Since the projections $s_es_e^*$, $e\in F$, are mutually orthogonal we get $\|\alpha_F(a)\|=\max_{e\in F} \|s_e as_e^*\|\leq \|a\|$, and thus $\alpha_F$ is a contraction. Let $a\in C^*(E)$. For any $\varepsilon>0$ there is a finite linear combination $b=\sum_{\mu,\nu \in K} c_{\mu,\nu} s_\mu s_\nu^*$  such that $\|a-b\|\leq \varepsilon$  ($K\subseteq E^*$ is finite set). Since $E$ has no infinite emitters  the set
$$
F=\{e\in E^1: s(e)=r(\mu)\textrm{ for some } \mu \in K\}=\bigcup_{v\in r(K)} s^{-1}(v)
$$
is finite. Clearly, for any finite set $F'\subseteq E^*$ containing $F$ we have $\alpha_{F'}(b)=\alpha_{F}(b)$. Thus
$$
\|\alpha_{F'}(a)-\alpha_F(a)\|\leq \| \alpha_{F'}(a)-\alpha_{F'}(b)\|+\|\alpha_{F}(b)-\alpha_{F}(a)\|\leq 2\varepsilon.
$$
Hence the net $\{\alpha_F(a)\}_F$ is  Cauchy and the sum $\sum_{e\in E^1} s_e a s_e^*$ converges in norm.

v)$\Rightarrow$vi). It is straightforward.

vi)$\Rightarrow$i). Note that the isomorphism $\D_E\cong C_0(\partial E)$ given by \eqref{isomorphism proposition fromula} intertwines $\Phi$ and $\alpha$.
\end{proof}

One can check that  the second formula in \eqref{composition with shift} defines a mapping $\LL:C_c(\partial E)\to C_c(\partial E)$ if and only if $E$ has no infinite receivers. But even if the graph $E$ is locally finite, this  mapping  might be unbounded.
On the other hand, if $E$ is locally finite,  we can adjust the formula for $\LL$ by adding averaging as  in \eqref{composition with shift preliminary},  and then $\LL$ has norm one, so in particular it extends to a self-map of $C_0(\partial E)$.
This motivates us to consider slightly more general averagings, which will allow us to get a bounded positive operator on $C_0(\partial E)$ for arbitrary graphs. Accordingly, we wish to consider strictly positive numbers $\lambda=\{\lambda_e\}_{e\in E^1}$ such that the formula
\begin{equation}\label{general formula for L of a graph}
 \LL_\lambda(a)(\mu)=\sum\limits_{e \in E^1,\, e\mu \in \partial E} \lambda_{e}\, a(e\mu)
\end{equation}
defines a mapping on $C_0(\partial E)$.  We note that fixing the family $\{\lambda_e\}_{e\in E^1}$ is equivalent to fixing  a  system   of  measures $\{\lambda_v\}_{v\in E^0}$ on $E^1$ making the graph $E$ into a (topological) quiver. Indeed,  the relation  $\lambda_e=\lambda_{s(e)}(\{e\})$ establishes a one-to-one correspondence between families  $\{\lambda_e\}_{e\in E^1}$ of strictly positive numbers and $s$-systems of measures $\{\lambda_v\}_{v\in E^0}$ on $E$, cf. Definition  \ref{topological quiver defn}.
In particular, if $E$ has no infinite emitters one can put $\lambda_e:=|s^{-1}(s(e))|^{-1}$, $e\in E^1$, which corresponds to the situation where all the measures $\{\lambda_v\}_{v\in E^0}$ are uniform probability distributions. In this case one recovers from \eqref{general formula for L of a graph} the second formula in \eqref{composition with shift preliminary}.

\begin{prop}\label{proposition for lambda transfer operators}
Let $\lambda=\{\lambda_e\}_{e\in E^1}$ be a family of strictly positive numbers. The following conditions are equivalent:
\begin{itemize}
\item[i)] the formula \eqref{general formula for L of a graph} defines a bounded operator $\LL_\lambda: C_0(\partial E)\to C_0(\partial E)$,
\item[ii)] the following conditions are satisfied:
\begin{equation}\label{conditions for convergence}
\left\{\sum_{e\in s^{-1}(v)} \lambda_e\right\}_{v\in s(E^1)}  \in \ell_\infty\left(s(E^1)\right),
\end{equation}
\begin{equation}\label{conditions for convergence2}
\left\{\sum_{e\in r^{-1}(v)\cap s^{-1}(w)} \lambda_e\right\}_{w\in s(r^{-1}(v))} \in c_0\left(s(r^{-1}(v))\right) \quad \textrm{for all } v\in r(E^1),
\end{equation}
\item[iii)] the sum
\begin{equation}\label{definition of u}
u_\lambda:=\sum\limits_{e\in E^1} \sqrt{\lambda_e} \, s_{e}
\end{equation}
converges  strictly  in $M(C^*(E))$,
 \item[iv)]  the sum $\sum_{e,f\in E^1} \sqrt{\lambda_e\lambda_f }s_e^* a s_f$ converges in norm for every $a\in C^*(E)$ and
\begin{equation}\label{general formula for L of a graph2}
 \Phi_{*,\lambda}(a):=\sum_{e,f\in E^1} \sqrt{\lambda_e\lambda_f }s_e^* a s_f, \qquad a\in C^*(E),
\end{equation}
defines a completely positive map $\Phi_{*,\lambda}: C^*(E)\to C^*(E)$.
\end{itemize}
If the above equivalent conditions hold then  $
\Phi_{*,\lambda}(a)=u_\lambda^* a u_\lambda$, $a\in C^*(E)$,
and the isomorphism $\D_E\cong C_0(\partial E)$ from Proposition \ref{isomorphism proposition} intertwines $\Phi_{*,\lambda}|_{\D_E}$ and $\LL_\lambda$.
\end{prop}
\begin{proof}
i)$\Rightarrow$ii). One readily sees that
\begin{equation}\label{relation determining L_lambda 1}
\LL_\lambda (\chi_{D_\eta})=\lambda_{\eta_1} \chi_{D_{\sigma(\eta)}}, \quad \textrm{ for any }\eta=\eta_1... \in \partial E \setminus E^{0}.
\end{equation}
Hence for any   $v\in s(E^1)$ and  any finite set  $F \subseteq s^{-1}(v)$ we get
$$
\sum_{e\in F} \lambda_e =\|\left(\sum_{e\in F}\lambda_e\right)  \chi_{D_v}\| =\|\LL_\lambda \left(\sum_{e\in F}\chi_{D_e}\right) \| \leq \|\LL_\lambda\|,
$$
which implies condition \eqref{conditions for convergence}. Now let $v\in r(E^1)$ and note that, for any $\mu \in \partial E$,
\begin{align*}
\LL_\lambda (\chi_{D_v})(\mu)&
=\sum\limits_{e \in r^{-1}(v),\, e\in s^{-1}(r(\mu))} \lambda_{e}\, \chi_{D_v}(e\mu)
=\sum\limits_{e \in r^{-1}(v)} \lambda_{e}\, \chi_{D_{s(e)}}(\mu)
\\
&=\sum_{w\in s(r^{-1}(v))}   \sum_{e\in r^{-1}(v)\cap s^{-1}(w)} \lambda_{e}\, \chi_{D_{w}}(\mu).
\end{align*}
Since the sets $D_{w}$ are disjoint and  open, $\LL_\lambda (\chi_{D_v})\in C_0(\partial E)$ implies condition \eqref{conditions for convergence2}.
For future reference, note that in view of the above calculation we have (we treat empty sums as zero)
\begin{equation}\label{relation determining L_lambda 2}
\LL_\lambda (\chi_{D_v}) = \sum\limits_{e \in r^{-1}(v)} \lambda_{e}\, \chi_{D_{s(e)}},\qquad  v\in  E^0.
\end{equation}

ii)$\Rightarrow$iii).
Let $v\in  s(E^1)$. For a finite set $F\subseteq s^{-1}(v)$  we have
$
\|\sum_{e \in F} \sqrt{\lambda_e} s_e\|^2=\|\sum_{e \in F} \lambda_e p_v \|=\sum_{e \in F} \lambda_e$.
Since $\sum_{e\in s^{-1}(v)} \lambda_e <\infty$, by \eqref{conditions for convergence}, it follows that the sum $u_v:=\sum_{e \in s^{-1}(v)} \sqrt{\lambda_e} s_e$ converges  in  norm. Thus for any finite set $F\subseteq s(E^1)$ we have
\begin{equation}\label{definition of u_F}
u_F:=\sum\limits_{v\in F} u_v=\sum\limits_{v\in F}\sum\limits_{e\in s^{-1}(v)} \sqrt{\lambda_e} \, s_{e}=\sum\limits_{e\in s^{-1}(F)} \sqrt{\lambda_e} \, s_{e}\in C^*(E).
\end{equation}
By  \eqref{conditions for convergence},  $M:=\sup_{v\in s( E^1)}\, \sum_{e\in s^{-1}(v)} \lambda_e$  is finite.  The set of elements $u_F$ is bounded:
\begin{equation}\label{norm estimation}
\|u_F\|^2=\|u_F^*u_F\|=\|\sum\limits_{v\in F} \sum\limits_{e\in s^{-1}(v)} \lambda_e  p_v\|=\max_{v\in F}  \sum\limits_{e\in s^{-1}(v)} \lambda_e \leq M.
\end{equation}
Condition \eqref{conditions for convergence2}  implies that for any $v\in E^0$ the sum $\sum_{e \in r^{-1}(v)} \sqrt{\lambda_e} s_e$ converges in norm. Indeed, for any  finite set  $F\subseteq r^{-1}(v)$ we have
$$
\|\sum_{e \in F} \sqrt{\lambda_e} s_e\|^2=\|\sum\limits_{e\in F} \lambda_e  p_{s(e)}\|=\max_{w\in s(r^{-1}(v))} \sum_{e\in r^{-1}(v)\cap s^{-1}(w)\cap F} \lambda_e,
$$
which by \eqref{conditions for convergence2} can be made arbitrarily small by choosing $F$ lying outside a sufficiently large finite subset of $r^{-1}(v)$.

Now fix a (nonzero) finite linear combination $a=\sum_{\mu,\nu \in K} \lambda_{\mu,\nu}s_{\mu} s_{\nu}^*$,  where  $K\subseteq E^*$ is finite.
Since we know, by \eqref{norm estimation}, that $\|\sum\limits_{e\in F} \sqrt{\lambda_e} \, s_{e}\|\leq \sqrt{M}$ for every finite $F\subseteq E^1$, to prove the strict convergence of the sum in \eqref{definition of u} it suffices to check the convergence in norm of the two series
$
 \sum\limits_{e\in E^1} \sqrt{\lambda_e} \, s_{e}a$ and $\sum\limits_{e\in E^1} \sqrt{\lambda_e} \, s_{e}^*a$.

Firstly,  note that for   $v\in E^0$ we have  $
u_v a = 0$ unless $v\in r(K)$. Hence for any finite set $F\subseteq  s(E^1)$ containing $r(K)$ we get $u_{F}a=u_{r(K)}a$. Recall, see  \eqref{definition of u_F}, that  $u_{r(K)}=\sum_{e\in   s^{-1}(r(K))} \sqrt{\lambda_e}s_e$ converges in norm. Therefore, for any  $\varepsilon >0$ there is a finite set $F_0\subseteq s^{-1}(r(K))\cap E^1$ such that for any finite $F\subseteq E^1$ disjoint with $F_0$  we have
$$
\| \sum_{e\in F } \sqrt{\lambda_e}s_e a\|= \| \sum_{e\in F\cap  s^{-1}(r(K))} \sqrt{\lambda_e}s_e a\|\leq \varepsilon.
$$
This means that  the sum $\sum\limits_{e\in E^1} \sqrt{\lambda_e} \, s_{e}a$ converges in $C^*(E)$.
\\
Secondly,  note that for   $e\in E^1$ we have $
s_e^* a = 0$ unless $e\mu \in K$ for some $\mu\in E^*$, or $r(e)\in K \cap E^0$.
 Recall that the sum $\sum_{e\in r^{-1}(v) } \sqrt{\lambda_e}s_e^* $ is norm convergent for all $v\in E^0$. Thus for a fixed  $\varepsilon>0$  we can find a finite set $F_1\subseteq E^1$ such that for any $F$ disjoint with $F_1$ we have
$$
\|\sum_{e\in r^{-1}(v)\cap F } \sqrt{\lambda_e}s_e^*\|\leq \frac{\varepsilon}{|K\cap E^0|\cdot \|a\| } \quad  \textrm{ for all }v\in K\cap E^0.
$$
Then for any finite $F\subseteq E^1$ lying outside  the finite set $F_0:=\{e\in E^1:   e\mu\in K, \mu\in E^* \}\cup F_1$  we get
\begin{align*}
\|\sum_{e\in F } \sqrt{\lambda_e}s_e^*a\| &=\| \sum_{v\in K\cap E^0}\sum_{e\in r^{-1}(v)\cap F } \sqrt{\lambda_e}s_e^*  a\|
\\
&\leq \sum_{v\in K\cap E^0} \|\sum_{e\in r^{-1}(v)\cap F }  \sqrt{\lambda_e}s_e^*  \| \cdot \|a\|\leq \varepsilon.
\end{align*}
Thus  $\sum\limits_{e\in E^1} \sqrt{\lambda_e} \, s_{e}^*a$ converges in $C^*(E)$.
This shows that  $u_\lambda=\sum\limits_{e\in E^1} \sqrt{\lambda_e} \, s_{e}$ converges in strict topology in $M(C^*(E))$.

iii)$\Rightarrow$iv). Plainly, as the  sum $u_\lambda=\sum\limits_{e\in E^1} \sqrt{\lambda_e} \, s_{e}$ is strictly convergent   the sum $u_\lambda^* a u_\lambda =\sum_{e,f\in E^1} \sqrt{\lambda_e\lambda_f }s_e^* a s_f$ converges in norm for every $a\in C^*(E)$.

iv)$\Rightarrow$i). Using relations \eqref{relation determining L_lambda 1}, \eqref{relation determining L_lambda 2} one readily verifies that the isomorphism given by \eqref{isomorphism proposition fromula} intertwines the restriction  $\Phi_{*,\lambda}|_{\D_E}$ of  $\Phi_{*,\lambda}$ to $\D_E$ with a mapping $\LL_\lambda:C_0(\partial E)\to C_0(\partial E)$ given by   \eqref{general formula for L of a graph}.
  \end{proof}
\begin{rem}\label{remark on probability distribution relation}  For $u_\lambda$ given by \eqref{definition of u} we have $u_\lambda^*u_\lambda=\sum_{v\in E^{0}}(\sum_{e\in s^{-1}(v)}\lambda_e)p_v$. Hence  $u_\lambda$ is a partial isometry if and only if the measures $\{\lambda_v\}_{v\in E^0}$ arising from $\lambda=\{\lambda_e\}_{e\in E^1}$ are normalized, that is if and only if
\begin{equation}\label{probability distribution}
\sum_{e\in s^{-1}(v)} \lambda_e =1,\,\,\,\,\, \textrm{ for all }\,\,\,\, v \in E^0.
\end{equation}
Clearly,  \eqref{probability distribution} implies \eqref{conditions for convergence} and if no vertex in $E$ receives  edges from infinitely many vertices then \eqref{conditions for convergence2} is trivial.  So in this case  $u_\lambda$ can be chosen to be a partial isometry. Nevertheless, in general  there might be no  systems satisfying \eqref{probability distribution} for which the sum   \eqref{definition of u} is strictly convergent (e.g. consider the infinite countable graph with a vertex receiving  one edge from each of the remaining ones). If $E$ is locally finite, one can let   $\lambda_v$, $v\in E^0$, to be uniform probability distributions by putting $\lambda_e:=|s^{-1}(s(e))|^{-1}$, $e\in E^1$. In the latter case and under the assumption that $E$ has no sinks or sources it was noted  implicitly in \cite[Theorem 5.1]{brv} and  explicitly in \cite[Section 5]{hr} that the  formula \eqref{definition of u} defines an isometry in  $M(C^*(E))$.
A detailed  discussion of history and analysis of operators \eqref{definition of u}, \eqref{general formula for L of a graph2} associated to systems of uniform probability measures for arbitrary finite graphs can be found in \cite{kwa-interact}.
\end{rem}

Let us note that $\Phi_{*,\lambda}$, given by  \eqref{general formula for L of a graph2},   restricted to $\D_E$ assumes the form
\begin{equation}\label{formula for Phi_lambda on core of a graph}
\Phi_{*,\lambda}(a)=\sum_{e\in E^1} \lambda_e s_e^* a s_e, \qquad a\in \D_E.
\end{equation}
In particular, in view of the last part of Proposition \ref{proposition for lambda transfer operators}, it is natural to call  $\Phi_{*,\lambda}: C^*(E)\to C^*(E)$ the \emph{non-commutative Perron-Frobenius operator} associated to the quiver  $(E^1,E^0,r,s, \lambda)$.

\subsection{Graph $C^*$-algebras  as crossed products $C^*(\D_E,\LL)$}
Now we are ready to state and  prove the  main result of this section. In previous subsections we have shown that for  positive numbers $\lambda=\{\lambda_e\}_{e\in E^1}$ satisfying \eqref{conditions for convergence}, \eqref{conditions for convergence2}  we have two mappings $\LL_\lambda:C_0(\partial E)\to C_0(\partial E)$ and $\Phi_{*,\lambda}:\D_E\to \D_E$, given respectively by \eqref{general formula for L of a graph} and \eqref{formula for Phi_lambda on core of a graph}. These mappings are intertwined by the isomorphism $C_0(\partial E)\cong \D_E$ determined by \eqref{isomorphism proposition fromula}. Thus one could express the following statement equally well in terms of  $(C_0(\partial E), \LL_\lambda)$ or $(\D_E,\Phi_{*,\lambda})$. We choose the second system, as it is more convenient for our proofs. In order to shorten the  notation we denote $\Phi_{*,\lambda}$ simply by $\LL$.

\begin{thm}\label{thm for Raeburn, Brownlowe and Vitadello}
 Suppose  $E=(E^0,E^1,s,r)$ is an arbitrary  directed graph and  choose the numbers $\lambda_e>0$, $e\in E^1$, such that the conditions \eqref{conditions for convergence}, \eqref{conditions for convergence2} hold. Then
  the sum
\begin{equation}\label{first approximation of a transfer operator}
\LL(a):=\sum_{e\in E^1} \lambda _e  s_e^* a s_e, \qquad a\in \D_E,
\end{equation}
is convergent in norm and defines  a (completely) positive map $\LL:\D_E\to \D_E$   such that
  $$
C^*(E)\cong C^*(\D_E, \LL),
$$
with the isomorphism determined by $a \mapsto j_{\D_E}(a)$,   $ a u_\lambda \mapsto   j_{\D_E}(a) s$,   $a\in \D_E$, where $u_\lambda$ is given by the strictly convergent sum \eqref{definition of u}. Further under these assumptions:
\begin{itemize}
\item[i)] If  $E$ has no infinite emitters, then  the following sum is convergent in norm:
\begin{equation}\label{Markov shift for local finite}
\alpha(a):=\sum_{e\in E^1}  s_e a s_e^*, \qquad a\in \D_E.
\end{equation}
It defines an endomorphism such that $(\D_E,\alpha,\LL)$ is an Exel system and
$$
 C^*(\D_E, \LL)= \D_E\rtimes_{\alpha,\LL} \N.
$$
Moreover,  $(\D_E,\alpha,\LL)$  is a regular Exel system if and only if \eqref{probability distribution} holds.
\item[ii)]   If  $E$ has no infinite receivers then $\LL$ is a transfer operator for a certain endomorphism $\alpha$ if and only if $E$ is locally finite. In this event $\alpha$ given by \eqref{Markov shift for local finite} is a unique endomorphism such that  $(\D_E,\alpha,\LL)$ is an Exel system and $\LL$ is faithful on $\alpha(\D_E)\D_E$.
\item[iii)] If  $E$ is locally finite and without sources then  $\LL$ is faithful and $\alpha$  given by \eqref{Markov shift for local finite}  is a unique endomorphism such that $(\D_E,\alpha,\LL)$ is an Exel system.
\end{itemize}
\end{thm}
\begin{rem} We  comment on the corresponding items in the above theorem:

i). Recall that \eqref{probability distribution} holds if and only if the operator $u_\lambda$  is a partial isometry. In particular, the general question for which graphs $E$ the numbers $\lambda_e>0$, $e\in E^1$, can be chosen so that  the  Exel system $(\D_E,\alpha,\LL)$  is  regular, seems to be a complex problem.

ii).  One could conjecture that in general $\LL$ is a transfer operator for a certain endomorphism if and only if $E$ has no infinite emitters, and then this endomorphism is the (non-commutative) Markov shift given by \eqref{Markov shift for local finite}.

iii). If  $E$ is locally finite and without sources then $\partial E=E^\infty$ and we can put $\lambda_e:=|s^{-1}(s(e))|^{-1}$, $e\in E^1$. In this case, identifying $\D_E$ with $C_0(E^\infty)$,  the mappings \eqref{Markov shift for local finite} and \eqref{first approximation of a transfer operator} coincide with those given by \eqref{composition with shift preliminary}. In particular, Theorem \ref{thm for Raeburn, Brownlowe and Vitadello} yields an isomorphism
$$
C^*(E)\cong C_0(E^\infty)\rtimes_{\alpha,\LL} \N
$$
proved  by Brownlowe in \cite[Proposition 4.6]{Brownlowe}, and when $E$ has  no sinks by Brownlowe, Raeburn and Vitadello in \cite[Theorem 5.1]{brv}.
\end{rem}

 The proof of Theorem \ref{thm for Raeburn, Brownlowe and Vitadello} will rely on the  following two lemmas.  We fix the notation from the assertion of Theorem \ref{thm for Raeburn, Brownlowe and Vitadello} and note that the map $\LL:\D_E\to \D_E$ is well defined by   Proposition \ref{proposition for lambda transfer operators}.
We denote by $E^0_{s}:=\{v\in E^0: r^{-1}(v)=\emptyset\}$ and $E_{inf}:=\{v\in E^0: |r^{-1}(v)|=\infty\}$   the set of sources and the set of infinite receivers, respectively.
\begin{lem}\label{description of the relevant ideals}
Let $X_\LL$ be the $C^*$-correspondence of $(\D_E, \LL)$. We have
$
N_\LL^\bot= \clsp\{ s_\mu s_\mu^*: \mu \in E^*\setminus E^0_{s}\}$ and $
J(X_\LL)=\clsp\{ s_\mu s_\mu^*:  \mu\in E^*\setminus  E^0_{inf}\}$. Hence
$$
N_\LL^\bot \cap J(X_\LL)=\clsp\{ s_\mu s_\mu^*:  \mu\in E^*\setminus  E^0\}.
$$
\end{lem}
\begin{proof}
 Note that $\D_E$ is a direct sum of two complemented ideals $\clsp\{ p_v : v\in E^{0}_{s}\}$ and $\clsp\{ s_\mu s_\mu^*: \mu \in E^*\setminus E^0_{s}\}$. One readily sees  that  $\LL$  vanishes on the first one  and is  faithful on the second one. Hence  $N_\LL=\clsp\{ p_v : v\in E^{0}_{s}\}$ and  $N_\LL^\bot=  \clsp\{ s_\mu s_\mu^*: \mu \in E^*\setminus E^0_{s}\}$.

Let    $\mu\in E^*\setminus E^0$ and put
$
K:=   \lambda_{\mu_1}^{-1} \Theta_{(s_{\mu } s_{\mu}^*\otimes 1),(s_{\mu } s_{\mu }^*\otimes 1)}
$
where $\mu_1\in E^1$ is such that $\mu_1\overline{\mu}=\mu$ for  $\overline{\mu}\in E^*$ (we recall that  $a\otimes 1\in X_{\LL}$, for  $a\in \D_E$, is given by \eqref{tensor one mappings2}). We claim that $\phi(s_\mu s_\mu^*)=K$. Indeed, for any $a,b \in \D_E$  we have
\begin{align}\label{K(a o b)}
K(a\otimes b)= K(a\otimes 1) b= \lambda_{\mu_1}^{-1}   \left(s_{\mu}s_{\mu }^*\otimes \LL (s_{\mu}s_{\mu }^* a )\right)b=\left(s_{\mu}s_{\mu }^*\otimes s_{\overline{\mu}}s_{\mu }^* a s_{\mu_1}\right)b.
\end{align}
Moreover, for  any $x,y \in \D_E$ we have
\begin{align*}
\langle s_{\mu }s_{\mu }^*\otimes s_{\overline{\mu}}s_{\mu }^* a s_{\mu_1}, x\otimes y\rangle_{\LL} & =  s_{\overline{\mu}}s_{\mu }^* a s_{\mu_1} \LL(s_{\mu }s_{\mu }^* x) y
\\
&=\lambda_{\mu_1}  s_{\overline{\mu}}s_{\mu }^* a s_{\mu }s_{\mu }^*  x s_{\mu_1} y
\\
&=\LL(s_{\mu }s_{\mu }^* a  x ) y
\\
&=\langle s_{\mu }s_{\mu }^* (a \otimes 1), x\otimes y\rangle_{\LL}.
\end{align*}
Hence $s_{\mu }s_{\mu }^*\otimes s_{\overline{\mu}}s_{\mu }^* a s_{\mu_1}= s_{\mu }s_{\mu }^* (a \otimes 1)$. Thus in view of \eqref{K(a o b)} we get
$
K(a\otimes b)= (s_{\mu }s_{\mu }^* a \otimes 1)b= \phi(s_{\mu }s_{\mu}^*) (a \otimes b),
$
which proves our claim.
If $v\in E^0\setminus  E^0_{inf}$,  then using what we have just shown we get
$$
\phi(p_v)=\phi(\sum_{f\in r^{-1}(v)} s_f s_f^*) =\sum_{f\in r^{-1}(v)}   \lambda_{f}^{-1} \Theta_{(s_{f} s_{f}^*\otimes 1),(s_{f} s_{f}^*\otimes 1)} \in \K(X_\LL).
$$
This shows that $\clsp\{ s_\mu s_\mu^*:  \mu\in E^*\setminus  E^0_{inf}\}
\subseteq J(X_\LL)$. Suppose, on the contrary, that this inclusion is proper. Then there exists an element in $J(X_\LL)$ of the form $a=\sum_{\mu \in E^*\setminus  E^0_{inf}} c_\mu s_\mu s_\mu^* + \sum_{v \in E^0_{inf}} c_v p_v$ where $c_\mu$, $c_v$ are complex numbers and there is $v_0 \in E^0_{inf}$ such that $c_{v_0}\neq 0$. Then $\sum_{v \in E^0_{inf}} c_v p_v=a-\sum_{\mu \in E^*\setminus  E^0_{inf}} c_\mu s_\mu s_\mu^* $ is in $J(X_\LL)$. Hence $p_{v_0}=c_{v_0}^{-1}p_{v_0} \sum_{v \in E^0_{inf}} c_v p_v$ is in $J(X_\LL)$.  We show  that the latter is impossible. Indeed, any operator in $\K(X_\LL)$ can be approximated by   $K\in \K(X_\LL)$  given by a finite linear combination of the form
$$
K=\sum_{\mu,\nu,\eta, \tau \in F}  \lambda_{\mu,\nu,\eta, \tau} \Theta_{(s_\mu s_\mu^*\otimes s_\nu s_\nu^*), (s_\eta s_\eta^*\otimes s_\tau s_\tau^*)}
$$
where $F\subseteq  E^*$ is a finite set. For any such combination we can find an edge $g\in r^{-1}(v_0)$ such that the projection $p_{s(g)}$ is orthogonal to every projection $s_\mu s_\mu^*$, $\mu \in F$. Then for $\tau \in F$, and any $\eta \in E^*$, we have
$$
\langle s_\eta s_\eta^*\otimes s_\tau s_\tau^*, s_g s_g^* \otimes 1\rangle_{\LL} = s_\tau s_\tau^*\LL(s_\eta s_\eta^* s_gs_g^*)= s_\tau s_\tau^* p_{s(g)} (\lambda_g s_g^* s_\eta s_\eta^* s_g) =0.
$$
This implies that $K (s_g s_g^* \otimes 1)=0$. Thus, as $\lambda_g^{-1}\|s_g s_g^* \otimes 1\|=1$, we get
\begin{align*}
\|\phi(p_{v_0})- K\|\geq \lambda_{g}^{-1}\|\phi(p_v) (s_g s_g^* \otimes 1) - K(s_g s_g^* \otimes 1)\|=\lambda_{g}^{-1}\|s_g s_g^*\otimes 1\|=1.
\end{align*}
Accordingly,  $\phi(p_{v_0})\notin \K(X_\LL)$ which is a contradiction. Thus  $\clsp\{ s_\mu s_\mu^*:  \mu\in E^*\setminus  E^0_{inf}\}=J(X_\LL)$.
\end{proof}
By Proposition \ref{non-commutative Markov shifts equivalences}, if $E$ has no infinite emitters then \eqref{Markov shift for local finite} defines an endomorphism $\alpha:\D_E\to \D_E$.
\begin{lem}\label{existence of endomorphism}
Suppose $E$ has no infinite emitters and $\alpha$ is given by \eqref{Markov shift for local finite}. Then
$$
\alpha(\D_E)\D_E =N_\LL^\bot \cap  J(X_\LL).
$$
\end{lem}
\begin{proof}
Let $=\mu_1\overline{\mu}\in E^*\setminus E^0$ where $\mu_1\in E^1$. Since
$$
s_\mu s_\mu^*=s_{\mu_1\overline{\mu}}s_{\mu_1\overline{\mu}}^*=s_{\mu_1} s_{\mu_1}^* \alpha(s_{\overline{\mu}}s_{\overline{\mu}}^*)\in  \alpha(\D_E)\D_E,
$$
it follows from Lemma \ref{description of the relevant ideals} that   $N_\LL^\bot \cap  J(X_\LL)\subseteq  \alpha(\D_E)\D_E$. For the reverse inclusion it suffices to show that for any $a\in \D_E$  we have $\alpha(a) \in N_\LL^\bot \cap  J(X_\LL)$. To this end, consider a net $\mu_F:=\sum_{e \in F} s_e s_e^*\in N_\LL^\bot \cap  J(X_\LL)=\clsp\{ s_\mu s_\mu^*:  \mu\in E^*\setminus  E^0\}$  indexed  by finite sets $F\subseteq E^1$ ordered  by inclusion. Clearly, $\mu_F \alpha(a)$ converges to $\alpha(a)$. Hence $\alpha(a)\in N_\LL^\bot \cap  J(X_\LL)$.
\end{proof}

\begin{Proof of}{Theorem \ref{thm for Raeburn, Brownlowe and Vitadello}:}
 By    Proposition \ref{proposition for lambda transfer operators} the sum \eqref{first approximation of a transfer operator} converges in norm and the operator  $u_\lambda=\sum\limits_{e\in E^1} \sqrt{\lambda_e}s_e$ converges  strictly  in  $M(C^*(E))$. Plainly, $\LL(a)=u_\lambda au_\lambda^*$ for $a \in \D_E$.  Let us treat $M(C^*(E))$ as a non-degenerate subalgebra of $\B(H)$. Then the pair $(id, u_\lambda)$ is  a faithful representation of $(\D_E,\LL)$ in $B(H)$.
We claim that it is covariant, in the sense of Definition \ref{definition of covariance}, i.e.  $
N_\LL^\bot\cap J(X_\LL) \subseteq     \overline{\D_E u_\lambda \D_E u^*_\lambda}.
$ Indeed,  taking $s_\mu s_\mu^*$ where $\mu\in E^*\setminus  E^0$, and writing $\mu=\mu_1\overline{\mu}$ where $\mu_1\in E^1$ and  $\overline{\mu}\in E^*$ we get
$$
s_\mu s_\mu^*=s_{\mu_1\overline{\mu}}s_{\mu_1\overline{\mu}}^*=\lambda_{\mu_1} ^{-1} s_{\mu_1} s_{\mu_1}^*  \, u_\lambda (s_{\overline{\mu}}s_{\overline{\mu}}^*) u_\lambda^* \, s_{\mu_1} s_{\mu_1}^* \in \overline{\D_E u_\lambda \D_E u^*_\lambda}.
$$
By virtue of Lemma  \ref{description of the relevant ideals} this proves our claim. Hence by Proposition \ref{universal description of crossed products}
the mapping $j_A(a) \mapsto a$,  $j_A(a)s \mapsto a u_\lambda$, $a\in \D_E$, gives rise to a homomorphism  from $C^*(\D_E,\LL)$ into $C^*(E)$.
 Let us denote it by $id\rtimes u_\lambda$ and note that  it is actually an epimorphism because  we have
$$
s_e= (\sqrt{\lambda_{e}})^{-1} (s_es_e^*) u_\lambda  p_{s(e)},\qquad \textrm{for all }e\in E^1.
$$
 Moreover, for the canonical gauge circle action $\gamma$ on $C^*(E)$ we have
$$
\gamma_z(a)=a, \qquad \gamma_z(au_\lambda )=zau_\lambda , \qquad \textrm{ for all }a\in \D_E, \, z\in \T.
$$
Thus applying Proposition \ref{Gauge invariant uniqueness theorem} we see that $id\rtimes u_\lambda$ is  an isomorphism.  This proves the main part of the assertion.

i). Suppose now that $E$ has no infinite emitters. Then \eqref{Markov shift for local finite} converges in norm by Proposition \ref{non-commutative Markov shifts equivalences}. Since
$$
\LL(\alpha(a)b)= \sum_{e,f\in E^1} \lambda_f s_f^* s_e as_e^* b s_f= \sum_{e\in E^1} \lambda_e p_{s(e)} a s_e^* b s_e =a \sum_{e\in E^1} \lambda_e  s_e^* b s_e=a\LL(b),
$$
for all $a$, $b\in \D_E$, the triple $(\D_E,\alpha,\LL)$ is an Exel system.
Similar calculations show that
 $$
 \alpha(\LL(\alpha(a)))= \sum_{e\in E^1}  \left(\sum_{f\in s^{-1}(s(e))}  \lambda_f\right)   s_e  a s_e^*.
 $$
Hence, in view of Proposition \ref{Ex2.3}, $\LL$ is a regular transfer operator for $\alpha$ if and only if  \eqref{probability distribution} holds.
 The crossed products $C^*(\D_E, \LL)$ and $\D_E\rtimes_{\alpha,\LL} \N$ coincide by Lemma \ref{existence of endomorphism}  and Theorem \ref{Exel crossed products as relative Pimsner}.

ii). Suppose $\alpha$ is an endomorphism such that $(\D_E,\alpha, \LL)$ is an Exel system. Putting $b=s_e s_e^*$, $e\in E^1$, in the equation $\LL(\alpha(a)b)=a\LL(b)$ we get $s_e^*\alpha(a)s_e=a s_e^* s_e$. This in turn implies that
$$
\alpha(a)s_es_e^*=s_ea s_e^*, \qquad e \in E^1.
$$
Lack of infinite receivers in $E$ implies that the projections $s_es_e^*$ sum up strictly to a projection in $M(C^*(E))$. Let us denote it  by $p$. It follows  that $\alpha(a)p=\sum_{e\in E^1} s_e a s_e^*$ is in $\D_E$ for any $a\in \D_E$. If there would be  an infinite emitter  $v\in E^0$, then $\alpha(p_v)p=\sum_{e\in s^{-1}(v)}s_e s_e^*$ would not be an element of $\D_E$ (otherwise it would correspond via the isomorphism $\D_E\cong C_0(\partial E)$ to a characteristic function of a non-compact set). Thus $E$ must be locally finite.  Furthermore, in view of Lemma  \ref{description of the relevant ideals}, we have $p\D_E=N_\LL^\bot$. Therefore if $\alpha(\D_E)\D_E\subseteq N_\LL^\bot$ then $\alpha$ has to be given by  \eqref{Markov shift for local finite}.

Item iii) follows  from item ii) because for a locally finite graph without sources we have $N_\LL^\bot=\D_E$ by Lemma \ref{description of the relevant ideals}.
\end{Proof of}


\begin{thebibliography}{99}

\bibitem{alnr} \textsc{S. Adji}, \textsc{M. Laca}, \textsc{M. Nilsen} and   \textsc{I. Raeburn}. Crossed products by semigroups of endomorphisms and the Toeplitz algebras of ordered groups. \emph{Proc. Amer. Math. Soc.} \textbf{122} (1994), 1133--1141.

 \bibitem{Ant-Bakht-Leb}
\textsc{A. B. Antonevich, V. I. Bakhtin} and  \textsc{A. V. Lebedev}. Crossed product of a $C^*$-algebra by an endomorphism,
  coefficient algebras and transfer operators. \emph{Math. Sbor.} (9) \textbf{202} (2011),   1253--1283.


   \bibitem{hecke5} \textsc{J. Arledge}, \textsc{M. Laca} and  \textsc{I. Raeburn}. Semigroup crossed products and Hecke algebras arising from number fields, \emph{Doc. Math.} \textbf{2} (1997), 115--138.



\bibitem{blackadar} \textsc{B. Blackadar}. Shape theory for $C^*$-algebras. \emph{Math. Scand.} \textbf{56} (1985), 249--275.


\bibitem{brenken} \textsc{B. Brenken}.  $C^*$-algebras associated with topological relations. \emph{J. Ramanujan Math. Soc.} \textbf{19.1} (2004), 35--55.





\bibitem{BMS} \textsc{L.G. Brown, J. Mingo} and \textsc{N. Shen}.  Quasi-multipliers and embeddings of Hilbert $C^*$-modules. \emph{Canad. J. Math.}
  \textbf{71 }(1994), 1150--1174.
	


 \bibitem{Brownlowe} \textsc{N. Brownlowe}. Realising the $C^*$-algebra of a higher-rank graph as an Exel's crossed product. \emph{J. Operator Theory} (1) \textbf{68} (2012),  101--130.

\bibitem{BLS} \textsc{N. Brownlowe}, \textsc{N. Larsen} and \textsc{S. Stammeier}. $C^*$-algebras of algebraic dynamical systems and right LCM semigroups. \emph{arXiv:1503.01599}.


\bibitem{br} \textsc{N. Brownlowe} and  \textsc{I. Raeburn}. Exel's crossed product and relative
Cuntz-Pimsner algebras.  \emph{Math. Proc.  Camb. Phil. Soc.}  \textbf{141} (2006), 497--508.

  \bibitem{brv} \textsc{N. Brownlowe, I. Raebrun}  and \textsc{S. T. Vittadello}.  Exel's crossed product for non-unital $C^*$-algebras. \emph{Math. Proc.  Camb. Phil. Soc.} \textbf{149} (2010), 423--444.


\bibitem{cuntz} \textsc{J. Cuntz}. Simple $C^*$-algebras generated by isometries. \emph{Commun. Math. Phys.} \textbf{57} (1977),  173--185.





 \bibitem{exel2} \textsc{R. Exel}.  A new look at the crossed-product of a $C^*$-algebra by an endomorphism. \emph{Ergodic Theory Dyn. Syst.}   \textbf{23} (2003),  1733--1750.

\bibitem{exel-inter} \textsc{R. Exel}. Interactions. \emph{J. Funct. Analysis}, \textbf{244} (2007), 26--62.

\bibitem{er} \textsc{R. Exel} and \textsc{D. Royer}. The crossed product by a partial endomorphism. \emph{Bull. Braz. Math. Soc.} \textbf{38} (2007), 219--261.

 \bibitem{exel_vershik}  \textsc{R. Exel} and \textsc{A. Vershik}. $C^*$-algebras of irreversible dynamical systems.
\emph{Canadian J. Math.} {\bf 58} (2006), 39--63.


\bibitem{fmy} \textsc{C. Farthing}.  \textsc{P. Muhly} and \textsc{T. Yeend}. Higher-rank graph $C^*$-algebras: an inverse semigroup approach.
\emph{Semigroup Forum.} \textbf{71} (2005), 159--187.

\bibitem{fowler} \textsc{N.J. Fowler}.  Discrete product systems of Hilbert bimodules, \emph{Pacific J. Math.} \textbf{ 204} (2002), 335–-375.

\bibitem{hr} \textsc{A. an Huef} and \textsc{I. Raeburn}.  Stacey crossed products associated to Exel systems. \emph{Integral Equations Operator Theory}  \textbf{72} (2012), 537--561.


\bibitem{imv}  \textsc{M. Ionescu, P. Muhly} and \textsc{V. Vega}. Markov operators and $C^*$-algebras.
\emph{Houston J. Math.} (3) \textbf{38} (2012),  775--798.


\bibitem{jp} \textsc{Ja. A. Jeong} and  \textsc{Gi Hyun Park}. Topological entropy for the canonical completely positive maps on graphs $C^*$-algebras. \emph{Bull. Austral. Math. Soc.} \textbf{70} (2004), 101--116.


\bibitem{kpw} \textsc{T. Kajiwara, C. Pinzari}   and \textsc{Y. Watatani}. Ideal structure and simplicity of the $C^*$-algebras generated by Hilbert bimodules. \emph{J. Funct. Anal.} \textbf{159 } (1998), 295--322.

\bibitem{katsura1}
\textsc{T. Katsura}. A construction of $C^*$-algebras from $C^*$-correspondences.  Contemp. Math. vol. 335, pp. 173-182, Amer. Math. Soc., Providence (2003)

\bibitem{katsura} \textsc{T. Katsura}. On $C^*$-algebras associated with $C^*$-correspondences. \emph{J. Funct. Anal.} (2) \textbf{217} (2004), 366--401.

\bibitem{ka1} \textsc{T. Katsura}. A class of $C^{*}$-algebras generalizing both graph algebras and homeomorphism
    $C^{*}$-algebras I, fundamental results.  \emph{Trans. Amer. Math. Soc.} \textbf{356} (2004), 4287--4322.

\bibitem{kakariadis} \textsc{E. T. A. Kakariadis} and  \textsc{J. R. Peters}. Representations of C*-dynamical systems implemented by Cuntz families.
 \emph{M{\"u}nster J. Math.} \textbf{6} (2013), 383--411.





\bibitem{kwa-trans} \textsc{B. K. Kwa\'sniewski}.  On transfer operators for $C^*$-dynamical systems. \emph{Rocky J. Math.} (3) \textbf{42} (2012), 919--938.

\bibitem{kwa-doplicher} \textsc{B. K. Kwa\'sniewski}. $C^*$-algebras generalizing both relative Cuntz-Pimsner and Doplicher-Roberts algebras.  \emph{Trans. Amer. Math. Soc.} {\bf 365} (2013), 1809--1873.


\bibitem{kwa-interact} \textsc{B. K. Kwa{\'s}niewski}. Crossed products  by interactions and graph algebras. \emph{Integral Equation Operator Theory} \textbf{80}(3) (2014),  415-451.

\bibitem{kwa-rever} \textsc{B. K. Kwa\'sniewski}. Ideal structure of crossed products by endomorphisms via reversible extensions of $C^*$-dynamical systems. \emph{Internat. J. Math.} \textbf{26} (2015),  1550022 (45 pages).




\bibitem{KL} \textsc{B. K. Kwa\'{s}niewski} and  \textsc{A. V. Lebedev}. Crossed products by endomorphisms and reduction of relations in relative Cuntz-Pimsner algebras. \emph{J. Funct. Anal.} \textbf{264} (2013),  1806--1847.



\bibitem{Lan} \textsc{E. C. Lance}. \emph{Hilbert $C^*$-modules: A
    toolkit for operator algebraists}. London Math. Soc. Lecture Note Series.
    vol. 210. (Cambridge Univ. Press, 1994).
		
\bibitem{diri} \textsc{M. Laca}. Semigroups of *-endomorphisms, Dirichlet series
and phase transitions,  \emph{J. Funct. Anal.} \textbf{152} (1998), 330--378.

\bibitem{bc-alg}  \textsc{M. Laca} and   \textsc{I. Raeburn}. A semigroup crossed product
arising in number theory, \emph{J. London Math. Soc.} \textbf{59} (1999), 330--344.



\bibitem{larsen1} \textsc{N. S. Larsen}. Non-unital semigroup crossed products, \emph{Math. Proc. Roy. Irish. Acad.} \textbf{100A}(2) (2000), 205--218.

\bibitem{Larsen} \textsc{N. Larsen}. Crossed products by abelian semigroups via transfer operators. \emph{Ergodic Theory Dynam. Systems}   \textbf{30} (2010), 1147--1164.

  \bibitem{loring} \textsc{T. A. Loring}.
$C^*$-algebra relations. \emph{ Math. Scand.} \textbf{107} (2010), 43--72.

	\bibitem{Lin-Rae} \textsc{J. Lindiarni} and \textsc{I. Raeburn}. Partial-isometric crossed products
by semigroups of endomorphisms. \emph{J.  Operator Theory}, \textbf{52} (2004), 61-87.


 \bibitem{ms} \textsc{P. S. Muhly} and \textsc{B. Solel}. Tensor algebras over $C^*$-correspondences (representations, dilations, and $C^*$-envelopes). \emph{J. Funct. Anal.}  \textbf{158} (1998),  389--457.


\bibitem{ms2}  \textsc{P. S. Muhly} and \textsc{B. Solel}. On the Morita equivalence of Tensor algebras. \emph{Proc. Lond. Math. Soc.} \textbf{81} (2000),  113--168.

\bibitem{mt} \textsc{P. S. Muhly} and \textsc{M. Tomforde}. Topological quivers.  \emph{Internat. J. Math.}   \textbf{16} (2005), 693--756.




\bibitem{Murphy}   \textsc{G. J. Murphy}.    Crossed products of $C^*$-algebras by endomorphisms.   \emph{Integral Equations Operator Theory} \textbf{24} (1996),
298--319.


\bibitem{Paulsen} \textsc{V. Paulsen}. \emph{Completely Bounded Maps and Operator Algebras}.  Cambridge
Studies in Advanced Mathematics (Cambridge University Press, Cambridge, 2002)

\bibitem{Paschke} \textsc{W. L. Paschke}. Inner product modules over B*-algebras. \emph{Trans. Amer. Math. Soc.} \textbf{182} (1973), 443--464.

\bibitem{Paschke2} \textsc{W. L. Paschke}. The crossed product of a $C^*$-algebra by an endomorphism. \emph{Proc. Amer. math. Soc.} \textbf{80} (1980), 113--118.


\bibitem{Pelczynski} \textsc{A. Pe\l czy\'nski}. Linear extensions, linear averagings, and their applications to linear topological classification of spaces of continuous functions, \emph{Dissertationes Math. (Rozprawy Mat.)} \textbf{58} (1968), 1-92.


 \bibitem{p} \textsc{M.V. Pimsner}.  A class of $C^*$-algebras generalizing both Cuntz-Krieger algebras and crossed products by $\Z$. Fields Inst. Comm. {\bf 12},   (1997) 189--212.



\bibitem{Raeburn}  \textsc{I. Raeburn}. \emph{Graph Algebras}.
CBMS Regional Conference Series in Math., vol. 103 (Amer. Math. Soc., Providence,  2005).


\bibitem{Rordam}  \textsc{M. R\o rdam}.  Classification of certain infinite simple $C^*$-algebras. \emph{J. Funct. Anal.} \textbf{131} (1995), 415--458.


\bibitem{Schweizer} \textsc{J. Schweizer}. Crossed products by $C^*$-correspondences and Cuntz–Pimsner algebras.
 $C^*$-Algebras: Proceedings of the SFB-Workshop on $C^*$-Algebras, Münster, 1999 (Springer-Verlag, Berlin, 2000), pages 203--226.


\bibitem{Szwajcar} \textsc{J. Schweizer}. Dilations of $C^*$-correspondences and the simplicity of Cuntz-Pimsner algebras. \textsc{J. Funct. Anal.} \textbf{180} (2001), 404--425.



\bibitem{Stacey}  \textsc{P. J. Stacey}. Crossed products of $C^*$-algebras by
 $^*$-endomorphisms. \emph{J.\ Austral.\ Math.\ Soc. Ser.~A} \textbf{54} (1993), 204--212.

\bibitem{Tak} \textsc{M. Takesaki}. \emph{Theory of operator algebras I}. EMS on Operator Algebras and Non-Commutative Geometry
    vol. V (Springer-Verlag, 2002).

\bibitem{Webster}
\textsc{S. B. G. Webster.} The path space of a directed graph. \emph{Proc. Amer. Math. Soc.}, \textbf{142}(1) (2014), 213--225.

\end{thebibliography}
\end{document}